\documentclass[a4paper]{article}

\usepackage[cyr]{aeguill}
\usepackage[english]{babel}
\usepackage[utf8x]{inputenc}
\usepackage[T1]{fontenc}

\usepackage[a4paper,top=3cm,bottom=2cm,left=3cm,right=3cm,marginparwidth=1.75cm]{geometry}

\usepackage{amsmath}
\usepackage{graphicx}
\usepackage[colorinlistoftodos]{todonotes}
\usepackage[colorlinks=true, allcolors=blue]{hyperref}
\usepackage{amssymb}
\usepackage{amsthm}
\usepackage{bbold}
\usepackage[all]{xy}
\usepackage{dirtytalk}
\usepackage[titletoc,title]{appendix}

\newtheorem{theorem}{Theorem}[section]

\newtheorem{lemma}[theorem]{Lemma}
\newtheorem{proposition}[theorem]{Proposition}

\theoremstyle{definition}
\newtheorem{definition}{Definition}[section]
\newtheorem{example}{Example}[section]

\newtheorem{remark}[example]{Remark}
\newtheorem{notations}[example]{Notations}

\newcommand\scalemath[2]{\scalebox{#1}{\mbox{\ensuremath{\displaystyle #2}}}}

\title{Julia sets of hyperbolic rational maps have positive Fourier dimension}
\author{Gaétan Leclerc}
\date{ }

\begin{document}

\maketitle

\begin{abstract}

Let $f:\widehat{\mathbb{C}}\rightarrow \widehat{\mathbb{C}}$ be a hyperbolic rational map of degree $d \geq 2$, and let $J \subset \mathbb{C}$ be its Julia set. We prove that $J$ always has positive Fourier dimension. The case where $J$ is included in a circle follows from a recent work of Sahlsten and Stevens \cite{SS20}. In the case where $J$ is not included in a circle, we prove that a large family of probability measures supported on $J$ exhibit polynomial Fourier decay: our result applies in particular to the measure of maximal entropy and to the conformal measure. 

\end{abstract}

\section{Introduction}

\subsection{Hausdorff dimension and Fourier transform}

The notion of Hausdorff dimension, first introduced in 1917, revealed itself useful to describe various geometric properties on fractals, and to give insight on underlying dynamical systems. Together with the development of measure theory and modern analysis, surprising links between fractal geometry and Fourier analysis arised. In his PhD thesis in 1935, \cite{Fr35}, Frostman introduced the notion of \say{energy integral} and related it to the Hausdorff dimension of a set. The result goes as follows: see \cite{Ma15} for a modern proof.

\begin{theorem}
Let $E \subset \mathbb{R}^d$ be a compact set. We denote by $\mathcal{P}(E)$ the space of borel probability measures supported in $E$. For $\mu \in \mathcal{P}(E)$, define its Fourier transform by $$ \widehat{\mu}(\xi) := \int_E e^{ - 2 i \pi x \cdot \xi } d\mu(x). $$
Then $$\dim_H(E) = \sup \left\{ \alpha \in [0, d] \ \Big{|} \ \exists \mu \in \mathcal{P}(E), \ \int_{\mathbb{R}^d} |\widehat{\mu}(\xi)|^2 |\xi|^{\alpha-d} d\xi  < \infty \right\} ,$$
where $\dim_H$ denotes the Hausdorff dimension.
\end{theorem}

We can interpret this equality in the following way. If $\alpha< \dim_H(E)$, then there exists a probability measure $\mu$ supported in $E$ such that $\widehat{\mu}(\xi)$ decays at least like $|\xi|^{- \alpha/2}$ \emph{on average}. At this point, it seems natural to ask whether this estimate can be improved to a \emph{pointwise} estimate. To investigate the question, we are lead to the notion of Fourier dimension.

\begin{definition}

Let $E \subset \mathbb{R}^d$ be a compact set. We define its Fourier dimension by
$$ \dim_F(E) := \sup \left\{ \alpha \in [0,d] \ | \ \exists \mu \in \mathcal{P}(E) , \exists C>0 , \forall \xi \in \mathbb{R}^d, \  |\widehat{\mu}(\xi)| \leq C(1+|\xi|)^{-\alpha/2}  \right\} .$$

\end{definition}

It is clear that the Fourier dimension will always be less than the Hausdorff dimension. But the other inequality is not always true: any set $E$ included in an affine subspace of $\mathbb{R}^d$ will always have zero Fourier dimension. One less trivial example is given by the triadic Cantor set in $\mathbb{R}$. Surprisingly, the linear structure of the Cantor set (more precisely, invariance under $\times 3 \mod 1$) is an obstruction to the Fourier decay of any probability measure supported on it.  \\

Very few explicit examples of sets with positive Fourier dimension are known, even though some works of Salem \cite{Sa51}, Kahane \cite{Ka66}, Bluhm \cite{Bl96} et al suggest that the property $\dim_F(E)=\dim_H(E)$ may be, in some sense and in some particular setting, generic. But constructing deterministic examples of such sets is difficult, as the Fourier dimension seems to be sensitive to the fine structure of the fractal. At the end of the 20th century, the only known examples were obtained by specific number-theoretic constructions in dimension 1, see for example \cite{Kau81} and \cite{QR03}. These constructions were generalized in dimension 2 in 2017, see \cite{Ha17}. 

\subsection{Recent development and main results}

Recently numerous advances have been made, involving various point of views and methods to study the Fourier transform of fractal measures. Using transfer operators, Jordan and Sahlsten \cite{JS16} studied invariant measures for the Gauss map. Li \cite{Li17} introduced a method based on renewal theorems for random walks to study stationary measures, that leads to several results in the linear IFS case for self-affine measures, see \cite{LS19} and related work \cite{So19}, \cite{Br19}, \cite{VY20}, \cite{Ra21}. See also \cite{ARW20} for a study of the nonlinear IFS case in dimension one. \\

The method that interest us here was introduced by Bourgain and Dyatlov in 2017 \cite{BD17}: they proved that the limit set of a non-elementary Fuchsian Schottky group, seen as a subset of $\mathbb{R}$, has positive Fourier dimension (notice that such a limit set is always a Cantor set). More specifically, they proved that Patterson-Sullivan measures exhibit polynomial Fourier decay in this setting. The new theoretical tool was the use of a previous theorem of Bourgain known as the \say{sum product phenomenon} \cite{Bo10}. An accessible introduction to these ideas can be found in the expository article of Green  \cite{Gr09}. A concrete example of such \say{sum-product} theorem is the following: 
\begin{theorem}[\cite{BD17}]
For all $\delta > 0$, there exist $\varepsilon_1, \varepsilon_2 > 0$ and $k \in \mathbb{N}$ such that the
following holds. Let $\mu$ be a probability measure on $\left[1/2, 1 \right]$ and let $N$ be a large integer.
Assume that for all $\sigma \in \left[N^{−1}, N^{−\varepsilon_1}\right]$, 
$$\sup_x \mu\left([x − \sigma, x + \sigma]\right) < \sigma^{\delta}. \quad (*) $$
Then for all $\eta \in \mathbb{R}$, $|\eta| \simeq N$ :
$$ \left| \int \exp(2 i \pi  \eta x_1 \dots x_k) d\mu(x_1) \dots d\mu(x_k) \right| \leq N^{- \varepsilon_2} . $$
\end{theorem}

Roughly, the underlying mechanics behind this kind of theorem is the idea that enhancing some \say{multiplicative structure} may spread the phase so that some cancellations happen. This kind of result is true if we suppose that $\mu$ does not concentrate too much the phase: this is the hypothesis $(*)$.
This plays a key role in their paper: to prove that the Fourier transform of a measure enjoys polynomial decay, one may relate it to a sum of exponential (an integral for a discrete measure) on which this sum product phenomenon applies. The difficulty then is to prove that $(*)$ is satisfied. \\

Soonly after, Li, Naud and Pan \cite{LNP19} generalized the result of Bourgain and Dyatlov for limit sets of general Kleinian Schottky groups. The limit set is still a Cantor set, but in a 2-dimensional setting. The proof follows the same idea as in $\cite{BD17}$: to prove that the Fourier transform of a measure decays, one may relate it to a sum of exponential on which a non-concentration property is satisfied. The main difficulty, once again, lies in the proof of this non-concentration property. \\

Recently, Sahlsten and Stevens \cite{SS20} generalized the result of Bourgain and Dyatlov in a broader setting. They showed that for any \say{totally non linear} Cantor set in the real line, with some hyperbolicity conditions on the underlying dynamic, a large class of invariant measures
called equilibrium measures exhibit polynomial Fourier decay. The core of the proof is the same, but three more ingredients are used: the large class of measure is introduced via the \emph{thermodynamical formalism}, a large deviation technique is used (these ideas already appeared in \cite{JS16}), and the non concentration property is obtained via contraction estimates for suitable transfer operators. \\

In this article, we build upon these previous papers to study the case of \emph{Julia sets} of hyperbolic rational maps in the Riemann sphere (see the section 2). This is the first result of this kind for sets that are not Cantor sets. More precisely, our main result is the following.

\begin{theorem}

Let $f:\widehat{\mathbb{C}} \rightarrow \widehat{\mathbb{C}}$ be a hyperbolic rational map of degree $d \geq 2$. Let $J$ denote its Julia set. If $J$ is included in a circle, then $J$ has positive Fourier dimension, seen as a compact subset of $\mathbb{R}$ after conjugation with a Möbius transformation. If $J$ is not included in a circle, then $J$ has positive Fourier dimension, seen as a compact subset of $\mathbb{C}$.

\end{theorem}

In fact, the case where $J$ is included in a circle is already known: Theorem 9.8.1, page 227 and remark page 230 in \cite{Be91} tells us that in this case, $J$ is either a circle, or a Cantor set. In the first case, the Fourier dimension is 1. In the second case, our hyperbolicity assumption, and Theorem 3.9 in \cite{OW17} ensure that the \say{total non-linearity} condition is satisfied, allowing us to apply the work of Sahlsten and Stevens. 
\newpage
In the case where $J$ is not included in a circle, we prove the following result.

\begin{theorem}

Let $f:\widehat{\mathbb{C}} \rightarrow \widehat{\mathbb{C}}$ be a hyperbolic rational map of degree $d \geq 2$. Denote by $J \subset \mathbb{C}$ its Julia set, and suppose that $J$ is not included in a circle. Let $V$ be an open neighborhood of $J$, and consider any potential $\varphi \in C^1(V,\mathbb{R})$. Let $\mu_\varphi \in \mathcal{P}(J)$ be its associated equilibrium measure. Then:
$$ \exists \varepsilon>0, \ \exists C>0, \ \forall \xi \in \mathbb{C}, \ |\widehat{\mu}_\varphi(\xi)| \leq C (1+|\xi|)^{-\varepsilon} .$$

\end{theorem}

In particular, our result applies to the \textbf{conformal measure} (also called the measure of maximal dimension) and to the \textbf{measure of maximal entropy}, see the section 2. Since the measure of maximum entropy is related to the harmonic measure in a polynomial setting \cite{MR92}, one may expect our result to have some corollaries on the Dirichlet problem with boundary conditions on quasicircles or to the Brownian motion (which is related to the heat equation). Finally, one should stress out that the conclusion of Theorem 1.4 no longer applies if $J$ is a whole circle: in this case, $f$ is conjugated to $z \mapsto z^d$ (\cite{OW17}), and so any invariant probability measure which enjoys Fourier decay must be the Lebesgue measure on the circle (this is an easy exercise using Fourier series).

\subsection{Strategy of the proof}

We follow the ideas in $\cite{SS20}$ and adapt them to our case, where topological difficulties arise from the 2-dimensional setting. The strategy of the proof and organization of the paper goes as follows. 

\begin{itemize}
    \item In section 2, we collect facts about thermodynamic formalism in the context of hyperbolic complex dynamics. The section 2.3 is devoted to the construction of \emph{two} families of open sets adapted to the dynamics. In the section 3.5 we state a large deviation result about Birkhoff sums.
    \item In the section 3 we use the large deviations to derive order of magnitude for some dynamically-related quantities.
    \item The proof of Theorem 1.4 begins in the section 4. Using the invariance of the equilibrium measure by a transfer operator, we relate its Fourier transform to a sum of exponentials by carefully linearizing the phase. We then use a generalized version of Theorem 1.2.
    \item The section 5 is devoted to a proof of the non-concentration hypothesis that is needed. To this end, we use a generalization of Theorem 2.5 in \cite{OW17}, which is a uniform contraction property of twisted transfer operators. 
\end{itemize}

Even if the strategy of the proof is borrowed from \cite{SS20}, they are some noticeable difficulties that arise in our setting that were previously invisible. In dimension 1, estimates of various diameters and linearization processes are made easier by the fact that connected sets are convex. In particular, in dimension 1, the dynamics map convex sets into convex sets.

In dimension 2, one may not associate to the Markov partition a family of open sets that are convex and still satisfy the properties that we usually ask for them: see the remark after Proposition 2.2. We overcome this difficulty by constructing \emph{two} families of open sets associated to the dynamics: the usual open sets related to Markov pieces, in which the theory of \cite{Ru78} and \cite{OW17} applies, and a new one where computations and control are made easier. 
The second difficulty is that the dynamics may twist and deform even the sets in our second family. We overcome this difficulty by taking advantage of the conformality of the dynamics, through the use of the Koebe 1/4-theorem which allows us to have a good control over such deformations. \\

Another difficulty comes in the proof of the non concentration hypothesis: the complex nature of the dynamics suggests non concentration in modulus and arguments of some dynamically related quantities. Arguments being defined modulo $2 \pi$ induces technicalities that are invisible in $\cite{SS20}$.

\subsection{Acknowledgments}

The author would like to thank his PhD supervisor, Frederic Naud, for  numerous helpful conversations and for pointing out various helping references. The author would also like to thank Jialun Li for introducing him to the method of Dolgopyat. The author would like to thank the referee for useful remarks concerning the historical overview and for spotting a few missing arguments in the text. This work is part of the author's PhD and is funded by the Ecole Normale Superieure de Rennes. The author have no relevant financial or non-financial interests to disclose. 

\newpage

\section{Thermodynamic formalism on hyperbolic Julia sets}

\subsection{Hyperbolic Julia sets}

We recall standard definitions and results about holomorphic dynamical systems. For more background, we recommend the notes of Milnor \cite{Mi90}. \\

Denote by $\widehat{\mathbb{C}}$ the Riemann sphere. Let $f: \widehat{\mathbb{C}} \rightarrow \widehat{\mathbb{C}}$ be a rational map of degree $d \geq 2$. \\
Recall that a familly of holomorphic maps defined on an open set $D \subset \widehat{\mathbb{C}}$ is called \emph{normal} if from every sequence of maps from the family there exists a subsequence that converges locally uniformly.
The Fatou set of $f$ is the largest open set in $\widehat{\mathbb{C}}$ where the family of iterates $ \left\{ f^n , \ n \in \mathbb{N}\right\}$ is a normal family. Its complement is called the Julia set and is denoted by $J$. In our case, it is always nonempty and compact. (Lemma 3.5 in \cite{Mi90}) \\

Since $f(J) = f^{-1}(J) = J$, the couple $(f,J)$ is a well defined dynamical system, describing a chaotic behavior. For example, the action of $f$ on $J$ is topologically mixing: for any open set $U$ such that $U \cap J \neq \emptyset$, there exists $n \geq 0$ such that $f^n(U \cap J) = J$. (See Corollary 11.2 in \cite{Mi90}) \\

A case where the dynamics of $f$ on $J$ is particularly well understood is when $f$ is supposed to be \textbf{hyperbolic}, and we will assume it from now on. It means that the orbit of every critical point converges to an attracting periodic orbit. (In other words, if $p \in \widehat{\mathbb{C}}$ is a critical point for $f$, then there exists $p_0 \in \widehat{\mathbb{C}}$ and $m>0$ such that $p_0$ is an attracting fixed point for $f^m$ and $f^{km}(p) \underset{k \rightarrow \infty}{\longrightarrow} p_0$.) In this case $J \neq \widehat{\mathbb{C}}$, and so by conjugating $f$ with an element of $PSL(2,\mathbb{C})$ we can always see $J$ as a compact subset of $\mathbb{C}$. The hyperbolicity condition is equivalent to the existence of constants $c_0$ and $1<\kappa<\kappa_1$, and of a small open neighborhood $V$ of $J$ such that:
$$ \forall x \in V, \ \forall n \geq 0 \ , c_0 \kappa^n \leq |(f^n)'(x)| \leq \kappa_1^n  .$$

This is Theorem 14.1 in \cite{Mi90}. From now on, we also assume that $J$ is not contained in a circle.

\subsection{Pressure and equilibrium states}

\begin{definition}[\cite{Ru89}, \cite{OW17}, \cite{Ru78}, \cite{PU17}] Let $\varphi \in C^1(V,\mathbb{R})$ be a potential.
Denote by $\mathcal{M}_f$ the compact space of all $f$-invariant probability measures on $J$, equipped with the weak*-topology. Denote, for $\mu \in \mathcal{M}_f$, $h_f(\mu)$ the entropy of $\mu$.
Then, the map $$ \mu \in \mathcal{M}_f \longmapsto h_f(\mu) + \int_J \varphi d \mu $$
is upper semi-continuous, and admits a unique maximum, denoted by $P(\varphi)$. \\
The unique measure $\mu_\varphi$ that satisfies
$$ P(\varphi) = h_f(\mu_\varphi) + \int_J \varphi d \mu_\varphi $$
is called the equilibrium state associated to the potential $\varphi$. \\
This measure is ergodic on $(J,f)$ and its support is $J$.

\end{definition}

Two potentials are of particular interest. Define the \textbf{distortion function} by $\tau(x) := \log( |f'(x)| )$ on $V$. If $V$ is chosen small enough, $\tau$ is real analytic, in particular it is $C^1$. We then know that there exists a unique $\delta_J \in \mathbb{R}$ such that $P(- \delta_J \tau)=0$.
In this case, $\delta_J$ is the Hausdorff dimension of $J$, and the equilibrium state $\mu_{- \delta_J \tau}$ is equivalent to the $\delta_J$-dimensional Hausdorff measure on $J$. It is sometimes called the \textbf{conformal measure}, or the \textbf{measure of maximal dimension}. Moreover, we have the formula $$ \dim_H(J) = h_f(\mu_{-\delta_J \tau})/ \int \tau d\mu_{- \delta_J \tau} .$$ See \cite{PU09}, corollary 8.1.7 and Theorem 8.1.4 for a proof.  \\

Another important example is the following. If we set $\varphi=0$, then the pressure is given by the largest entropy available for invariant measures $\mu$. The associated equilibrium state is then called the \textbf{measure of maximal entropy}. In the context where $f$ is a polynomial, this measure coincides with the \textbf{harmonic measure} with respect to $\infty$, see \cite{MR92}. 

\newpage

\subsection{Markov partitions}

Hyperbolic rational maps are especially easy to study thanks to the existence of \textbf{Markov partitions} of the Julia set. Proposition 2.1 and some of the following results are extracted from the subsection 2 of $\cite{Ru89}$ and $\cite{OW17}$.

\begin{proposition}[Markov partitions]

For any $\alpha_0 > 0$ we may write $J$ as a finite union $J = \cup_{a \in \mathcal{A}} P_a$ of compact nonempty sets $P_a$ with $\text{diam} P_a < \alpha_0$, and $|\mathcal{A}| \geq d$. Furthermore, with the topology of $J$,
\begin{itemize}
    \item $\overline{\text{int}_J P_a} = P_a $,
    \item $\text{int}_J P_a \cap \text{int}_J P_b = \emptyset$   if $a \neq b$,
    \item each $f(P_a)$ is a union of sets $P_b$.
\end{itemize}

\end{proposition}

Define $M_{ab}=1$ if $f(P_a) \supset P_b$ and $M_{ab}$ = 0 otherwise.
Then some power $M^N$ of the $|\mathcal{A}| \times |\mathcal{A}|$ matrix $(M_{ab})$ has all its entries positive. 

\begin{remark}
Julia sets are always singleton or perfect sets (Corollary 3.10 in \cite{Mi90}), and in our case, since any point in $J$ always has exactly $d \geq 2$ preimages in $J$, $J$ is always a perfect set. In particular, the condition $\text{int}_J P_a \neq \emptyset$ implies  $\text{diam} (P_a) > 0$ for all $a$. This condition of having $ \geq 2$ preimages for any point in $J$ is what makes the proof of Proposition 2.6 and 2.7 works.
\end{remark}

If $\alpha_0$ is chosen small enough, we may also consider open neighborhoods around the $(P_a)$ that behave well with the dynamics. They will help us do computations with our smooth map $f$.

\begin{proposition}

If $\alpha_0$ is small enough, we can choose a Markov partition $(P_a)_{a \in \mathcal{A}}$ and two families of open sets $(U_a)_{a \in \mathcal{A}}$ and $(D_a)_{a \in \mathcal{A}}$ such that $ P_a \subset U_a \subset \overline{U_a} \subset D_a \subset \overline{D_a} \subset V $, and:

\begin{enumerate}

    \item $\text{diam}(D_a) \leq \alpha_0$
    \item $f$ is injective on $\overline{D_a}$, for all $a \in \mathcal{A}$
    \item $f$ is injective on $D_a \cup D_b$ whenever $D_a \cap D_b \neq \emptyset$

    \item For every $a,b$ such that $f(P_a) \supset P_b$, we have a local inverse $g_{ab} : \overline{D_b} \rightarrow \overline{D_a}$ for $f$. $g_{ab}$ is holomorphic on a neighborhood of $\overline{D_b}$.
    
    \item If $f(P_a) \supset P_b$ for some $a,b \in \mathcal{A}$, then $f(U_a) \supset \overline{U_b}$ and $f(D_a) \supset \overline{D_b}$.
    
    \item $D_a$ is convex. 
    \item For any $a \in \mathcal{A}$, $P_a \nsubseteq \overline{\cup_{b \neq a} U_b} $

\end{enumerate}

\end{proposition}

Usually, only the sets $(U_a)_{a \in \mathcal{A}}$ are considered when dealing with hyperbolic conformal dynamics: they are the open sets introduced in $\cite{OW17}$ and $\cite{Ru89}$, and so are the sets where their papers apply. But they are sometimes not easy to work with, especially because they may not be connected. The $(D_a)_{a \in \mathcal{A}}$ have the advantage to be convex, which will make the computations of section 3 doable. But they have the disadvantage that some $P_b$ may be entirely contained in $D_a$ even if $b \neq a$. 

\begin{proof}
The construction of the sets $(U_a)$ is borrowed from $\cite{Ru89}$, where Ruelle does it for expanding maps. The main problem is that $f$ is not necessarily expanding here, so we will have to introduce a modified metric (sometimes called the Mather metric).  Since $f$ is hyperbolic, we know that there exists some $N \in \mathbb{N}$ such that $f^N$ satisfy $|(f^N)'(x)|>1$ on $J$. So, there exists some small neighborhood $V$ of $J$ where $f^{N} : V \rightarrow \mathbb{C}$ is well defined, and where $ |(f^N)'(x)| \geq \kappa > 1 $. \\

Define $\rho(z) := \sum_{k=0}^{N-1} \kappa^{-k/N} |(f^k)'(z)| $.
Since $f'$ doesn't vanish, it is a smooth and positive function on $V$, and so $ds:= \rho(z) |dz|$ is a well defined conformal metric on $V$. 
Moreover,
$$ \rho(f(z)) |f'(z)| = \sum_{k=0}^{N-1} \kappa^{-k/N} |(f^k)'(f(z))| \ |f'(z)| = \sum_{k=1}^{N} \kappa^{-(k-1)/N} |(f^k)'(z)|  $$
$$ \geq \kappa^{1/N} \sum_{k=0}^{N-1} \kappa^{-k/N} |(f^k)'(z)| = \kappa^{1/N} \rho(z) ,$$
and so $f$ is expanding for the distance $d_\rho$ induced by the conformal metric $\rho(z)|dz|$. In particular, reducing $V$ if necessary, there exists $\alpha > 0$ such that
$$ \forall x,y \in V, d_\rho(x,y) \leq \alpha \Rightarrow d_\rho(f(x),f(y)) \geq \kappa^{1/N} d_\rho(x,y) .$$

In addition, the euclidean distance and the constructed conformal metric are equivalent: there exists a constant $G \geq 1$ such that $$ G^{-1} |x-y| \leq d_\rho(x,y) \leq G|x-y| ,$$ and so the property may become, taking $\alpha$ smaller if necessary:

$$ \forall x,y \in V, |x-y| \leq \alpha \Rightarrow d_\rho(f(x),f(y)) \geq \kappa^{1/N} d_\rho(x,y) .$$

Finally, define $L>1$ a Lipschitz constant of $f$ for the distance $d_\rho$. Now, let $r< G^{-1} \alpha/4$, \ $\alpha_0 < r  \min(G^{-1}/20,L^{-1}(\kappa^{1/N}-\kappa^{1/(2N)})) $, and let $(P_a)_{a \in \mathcal{A}}$ be a Markov partition such that $\text{diam}(P_a) < \alpha_0$.
Define, for $a \in \mathcal{A}$, $$D_a := \text{Conv}\left(D_\rho( x_a , r ) \right) \supset P_a$$ for some fixed $x_a \in \text{int}_J P_a$, where $\text{Conv}$ denotes the euclidean convex hull, and where $D_\rho$ is an open ball for the distance $d_\rho$. The point (1) is satisfied taking $\alpha_0$ smaller if necessary, and the points (2) and (3) follows immediately, provided $\alpha_0$ is small enough, since $f$ is hyperbolic (hence a local biholomorphism). The point (6) follows by the definition of $D_a$. We prove point (5) for $(D_a)_a$.
Since $f$ is a biholomorphism from $D_a$ to $f(D_a)$, and since $\text{diam}_\rho( f(D_a) ) \leq L \alpha_0$, we get that
$$ f(D_a) \supset f\left(D_\rho(x_a, r)\right) \supset D_\rho(f(x_a), \kappa^{1/N} r) \supset  \overline{ D_\rho(x_b, \kappa^{1/(2N)} r) }  $$
for each $b \in \mathcal{A}$ such that $f(P_a) \supset P_b$.\\

We then have to show that $ D_\rho(x_b, \kappa^{1/(2N)} r) \supset{D}_b $.
Let $x,y \in D_\rho(x_b, r)$, and for any $\lambda \in [0,1]$, let $z := \lambda x + (1- \lambda)y \in D_b$. By definition, there exists a path $\gamma_x$ from $x_b$ to $x$ with length $< r$, and a path $\gamma_y$ from $x_b$ to $y$ with length $< r$. Set $\gamma := \lambda \gamma_x + (1- \lambda ) \gamma_y$. It is a well defined path in $D_b$ from $x_b$ to $z$. Its length satisfies

$$ \int_0^1 |\gamma'(t)| \rho(\gamma(t)) | dt \leq \left( \int_0^1 |\gamma'(t)| \rho(x_a) dt  \right) \ e^{ r \| \rho'/\rho \|_{\infty,D_b}  } $$
$$ \leq \left( \lambda \int_0^1 |\gamma_x'(t)| \rho(x_a) dt + (1-\lambda) \int_0^1 |\gamma_y'(t)| \rho(x_a) dt \right)  \ e^{ r \| \rho'/\rho \|_{\infty,D_b}  } $$
$$  \leq \left( \lambda \int_0^1 |\gamma_x'(t)| \rho(\gamma_x(t)) dt + (1-\lambda) \int_0^1 |\gamma_y'(t)| \rho(\gamma_y(t)) dt \right)  \ e^{2 r \| \rho'/\rho \|_{\infty,D_b}  }   $$
$$ \leq  r e^{ 2 r \| \rho'/\rho \|_{\infty,D_b}  } < \kappa^{1/(2N)} r $$
as soon as $ r < \frac{\ln(\kappa)}{4N} \| \rho'/\rho \|_\infty^{-1}  $. Hence (5) is true for $(D_a)_a$. We finished constructing our sets $D_a$. Notice that we can choose $r$ arbitrary small, and so the diameters of the $D_a$ can be chosen as small as we want. The point (4) follows by considering the inverse branches of $f$ induced by (5): they are holomorphic and $\kappa^{1/N}$ contracting for $d_\rho$. \\

The construction of the sets $(U_a)_a$ is easier.
First of all, there exists $\beta>0$ such that $D_\rho(x_a,\beta) \cap P_b = \emptyset $ whenever $a \neq b$. Then, since all of the $P_a$ are compactly contained in $D_a$, there exists a parameter $s< \beta/3$ such that for all $x \in P_a$, $\overline{D_\rho(x,s)} \subset D_a $. Define:

$$ U_a := \{ x \in V, \ d_\rho(x,P_a) < s \} \subset D_a .$$

First,  the fact that $s< \beta/3$ ensures that $P_a \nsubseteq \overline{\cup_{b \neq a} U_b} $, since $x_a \notin \overline{\cup_{b \neq a} U_b}$, hence proving (7). We prove (5). Let $b$ be such that $f(P_a) \supset P_b$. We prove that $f(U_a) \supset U_b$. Let $z \in U_b$. By definition, there exists $z_b \in P_b$ such that $d_\rho(z,z_b)<s$. Notice that $z=f(g_{ab}(z))$: to conclude, it suffice to prove that $g_{ab}(z) \in U_a$ (we only know that it is in $D_a$ at this point). Since $f(P_a) \supset P_b$ and since $f:D_a \rightarrow \mathbb{C}$ is injective, we know that $g_{ab}(z_b) \in P_a$. The fact that $ d_\rho(g_{ab}(z),g_{ab}(z_b)) \leq \kappa^{-1/N} s < s $ allows us to conclude.

 \end{proof}

To study the dynamics, we need to introduce some notations. \\
A finite word $(a_n)_n$ with letters in $\mathcal{A}$ is called admissible if $M_{a_n a_{n+1}=1}$ for every $n$. Then define:

\begin{itemize}
    \item $ \mathcal{W}_n := \{ (a_k)_{k=1, \dots, n} \in \mathcal{A}^n , \ (a_k) \ \text{is admissible} \}  $
    \item For $\textbf{a} = a_1 \dots a_n \in \mathcal{W}_n$, define $g_{\mathbf{a}} := g_{a_1 a_2} g_{a_2 a_3} \dots g_{a_{n-1} a_n} : D_{a_n} \rightarrow D_{a_1} $.
    \item For $\textbf{a} = a_1 \dots a_n \in \mathcal{W}_n$, define $P_{\textbf{a}} := g_{\textbf{a}}(P_{a_n}) \subset P_{a_1}$, $U_{\textbf{a}} := g_{\textbf{a}}( U_{a_n}) \subset U_{a_1}$, and \mbox{$D_{\textbf{a}} := g_{\textbf{a}}( D_{a_n}) \subset D_{a_1}$.}
\end{itemize}

We begin by an easy remark on the diameters of the $D_\mathbf{a}$ and of the behavior of $\varphi$ on those sets.

\begin{remark}
Since our potential is smooth, it is Lipshitz on $\overline{D} := \overline{ \bigcup_a D_a } \subset V$. There exists a constant $C_\varphi>0$ such that:
$$ \forall x,y \in D, \ |\varphi(x) - \varphi(y)| \leq C_\varphi |x-y|.$$

Moreover, we have some estimates on the diameters of the $P_{\textbf{a}}$.
Since $f$ is supposed to be hyperbolic, we have the following estimate for the local inverses:

$$ \forall n \geq 1, \ \forall \textbf{a} \in \mathcal{W}_{n+1}, \ \forall x \in D_{\textbf{a}}, \ \kappa_1^{-n} \leq |g_{\textbf{a}}'(x)| \leq c_0^{-1} \kappa^{-n} .$$

Hence, since each $D_a$ are convex, $$ \text{diam}(P_\textbf{a}) \leq \text{diam}(D_{\textbf{a}}) = \text{diam}(g_{\textbf{a}}(D_{a_{n+1}})) \leq c_0^{-1} \kappa^{-n} $$ decreases exponentially fast. We will say that $\varphi$ has \textbf{exponentially decreasing variations}, as
$$ \max_{ \textbf{a} \in \mathcal{W}_n } \sup_{x,y \in D_{\textbf{a}}} |\varphi(x) - \varphi(y)| \leq C_\varphi c_0^{-1} \kappa^{-n} .$$
\end{remark}

Notice the following technical difficulty: for $a \in \mathcal{A}$, $D_a$ may be convex but for a word $\mathbf{a} \in \mathcal{W}_n$, $D_{\mathbf{a}}$ will eventually be twisted by $g_{\mathbf{a}}$ and not be convex anymore. Fortunately, we still have the following result, which relies heavily on the fact that $f$ is holomorphic:

\begin{lemma}

For all $\mathbf{a} \in \mathcal{W}_n, \ \text{Conv}(P_{\mathbf{a}}) \subset D_{\mathbf{a}} $.

\end{lemma}

\begin{proof}

We will need to recall some results on univalent holomorphic functions $g:\mathbb{D} \rightarrow \mathbb{C}$. First of all, we have the Koebe quarter theorem, which states that if $g$ is such a map, then $g(\mathbb{D}) \supset B\left(g(0),\frac{|g'(0)|}{4} \right)$ ($\mathbb{D}$ is the unit disk, and $B(\cdot,\cdot)$ denotes an open euclidean ball). Secondly, the Koebe distortion theorem states that in this case, we also have that $ |g(z)-g(0)| \leq |g'(0)| \frac{|z|}{\left(1-|z|\right)^2} $. Combining those two results gives us the following fact:
$$ (*) \ \text{For any injective holomorphic } g:B(z_0,r) \rightarrow \mathbb{C} \text{, if } |z-z_0| \leq r/10, \text{ then } [g(z_0),g(z)] \subset g(B(z_0,r)). $$
Now recall the construction of $D_a$: we have $P_a \subset D_a$ with $\text{diam}(P_a) \leq \alpha_0 < r G^{-1}/20$, and $D_a = \text{Conv}(D_\rho(x_a,r))$. 
Since the distances are equivalent with associated constant $G$, we can write that, for any $x \in P_a$:
$$ P_a \subset B(x, r G^{-1}/20) \subset B(x, r G^{-1}/2) \subset B(x_a, r G^{-1}) \subset D_a. $$

Hence we can apply the fact $(*)$ to the map ${g_{\mathbf{a}}}_{| B(x, r G^{-1}/2)}$: every $y \in P_a$ is in $ B(x, r G^{-1}/20) $, and so $ [x,y] \subset D_\mathbf{a} $. We have proved that $\text{Conv}(P_\mathbf{a}) \subset D_{\mathbf{a}}$.
\end{proof}

We end this topological part with some final remarks, extracted from $\cite{OW17}$. The following \say{partition result} is true:

    $$ J = \bigcup_{\textbf{a} \in \mathcal{W}_n} P_{\textbf{a}} \text{ , and }\text{int}_J P_{\textbf{a}} \cap \text{int}_J P_{\textbf{b}} = \emptyset \text{ if }\textbf{a} \neq \textbf{b} \in \mathcal{W}_n. $$
    
This allows us to see that $\bigcup_{a \in \mathcal{A}} \partial P_a$ is a closed $f$-invariant subset of $J$. Since its complementary is open and nonempty, and since $\mu_\varphi$ has full support, the ergodicity of $\mu_\varphi$ implies that $\mu_\varphi \left(\bigcup_{a \in \mathcal{A}}  \partial P_a \right) = 0$. \\

In particular, it implies that, for any $n \geq 1$, we have the relation

$$ \forall f \in C^0(J,\mathbb{C}), \ \int_J f d\mu_\varphi = \sum_{\textbf{a} \in \mathcal{W}_n} \int_{P_{\textbf{a}}} f \ d\mu_\varphi ,$$

which will be useful later.

\subsection{Transfer operators}

Let $\varphi \in C^1(V,\mathbb{R})$ be a smooth potential.
Let $U := \bigcup_{a \in \mathcal{A}} U_a$, and notice that $\overline{f^{-1}(U)} \subset U$. \\
We define the associated transfer operator $\mathcal{L}_\varphi : C^1(U) \rightarrow C^1(U) $ by $$ \mathcal{L}_\varphi h(x) := \sum_{y , f(y)=x} e^{\varphi(y)} h(y) .$$
Notice that, if $x \in U_a$, then 
$$ \mathcal{L}_\varphi h(x) = \sum_{b, M_{ba}=1} e^{\varphi( g_{ba}(x) )} h(g_{ba}(x)) .$$
We have the following formula for the iterates:
$$ \mathcal{L}_\varphi^n h(x) = \sum_{f^n(y)=x} e^{S_n \varphi(y)} h(y) ,$$
where $S_n \varphi := \sum_{k=0}^{n-1} \varphi \circ f^k$ is a Birkhoff sum. This can be rewritten, if $x \in U_{b}$, in the following form:
$$ \mathcal{L}_\varphi^n h(x) = \underset{a_{n+1}=b}{\sum_{ \textbf{a} \in \mathcal{W}_{n+1} } } e^{S_n \varphi( g_{\textbf{a}}(x) )} h( g_{\textbf{a}}(x) ) .$$
Finally, note that our transfer operator also acts on the set of probability measures on $J$, by duality, in the following way:
$$ \forall h \in C^0(J,\mathbb{C}), \ \int_J h \ d \mathcal{L}_\varphi^* \nu := \int_J  \mathcal{L}_\varphi h \ d\nu.$$
Transfer operators satisfy the following theorem, extracted from \cite{Ru89}, Theorem 3.6 :
\begin{theorem}[Perron-Frobenius-Ruelle] 

With our choice of open set $U \supset J$, and for any \underline{real} \\ potential $\varphi \in C^1(U,\mathbb{R})$:

\begin{itemize}
    \item the spectral radius of $\mathcal{L}_\varphi$, acting on $C^1(U,\mathbb{C})$, is equal to $e^{P(\varphi)}$.
    \item there exists a unique probability measure $\nu_\varphi$ on $J$ such that $\mathcal{L}_\varphi^* \nu_\varphi = e^{P(\varphi)} \nu_\varphi $.
    \item there exists a unique map $h \in C^1(U,\mathbb{C})$ such that $\mathcal{L}_\varphi h = e^{P(\varphi)} h$ and $\int h d\nu_\varphi =1$. \\ Moreover, $h$ is positive.
    \item The product $h \nu_\varphi$ is equal to the equilibrium measure $\mu_\varphi$.

\end{itemize}

\end{theorem}

The Perron-Frobenius-Ruelle theorem allows us to link $\mu_\varphi$ to $\nu_\varphi$, and this will allow us to prove some useful estimates, called the Gibbs estimates.

\begin{proposition}[Gibbs estimate, \cite{PP90}]

$$\exists C_0 \geq 1, \ \forall \textbf{a} \in \mathcal{W}_n, \ \forall x_\textbf{a} \in P_{\textbf{a}}, \ C_0^{-1} e^{S_n \varphi(x_{\textbf{a}}) - n P(\varphi)} \leq \mu_\varphi( P_{\textbf{a}} ) \leq C_0 e^{S_n\varphi(x_{\textbf{a}}) - n P(\varphi)} .$$

\end{proposition}

\begin{proof}

It is enough to prove the estimate for $\nu_\varphi$ since $h$ is continuous on the compact $J$, and since $h \nu_\varphi = \mu_\varphi$. We have
$$ \int_J e^{- \varphi} \mathbb{1}_{P_{a_1 \dots a_n}} d\nu_\varphi = e^{-P(\varphi)} \int_J \mathcal{L}_\varphi\left( e^{- \varphi} \mathbb{1}_{P_{a_1 \dots a_n}} \right) d\nu_\varphi  = e^{-P(\varphi)} \int_{J} \mathbb{1}_{P_{a_2 \dots a_n}} d\nu_\varphi = e^{-P(\varphi)} \nu_\varphi(P_{a_2 \dots a_n}). $$
Moreover, since $\varphi$ has exponentially decreasing variations, we can write that
$$\forall x_{\textbf{a}} \in P_{a_1 \dots a_n}, \ e^{-\varphi(x_{\textbf{a}}) - C \kappa^{-n}} \nu_\varphi(P_{a_1 \dots a_n}) \leq \int_J e^{- \varphi} \mathbb{1}_{P_{a_1 \dots a_n}} d\nu_\varphi \leq  e^{-\varphi(x_{\textbf{a}}) + C \kappa^{-n}} \nu_\varphi(P_{a_1 \dots a_n})  ,$$
and so $$ \forall x_{\textbf{a}} \in P_{a_1 \dots a_n}, \ e^{-\varphi(x_{\textbf{a}}) - C \kappa^{-n}}  \leq \frac{\nu_\varphi(P_{a_2 \dots a_n})}{\nu_\varphi(P_{a_1 \dots a_n}) } e^{-P(\varphi)} \leq  e^{-\varphi(x_{\textbf{a}}) + C \kappa^{-n}} . $$
Multiplying those inequalities gives us the desired relation, with $C_0 := e^{ C/(\kappa-1) }$. \end{proof}

It will be useful, in our future computations, to get rid of the pressure term in our exponential: in the case where $P(\varphi)=0$, we see that $\mu_\varphi(P_\textbf{a})  \simeq e^{S_n \varphi(x_{\textbf{a}})}$ for $x_{\textbf{a}} \in P_{\textbf{a}}$.  

\begin{proposition}

Let $\psi \in C^1(U)$, and let $\mu_\psi$ be its associated equilibrium state.
There exists $\varphi \in C^1(U)$ such that $\mu_\varphi = \mu_{\psi}$, and that is normalized. \\ That is: $P(\varphi) =0$, $ \varphi < 0 \text{ on } J$, $ \mathcal{L}_\varphi 1 = 1 $ and $ \mathcal{L}_\varphi^* \mu_\varphi = \mu_\varphi $.  
\end{proposition}

\begin{proof}

The Perron-Frobenius-Ruelle theorem tells us that there exists a $C^1$ map $h>0$ and a probability measure $\nu_\psi$ such that $\mathcal{L}_\psi^* \nu_\psi = e^{P(\psi)} \nu_{\psi} $, $\mathcal{L}_\psi h = e^{P(\psi)} h$ and $\int h d\nu_{\psi} = 1$. \\
It is then a simple exercise to check that $\varphi := \psi - \log( h \circ f ) + \log(h) - P(\psi)$ defines a normalized potential, and that its equilibrium measure $\mu_\varphi$ is equal to $\mu_\psi$. \end{proof}

This theorem has the following consequence: we can always suppose that our equilibrium measure comes from a \emph{normalized} potential, by eventually choosing another smaller Markov partition afterwards. It gives us for free the invariance under some transfer operator, which completes the already fine properties of $f$-invariance and ergodicity. It also allows us to prove a useful regularity property.

\begin{proposition}

The equilibrium measure $\mu_\varphi$ is upper regular. More precisely, there exists $C,\delta_{AD}>0$ such that:
$$ \forall x \in \mathbb{C}, \ \forall r>0, \ \mu_\varphi(B(x,r)) \leq C  r^{\delta_{AD}}.$$

\end{proposition}

\begin{proof}

First of all, since $\mu_\varphi$ is a probability measure, we may only prove this estimate for $r$ small enough. Then, we know that $\mu_\varphi$ is supported in $J$, and so we just have to verify the estimate if $B(x_0,r) \cap J \neq \emptyset$. Without loss of generality, we can suppose that $x_0 \in J$. \\

The main idea is to cover $B(x_0,r) \cap J$ by some $P_{\mathbf{b}}$, but estimating the number of such $P_{\mathbf{b}}$ that are needed to do so is difficult. To bypass this difficulty, we use the notion of Moran cover. For any $x \in J$, define $n(x,r)$ as the only integer such that 

$$ |(f^{n(x,r)-1})'(x)|^{-1} \geq r \quad \text{and} \quad |(f^{n(x,r)})'(x)|^{-1} < r.$$

We get from the hyperbolicity condition $|(f^{n})'| \geq c_0 \kappa^n$ the following bound: $$ \forall x, \ -n(x,r) \leq \ln\left( 2 r c_0^{-1} \right)/\ln \kappa. $$

For any $x \in J \setminus \bigcup_{n \geq 0} f^{-n} \left( \bigcup_{a \in \mathcal{A}} \partial P_a \right)$ and for any $n$, there exists a unique $\mathbf{a} \in \mathcal{W}_n$ such that $x \in P_{\mathbf{a}}$. We denote it $P_n(x)$. Notice that $x \in P_{n(x,r)}(x)$. If $y \in P_{n(x,r)}(x)$ and $n(y,r) \leq n(x,r)$ then $ P_{n(x,r)}(x) \subset P_{n(x,r)}(y) $. Let $P(x)$ be the largest cylinder containing $x$ of the form $P_{n(y,r)}(x)$ for some $y \in P(x)$ and satisfying $P_{n(z,r)}(x) \subset P(x)$ for any $z \in P(x)$. The sets $(P(x))_{x \in J}$ are equal or disjoint (mod the boundary), and hence produce a cover of $J$ called a \emph{Moran cover} . Denote this Moran cover $\mathcal{P}_r$. An important property of this cover is the following: there exists a constant $M$ \emph{independent of $x_0$ and $r$} such that we can cover the ball $B(x_0,r)$ by $M$ elements of $\mathcal{P}_r$. Moreover, every element of the Moran cover have diameter strictly less than $r$. See \cite{PW97} page 243, \cite{Pe98} section 20, or \cite{WW17}. The proof uses the conformality of the dynamics. \\

We can then conclude our proof. The following holds:

$$ \mu_\varphi(B(x_0,r)) \leq \underset{B(x_0,r) \cap P \neq \emptyset}{\sum_{P \in \mathcal{P}_r} } \mu_\varphi(P) . $$

By the Gibbs estimate, since each $P \in \mathcal{P}_r$ is of the form $P_{\mathbf{b}}$ for some $\mathbf{b} \in \mathcal{W}_{n(x,r)}$, $x \in J$, and by the bound on $n(x,r)$, we get:
$$ \forall P \in \mathcal{P}_r, \ \mu_\varphi(P) \leq C_0 e^{-n(x,r) |\sup_J \varphi|} \leq C r^{\delta_{AD}}  $$
for some $C, \delta_{AD} >0$.
Hence, since $B(x_0,r) \cap P$ occurs at most $M$ times, we get our desired bound $$ \mu_\varphi(B(x_0,r)) \leq M C r^{\delta_{AD}} .$$ \end{proof}

\newpage

\subsection{Large deviation estimates }

From the ergodicity of $\mu_\varphi$, it is natural to ask if a large deviation theorem holds for the Birkhoff sums of potentials. The following theorem is true, we detail its proof in the annex A.

\begin{theorem}

Let $\mu_\varphi$ be the equilibrium measure associated to a normalized $C^1(U,\mathbb{R})$ potential. Let $\psi \in C^1(U,\mathbb{R})$ be another potential. Then, for all $\varepsilon > 0$, there exists $C,\delta_0 > 0$ such that
$$\forall n \geq 1, \ \mu_\varphi\left( \left\{ x \in J \ , \ \left|\frac{1}{n} S_n \psi(x) - \int_J \psi d\mu_{\varphi} \right| > \varepsilon \right\} \right) \leq C e^{- n \delta_0} .$$

\end{theorem}

\begin{definition}

Let $\varphi \in C^1(U,\mathbb{R})$ be a normalized potential with equilibrium measure $\mu_\varphi$. \\ Let $\tau = \log |f'| \in C^1(U,\mathbb{R})$ be the distortion function. We call 
$$ \lambda_f(\mu_\varphi) := \int_J \tau \ d\mu_{\varphi} $$
the Lyapunov exponent of $\mu_\varphi$, and $ \delta := h_f(\mu_\varphi)/ \lambda(\mu_{\varphi}) $ the dimension of $\mu_\varphi$.

\end{definition}

\begin{remark}

The hyperbolicity and normalization assumptions ensure that $ h_f(\mu_\varphi), \lambda_f(\mu_\varphi) > 0$.
Indeed, we know that $\varphi<0$ on all $J$ and $P(\varphi)=0$, and so
$$ h_f(\mu_\varphi) = P(\varphi) - \int_J \varphi d\mu_\varphi > 0.$$
For the Lyapunov exponent, using the fact that $\mu_\varphi$ is $f$-invariant, we see that
$$ \lambda_f(\mu_\varphi) = \int_J \log |f'| d\mu_\varphi = \frac{1}{n} \sum_{k=0}^{n-1}  \int_J \log |f' \circ f^k| d\mu_\varphi $$ $$= \frac{1}{n} \int_J \log |(f^n)'| d\mu_\varphi \geq \frac{\log(c_0)}{n} + \log(\kappa) \rightarrow \log(\kappa) > 0 .$$

\end{remark}

\begin{proposition}
Let $\varphi \in C^1(U,\mathbb{R})$ be a normalized potential with equilibrium measure $\mu_\varphi$. Denote by $\lambda>0$ its Lyapunov exponent and $\delta>0$ its dimension. Then, for every $\varepsilon > 0$, there exists $C,\delta_0>0$ such that

$$ \forall n \geq 1, \ \mu_\varphi \left( \left\{ x \in J \ , \ \left|\frac{1}{n}S_n \tau(x) -  \lambda \right| \geq \varepsilon \text{ or } \left|\frac{S_n \varphi(x)}{S_n \tau(x)} + \delta \right| \geq \varepsilon \right\} \right) \leq C e^{- \delta_0 n} .$$

\end{proposition}

\begin{proof}

Let $\varepsilon>0$.
Applying Theorem 2.8 to $\psi = \tau$ gives  $$ \mu_\varphi \left( \left\{ x \in J \ , \ \left|\frac{1}{n}S_n \tau(x) - \lambda\right| \geq \varepsilon \right\} \right) \leq C e^{- \delta_0 n} $$
for some $C$ and $\delta_0>0$. Next, if $x \in J$ satisfies $$\left|\frac{S_n \varphi(x)}{S_n \tau(x)} + \delta \right| \geq \varepsilon, $$
then we have $$ |S_n \Phi(x)| \geq \varepsilon |S_n \tau(x)| \geq \varepsilon (n \log \kappa + \log c_0) $$
for the modified potential $\Phi := \varphi + \delta \tau$. Notice that this potential is $C^1$, and that $$ \int_J \Phi \ d\mu_\varphi = \int_J \varphi d\mu_\varphi + \frac{h_f(\mu_\varphi)}{\lambda_f(\mu_\varphi)} \int_J \tau \ d\mu_\varphi = 0 .$$ For $n$ large enough, we get $|S_n \Phi(x)| \geq \varepsilon $, and so we can apply Theorem 2.8 to $\Phi$ again and conclude. \end{proof}

For clarity, we will replace $\mu_\varphi$ by $\mu$ in the rest of the paper. The dependence on $\varphi$ will be implied. 


\section{Computing some orders of magnitude}

In this section, we derive various orders of magnitude of quantities that appear when we iterate our transfer operator. We need to recall some useful formalism used in $\cite{BD17}$. 

\begin{itemize}
    \item For $n \geq 1$, recall that $\mathcal{W}_n$ is the set of admissible words of length $n$. (A word $\mathbf{a}$ is admissible if $f(P_{a_i}) \supset P_{a_{i+1}}$ for all $i$.) If $\textbf{a} = a_1 \dots a_n a_{n+1} \in \mathcal{W}_{n+1}$, define $\mathbf{a}' := a_1 \dots a_{n} \in \mathcal{W}_{n}$. 
    
    \item For $\textbf{a}=a_1 \dots a_{n+1} \in \mathcal{W}_{n+1}$, $\textbf{b} = b_1 \dots b_{m+1} \in \mathcal{W}_{m+1}$, we write $\textbf{a} \rightsquigarrow \textbf{b}$ if $a_{n+1} = b_1$. Note that when $\textbf{a} \rightsquigarrow \textbf{b}$, the concatenation $\textbf{a}'\textbf{b}$ is an admissible word of length $n+m+1$.
    
    \item For $\mathbf{a} \in \mathcal{W}_{n+1}$, define $b(\mathbf{a}) := a_{n+1}$.

\end{itemize}

With those notations, we can reformulate our formula for the iterate of our transfer operator. For a function $h : U \rightarrow \mathbb{C}$, we have:

$$ \forall x \in P_b, \ \mathcal{L}_\varphi^n h(x) = \underset{\mathbf{a} \rightsquigarrow b}{\sum_{\mathbf{a} \in \mathcal{W}_{n+1}}} e^{S_n \varphi(g_{\mathbf{a}}(x))} h(g_{\mathbf{a}}(x)) = \underset{\mathbf{a} \rightsquigarrow b}{\sum_{\mathbf{a} \in \mathcal{W}_{n+1}}} h(g_{\mathbf{a}}(x)) w_{\mathbf{a}}(x) ,$$

where $$ w_{\mathbf{a}}(x) := e^{S_n\varphi(g_{\mathbf{a}}(x))} .$$

Iterating $\mathcal{L}_\varphi^n$ again leads us to the formula
$$ \forall x \in P_b, \ \mathcal{L}_\varphi^{nk} h(x)  = \sum_{\mathbf{a}_1 \rightsquigarrow \dots \rightsquigarrow \mathbf{a_k} \rightsquigarrow b} h(g_{\mathbf{a}_1 ' \dots \mathbf{a}_{k-1} ' \mathbf{a}_k}(x) ) w_{\mathbf{a}_1 ' \dots \mathbf{a}_{k-1} ' \mathbf{a}_k}(x) .$$

We are interested in the behavior of, for example, $w_\mathbf{a}$ for well behaved $\mathbf{a}$. For this, we use the previously mentioned large deviation estimate. \\

This part is adapted from $\cite{SS20}$ and $\cite{JS16}$. Remember that $\overline{f^{-1}(D)} \subset D$.

\begin{definition}
For $\varepsilon>0$ and $n \geq 1$, write
$$ A_{n}(\varepsilon) := \left\{ x \in f^{-n}(D) \ , \ \left|\frac{1}{n}S_n \tau(x) -  \lambda \right| < \varepsilon \text{ and } \left|\frac{S_n \varphi(x)}{S_n \tau(x)} + \delta \right| < \varepsilon  \right\}. $$

Then Proposition 2.9 says that, for all $\varepsilon>0$, there exists $n_0(\varepsilon) \in \mathbb{N}$ and $\delta_0(\varepsilon)>0$ such that $$\forall n \geq n_0(\varepsilon), \ \mu(J \setminus A_n(\varepsilon)) \leq e^{- \delta_0(\varepsilon) n} .$$

\end{definition}

\begin{notations}
To simplify the reading, when two quantities dependent of $n$ satisfy $b_n \leq C a_n $ for some constant $C$, we denote it by $a_n \lesssim b_n$. If $a_n \lesssim b_n \lesssim a_n$, we denote it by $a_n \simeq b_n$. If there exists $c,C$ and $\alpha$, independent of $n$ and $\varepsilon$, such that $ c e^{- \varepsilon \alpha n} a_n \leq b_n \leq C e^{\varepsilon \alpha n} a_n$, we denote it by $a_n \sim b_n$. Throughout the text $\alpha$ will be allowed to change from line to line. It correspond to some positive constant. 
\end{notations}

Eventually, we will chose $\varepsilon$ small enough such that this exponentially growing term gets absorbed by the other leading terms, so we can neglect it.

\begin{proposition}
Let $\mathbf{a} \in \mathcal{W}_{n+1}$ be such that $D_{\mathbf{a}} \subset A_n(\varepsilon)$. Then:

\begin{itemize}
    \item uniformly on $x \in D_{b(\mathbf{a})}, \  |g_\mathbf{a}'(x)| \sim e^{- n \lambda}$
    \item $\mathrm{diam}(P_\mathbf{a}), \ \mathrm{diam}(U_\mathbf{a}), \ \mathrm{diam}(D_\mathbf{a})  \sim e^{- n \lambda} $
    \item uniformly on $x \in D_{\mathbf{a}}, \ w_{\mathbf{a}}(x) \sim e^{- \delta \lambda n} $
    \item $\mu(P_\textbf{a}) \sim e^{- \delta \lambda n}$
    
\end{itemize}

\end{proposition}

\begin{remark}
Intuitively speaking, here is what is happening.
Proposition 2.9 states that, for most $x \in J$, $\frac{1}{n} S_n \tau \cong \lambda$ and $ \frac{1}{n} S_n \varphi \cong - \lambda \delta $. Then, recall that $ \text{diam}( P_{\mathbf{a}} ) \simeq |(f^n)'|^{-1}  = e^{-S_n{ \tau } },$ and so $\text{diam} P_{\mathbf{a}} \sim e^{- \lambda n}$ for most words $\textbf{a}$. Notice that the presence of the Lyapunov exponent in the exponential is not surprising, since it is defined to represent a characteristic frequency of our problem.
Samely, we can argue that since our equilibrium measure satisfies the Gibbs estimate, we have $ \mu( P_{\mathbf{a}} ) \simeq e^{S_n{\varphi}} ,$
and so $ \mu(P_{\mathbf{a}}) \sim e^{- \delta \lambda n} $
for most words $\mathbf{a}$. Again, it is no surprise that this exponent appears here: we recognize that $ \mu(P_{\mathbf{a}}) \sim \text{diam}(P_\mathbf{a})^\delta $, where $\delta$ is the dimension of our measure.

\end{remark}

\begin{proof}

Let $\mathbf{a} \in \mathcal{W}_{n+1}$ be such that $D_{\mathbf{a}} \subset A_n(\varepsilon)$. We have

$$ \forall x \in D_{b(\mathbf{a})}, \ |g_{\mathbf{a}}'(x)| = e^{-S_n \tau (g_{\mathbf{a}}(x))}  ,$$
and so $$ \forall x \in D_{b(\mathbf{a})}, \  e^{-n \lambda } e^{-n \varepsilon} \leq |g_{\mathbf{a}}'(x)| \leq e^{- n \lambda} e^{ n \varepsilon}. $$
For the diameters, the argument uses the conformal setting, through the Koebe quarter theorem. By  lemma 2.3, $\text{Conv}(P_{\mathbf{a}}) \subset D_{\mathbf{a}} \subset A_n(\varepsilon)$. Hence:

$$ \forall x,y \in P_{a_n}, \ |x-y| = |f^n(g_{\mathbf{a}}(x)) - f^n(g_{\mathbf{a}}(y))| $$ 
$$ \leq \int_{0}^1 |(f^n)'\left(g_{\mathbf{a}}(y)+t(g_{\mathbf{a}}(x)-g_{\mathbf{a}}(y))\right)| dt \ |(g_{\mathbf{a}}(x)-g_{\mathbf{a}}(y))| \leq e^{\varepsilon n} e^{n \lambda} \text{diam}(P_{\mathbf{a}}) ,$$
and so $ e^{- \varepsilon n} e^{-\lambda n} \text{diam}(P_{a_n}) \leq \text{diam}(P_{\mathbf{a}}) $. Next, we write
$$ \text{diam}(P_{\mathbf{a}}) \leq \text{diam}(U_{\mathbf{a}}) \leq \text{diam}(D_{\mathbf{a}})$$
and 
$$ \text{diam}(D_{\mathbf{a}}) = \text{diam}(g_{\mathbf{a}}(D_{a_n})) \leq e^{\varepsilon n} e^{-\lambda n} \text{diam}(D_{a_n}) .$$
by convexity of $D_{a_n}$.
We have proved that $\text{diam}(P_\mathbf{a}), \ \text{diam}(U_\mathbf{a}), \ \text{diam}(D_\mathbf{a})  \sim e^{- n \lambda} $. \\
Next, consider the weight $w_{\mathbf{a}}(x)$. We have $$ w_{\mathbf{a}}(x)   = e^{S_n \varphi(g_{\mathbf{a}}(x))},$$
so
$$ e^{- \delta S_n \tau(x)} e^{- \varepsilon |S_n \tau(x)|} \leq w_{\mathbf{a}}(x) \leq e^{- \delta S_n \tau(x)} e^{\varepsilon |S_n \tau(x)|} ,$$ and hence
$$ e^{- \delta \lambda n} e^{- \varepsilon (\lambda + \delta + \varepsilon) n } \leq w_{\mathbf{a}}(x) \leq e^{- \delta \lambda n} e^{\varepsilon (\lambda + \delta + \varepsilon) n } .$$
Finally, since $\mu$ is a Gibbs measure for some constant parameter $C_0$, and with pressure 0, we can write:
$$ C_0^{-1} e^{- \delta \lambda n} e^{- \varepsilon (\lambda + \delta + \varepsilon) n } \leq \mu(P_{\mathbf{a}}) \leq C_0 e^{- \delta \lambda n} e^{\varepsilon (\lambda + \delta + \varepsilon) n }. $$ \end{proof}

\begin{definition}

Define the set of ($\varphi$-)regular words by
$$ \mathcal{R}_{n+1}(\varepsilon) := \left\{ \mathbf{a} \in \mathcal{W}_{n+1} \ | \ D_{\mathbf{a}} \subset A_{n}(\varepsilon) \right\}, $$
and the set of regular k-blocks by $$ \mathcal{R}_{n+1}^k(\varepsilon) = \left\{ \mathbf{A}=\mathbf{a}_1' \dots \mathbf{a}_{k-1}' \mathbf{a}_k \in \mathcal{W}_{nk+1} \ | \ \forall i, \ \mathbf{a}_i \in \mathcal{R}_{n+1}(\varepsilon) \right\} .$$ 
Finally, define the associated geometric points to be $$ R_{n+1}^k(\varepsilon) := \bigcup_{\mathbf{A} \in \mathcal{R}_{n+1}^k(\varepsilon) } P_{\mathbf{A}} .$$

\end{definition}

\begin{lemma}
 There exists $n_1(\varepsilon)$ such that, for all $n \geq n_1(\varepsilon)$, we have:
$$ J \cap  A_{n}(\varepsilon/2) \subset R_{n+1}(\varepsilon).$$

\end{lemma}

\begin{proof}

Let $x \in A_{n}(\varepsilon/2)$. There exists $\mathbf{a} \in \mathcal{W}_{n+1}$ such that $x \in D_{\mathbf{a}}$. To conclude, it suffices to show that $D_{\mathbf{a}} \subset A_n(\varepsilon)$. So let $y \in D_{\mathbf{a}}$.
We already saw in remark 2.1 that Lipschitz potentials have exponentially decreasing variations. It implies in particular the existence of some constant $C>0$, which depends only on $f$ here, such that
$$ \left| S_n \tau(y) - S_n \tau(x) \right| \leq C .$$
Hence, we have
$$ \left| \frac{S_n \tau(y)}{n} - \lambda \right| = \left| \frac{S_n \tau(x)}{n} - \lambda\right| + \frac{1}{n}\left| S_n \tau(y) - S_n \tau(x) \right| \leq \varepsilon/2 + \frac{C}{n} \leq \varepsilon $$
as long as we chose $n$ large enough, depending on $\varepsilon$. Samely, we can write
$$ \left| \frac{S_n \varphi(y)}{S_n \tau(y)} + \delta \right| \leq \varepsilon/2 + \left| \frac{S_n \varphi(y)}{S_n \tau(y)} - \frac{S_n \varphi(x)}{S_n\tau(x)} \right| .$$
Since $S_n \tau = \log |(f^n)'| \geq \log(c_0) + n \log(\kappa)$ uniformly on $D$ for some $\kappa > 1$, we get 
$$ \left| \frac{S_n \varphi(y)}{S_n \tau(y)} + \delta \right| \leq \varepsilon/2 + \frac{1}{(\log(c_0) + n \log(\kappa))^2}\left| S_n \varphi(y) S_n \tau(x) - S_n \varphi(x) S_n \tau(y) \right| $$
$$ \leq \varepsilon/2 + \frac{|S_n \tau(x)| }{(\log(c_0) + n \log(\kappa))^2}\left| S_n \varphi(y)  - S_n \varphi(x) \right| + \frac{ |S_n \varphi(x)|}{(\log(c_0) + n \log(\kappa))^2}\left|  S_n \tau(x) - S_n \tau(y) \right| $$
$$ \leq \varepsilon/2 + \frac{C}{n} $$

for some constant $C$, where we used the fact that $(S_n \varphi)/n$ is uniformly bounded on $f^{-n}(\overline{D})$ and the preceding remark on potentials with exponentially vanishing variations.
Again, choosing $n$ large enough depending on $\varepsilon$ allows us to conclude. \end{proof}

\begin{proposition}

We have the following cardinality estimate:
$$ \# \mathcal{R}_{n+1}(\varepsilon) \sim e^{\delta \lambda n} .$$

Moreover, there exists $n_2(\varepsilon)$ and $\delta_1(\varepsilon)>0$ such that
$$ \forall n \geq n_2(\varepsilon), \ \mu\left( J \setminus R_{n+1}(\varepsilon) \right) \leq  e^{ - \delta_1(\varepsilon) n} .$$

\end{proposition}

\begin{proof}

By the preceding lemma, we can write, for $n \geq n_1(\varepsilon)$:
$$ J \cap A_{n}(\varepsilon/2) \subset R_{n+1}(\varepsilon).$$
Moreover, we also know that there exists $n_0(\varepsilon/2)$ and $\delta_0(\varepsilon/2)$ such that, for all $n \geq n_0(\varepsilon/2)$, we have
$$ \mu\left( J \setminus A_n(\varepsilon/2) \right) \leq e^{- \delta_0(\varepsilon/2) n} .$$

So define $n_2(\varepsilon) := \max(n_1(\varepsilon), n_0(\varepsilon/2)/\varepsilon_0 ).$ For all $n \geq n_2(\varepsilon)$, we then have:

$$ \mu(J \setminus R_{n+1}(\varepsilon)) \leq \mu\left(J \setminus A_n(\varepsilon/2) \right) \leq e^{- \delta_0(\varepsilon/2) n } .$$

Next, the cardinality estimates follow from the bound on the measure. Indeed, we know that, for $n \geq n_2(\varepsilon)$:
$$ 1 = \mu(R_{n+1}(\varepsilon)) + \mu(J \setminus R_{n+1}(\varepsilon)) \leq \sum_{\mathbf{a} \in \mathcal{R}_{n+1}(\varepsilon)} \mu(P_\mathbf{a})  + e^{- \delta_0(\varepsilon/2) n} ,$$
and so
$$ 1 - e^{- \delta_0(\varepsilon/2) n} \leq \sum_{\mathbf{a} \in \mathcal{R}_{n+1}(\varepsilon)} \mu(P_\mathbf{a}) \leq 1 .$$

We then use the estimate obtained for $\mu(P_\mathbf{a})$, that is,

$$ C_0^{-1} e^{- \delta \lambda n} e^{- \varepsilon (\lambda + \delta + \varepsilon) n } \leq \mu(P_{\mathbf{a}}) \leq C_0 e^{- \delta \lambda n} e^{\varepsilon (\lambda + \delta + \varepsilon) n } ,$$

and we obtain
$$ C_0^{-1} e^{\delta \lambda n} e^{- \varepsilon(\lambda + \delta + \varepsilon)n} \left(  1 - e^{- \delta_0(\varepsilon/2) n} \right) \leq \# \mathcal{R}_{n+1}(\varepsilon) \leq C_0 e^{\delta \lambda n} e^{\varepsilon(\lambda + \delta + \varepsilon)}, $$

which proves that $\# \mathcal{R}_{n+1}(\varepsilon) \sim e^{\delta \lambda n}$.

\end{proof}

\begin{proposition}

For all $n \geq n_2(\varepsilon)$,  $$ \mu\left(J \setminus R_{n+1}^k(\varepsilon) \right) \leq k e^{ - \delta_1(\varepsilon) n} ,$$
and so $$ \# \mathcal{R}_{n+1}^k(\varepsilon) \sim e^{k \delta \lambda n} .$$

\end{proposition}

\begin{proof}

Define $\tilde{R}_{n+1}(\varepsilon) := \bigsqcup_{\mathbf{a} \in \mathcal{R}_{n+1}} \text{int}_J P_{\mathbf{a}} $. From the point of view of the measure $\mu$, it is indistinguishable from $R_{n+1}(\varepsilon)$. First, we prove that
$$ \bigcap_{i=0}^{k-1} f^{-n i} \left( \tilde{R}_{n+1}(\varepsilon) \right) \subset R_{n+1}^k(\varepsilon) .$$
Let $x \in \bigcap_{i=0}^{k-1} f^{-n i} \left( \tilde{R}_{n+1}(\varepsilon) \right) $. Since there exists $\mathbf{A}=\mathbf{a}_1' \dots \mathbf{a}_{k-1}' \mathbf{a} _k\in \mathcal{W}_{kn+1}$ such that $x \in P_{\mathbf{A}}$, we see that for any $i$ we can write $f^{ni}(x) \in P_{\mathbf{a}_{1+i}' \dots \mathbf{a}_k} \cap \tilde{R}_{n+1}(\varepsilon) $.
So there exists $\mathbf{b}_{i+1} \in \mathcal{R}_{n+1}(\varepsilon)$ such that $ P_{\mathbf{a}_{1+i}' \dots \mathbf{a}_k} \cap \text{int}_J P_{\mathbf{b}_{i+1}} \neq \emptyset$. Then $\mathbf{b}_{i+1} = \mathbf{a}_{i+1}$, for all $i$, which implies that $\mathbf{A} \in \mathcal{R}_{n+1}^k(\varepsilon)$. \\
Now that the inclusion is proved, we see that $$ \mu\left( J \setminus R_{n+1}^k \right) \leq \sum_{i=0}^{k-1} \mu\left( f^{-n i} \left( J \setminus \tilde{R}_{n+1} \right) \right) $$
$$ = k \mu( J \setminus R_{n+1} ), $$
and we conclude by the previous theorem. The cardinal estimate is done as before. \end{proof}


\section{Reduction to sums of exponentials}

We can finally begin the proof of the main Theorem 1.4. Recall that $f$ is a hyperbolic rational map of degree $d \geq 2$, and that $J \subset \mathbb{C}$ denotes its Julia set, which is supposed \textbf{not to be included in a circle}. Fix a small Markov partition $(P_a)_{a \in \mathcal{A}}$ and open sets $(U_a)_{a \in \mathcal{A}}$ and $(D_a)_{a \in \mathcal{A}}$ as in Proposition 2.2. Finally, fix a normalized $\varphi \in C^1(V,\mathbb{R})$, and denote by $\mu$ its associated equilibrium state.  \\

We wish to prove that $\widehat{\mu}$ exhibits some polynomial decay. For this, recall that
$$ \widehat{\mu}(\xi) = \int_J e^{-2i \pi x \cdot \xi} d\mu(x) \ ,$$
where $x$ and $\xi$ are seen in $\mathbb{R}^2$, and where $\cdot$ is the usual inner product. We will use the invariance of $\mu$ by the transfer operator. Since it involves the inverse branches $g_{\mathbf{a}}$, we will rewrite this integral in a more complex fashioned way. We can write:

$$ \widehat{\mu}(\xi) = \int_J e^{-2 i \pi \text{Re}(x \overline{\xi})} d\mu(x) \ ,$$

where this time, $x$ and $\xi$ are seen as complex numbers.
As we will be interested in intertwining blocks of words, we need a new set of notations, inspired from the one used in \cite{BD17}. For a fixed $n$ and $k$, denote:

\begin{itemize}
    \item $\textbf{A}=(\textbf{a}_0, \dots, \textbf{a}_k) \in \mathcal{W}_{n+1}^{k+1} \ , \ \textbf{B}=(\textbf{b}_1, \dots, \textbf{b}_k) \in \mathcal{W}_{n+1}^{k} $.
    \item We write $\textbf{A} \leftrightarrow \textbf{B}$ iff $\textbf{a}_{j-1} \rightsquigarrow \textbf{b}_j \rightsquigarrow \textbf{a}_j$ for all $j=1,\dots k$.
    \item If $\textbf{A} \leftrightarrow \textbf{B}$, then we define the words $\textbf{A} * \textbf{B} := \textbf{a}_0' \textbf{b}_1' \textbf{a}_1' \textbf{b}_2' \dots \textbf{a}_{k-1}' \textbf{b}_k' \textbf{a}_k$ and $\textbf{A} \# \textbf{B} :=  \textbf{a}_0' \textbf{b}_1' \textbf{a}_1' \textbf{b}_2' \dots \textbf{a}_{k-1}' \textbf{b}_k$.
    \item Denote by $b(\textbf{A}) \in \mathcal{A}$ the last letter of $\textbf{a}_k$.
    
\end{itemize}

Then, we can write:

$$ \forall x \in P_{b}, \ \mathcal{L}_\varphi^{(2k+1)n} h(x) = \underset{\mathbf{A} \rightsquigarrow b }{\sum_{\mathbf{A} \leftrightarrow \mathbf{B}}} h(g_{\mathbf{A} * \mathbf{B}}(x)) w_{\mathbf{A} * \mathbf{B}}(x) .$$
In particular, the invariance of $\mu$ under $\mathcal{L}_\varphi$ allows us to write the following formula:

$$ \widehat{\mu}(\xi) = \sum_{\mathbf{A} \leftrightarrow \mathbf{B}} \int_{P_{b(\mathbf{A})}} e^{-2 i \pi \text{Re}\left( \overline{\xi} g_{\mathbf{A} * \mathbf{B}}(x) \right)} w_{\mathbf{A} * \mathbf{B}}(x) d\mu(x) .$$

In this section, our goal is to relate this quantity to a well behaved sum of exponentials. To this end, we will need to introduce various parameters that will be chosen in section 5. Before going on, let us explain the role of those different quantities. 

Five quantities will be at play: $\xi,n,k,\varepsilon_0$ and $\varepsilon$. The parameters $k,\varepsilon_0$ and $\varepsilon$ must be thought as being \emph{fixed}. $k$ will be chosen by an application of Theorem 5.3. $\varepsilon_0$ will be chosen at the end of the proof of Proposition 6.5. $\varepsilon_0$ will be chosen small compared to $\lambda$, and $\varepsilon$ will be chosen small compared to $\varepsilon_0$, $\lambda$, $\delta$ and every other constant that might appear in the proof. \\

The only variables are $\xi$ and $n$, but they are related. We think of $\xi$ as a large enough variable, $n$ will be depending on $\xi$ with a relation of the form $ n \simeq \ln \xi $. \\

We prove the following reduction.

\begin{proposition}

Define $$J_n := \{ e^{\varepsilon_0 n/2} \leq |\eta| \leq  e^{2 \varepsilon_0 n} \}$$ and $$  \zeta_{j,\mathbf{A}}(\mathbf{b}) := e^{2 \lambda n} g_{\mathbf{a}_{j-1}' \mathbf{b}}'(x_{\mathbf{a}_j}) $$
for some choice of $x_\mathbf{a} \in \text{int}_J P_{\mathbf{a}}$ for any finite admissible words $\mathbf{a}$.
There exists a constant $\alpha>0$ such that, for $ |\xi| \simeq e^{(2k+1) \lambda n} e^{\varepsilon_0 n} $ and $n$ large enough depending on $\varepsilon$:
$$ e^{- \varepsilon  \alpha n} |\widehat{\mu}(\xi)|^2 \lesssim e^{-\lambda \delta (2k+1) n} \sum_{\mathbf{A} \in \mathcal{R}_{n+1}^{k+1}} \sup_{\eta \in J_n} \Bigg{|} \underset{\mathbf{A} \leftrightarrow \mathbf{B}}{\sum_{\mathbf{B} \in \mathcal{R}_{n+1}^k}} e^{2 i \pi \text{Re}\left( \eta \zeta_{1,\mathbf{A}}(\mathbf{b}_1) \dots \zeta_{k,\mathbf{A}}(\mathbf{b}_k) \right)} \Bigg| $$
$$ \quad \quad  \quad \quad  \quad \quad  \quad \quad +  e^{- \varepsilon  \alpha n} \mu(J \setminus R_{n+1}^{2k+1}(\varepsilon) )^2 + \kappa^{-2n} + e^{-(\lambda-\varepsilon_0) n } + e^{- \varepsilon_0 \delta_{AD}n/2} .$$
\end{proposition}

Once Proposition 4.1 is established, if we manage to prove that the sum of exponentials enjoys exponential decay in $n$, then choosing $\varepsilon$ small enough will allow us to see that $|\widehat{\mu}(\xi)|^2$ enjoys polynomial decay in $\xi$, and Theorem 1.4 will be proved. We prove Proposition 4.1 through a succession of lemmas.

\begin{lemma}

$$ |\widehat{\mu}(\xi)|^2 \lesssim \Bigg{|} \underset{\mathbf{B} \in \mathcal{R}_{n+1}^k}{\underset{\mathbf{A} \in \mathcal{R}_{n+1}^{k+1} }{\sum_{\mathbf{A} \leftrightarrow \mathbf{B}}}} \int_{P_{b(\mathbf{A})}} e^{-2 i \pi \text{Re}\left( \overline{\xi} g_{\mathbf{A} * \mathbf{B}}(x) \right)} w_{\mathbf{A} * \mathbf{B}}(x) d\mu(x) \Bigg{|}^2 + \mu(J \setminus R_{n+1}^{2k+1}(\varepsilon) )^2 .$$

\end{lemma}

\begin{proof}

We have
$$ \widehat{\mu}(\xi) = \sum_{\mathbf{A} \leftrightarrow \mathbf{B}} \int_{P_{b(\mathbf{A})}} e^{-2 i \pi \text{Re}\left( \overline{\xi} g_{\mathbf{A} * \mathbf{B}}(x) \right)} w_{\mathbf{A} * \mathbf{B}}(x) d\mu(x) .$$
We are only interested on blocks $\mathbf{A}$ and $\mathbf{B}$ that allow us to get some control on the different quantities that will appear: those are the regular words. We have:
$$ \widehat{\mu}(\xi) = \underset{\mathbf{B} \in \mathcal{R}_{n+1}^k}{\underset{\mathbf{A} \in \mathcal{R}_{n+1}^{k+1} }{\sum_{\mathbf{A} \leftrightarrow \mathbf{B}}}} \int_{P_{b(\mathbf{A})}} e^{-2 i \pi \text{Re}\left( \overline{\xi} g_{\mathbf{A} * \mathbf{B}}(x) \right)} w_{\mathbf{A} * \mathbf{B}}(x) d\mu(x) +  \underset{\text{or} \ \mathbf{B} \notin \mathcal{R}_{n+1}^k}{\underset{\mathbf{A} \notin \mathcal{R}_{n+1}^{k+1} }{\sum_{\mathbf{A} \leftrightarrow \mathbf{B}}}} \int_{P_{b(\mathbf{A})}} e^{-2 i \pi \text{Re}\left( \overline{\xi} g_{\mathbf{A} * \mathbf{B}}(x) \right)} w_{\mathbf{A} * \mathbf{B}}(x) d\mu(x) $$

where we see blocks in $\mathcal{R}_{n+1}^k$ as blocks in $\mathcal{W}_{n+1}^k$ in the obvious way.
We can bound the contribution of the non-regular part by

$$ \Bigg{|} \underset{\text{or} \ \mathbf{B} \notin \mathcal{R}_{n+1}^k}{\underset{\mathbf{A} \notin \mathcal{R}_{n+1}^{k+1} }{\sum_{\mathbf{A} \leftrightarrow \mathbf{B}}}} \int_{P_{b(\mathbf{A})}} e^{-2 i \pi \text{Re}\left( \overline{\xi} g_{\mathbf{A} * \mathbf{B}}(x) \right)} w_{\mathbf{A} * \mathbf{B}}(x) d\mu(x) \Bigg{|} \leq \sum_{\mathbf{C} \notin \mathcal{R}_{n+1}^{2k+1}} \int_{P_{b(\mathbf{C})}} w_{\mathbf{C}} d\mu $$
$$ \lesssim \sum_{\mathbf{C} \notin \mathcal{R}_{n+1}^{2k+1}} \mu(P_{\mathbf{C}} ) \leq \mu\left( J \setminus R_{n+1}^{2k+1}(\varepsilon) \right), $$
where we used the fact that $\mu$ is a Gibbs measure.
Once $\varepsilon$ will be fixed, this term will enjoy exponential decay in $n$, thanks to Proposition 3.4. 
\end{proof}

\begin{lemma}

There exists some constant $\alpha>0$ such that, for $n \geq n_2(\varepsilon)$ :

$$ e^{-\varepsilon \alpha n}  \Bigg{|} \underset{\mathbf{B} \in \mathcal{R}_{n+1}^k}{\underset{\mathbf{A} \in \mathcal{R}_{n+1}^{k+1} }{\sum_{\mathbf{A} \leftrightarrow \mathbf{B}}}} \int_{P_{b(\mathbf{A})}} e^{-2 i \pi \text{Re}\left( \overline{\xi} g_{\mathbf{A} * \mathbf{B}}(x) \right)} w_{\mathbf{A} * \mathbf{B}}(x) d\mu(x) \Bigg{|}^2 \quad \quad \quad \quad \quad \quad \quad \quad \quad \quad \quad \quad $$ $$ \quad \quad \quad \quad  \lesssim e^{\lambda \delta (2k-1) n} \underset{\mathbf{B} \in \mathcal{R}_{n+1}^k}{\underset{\mathbf{A} \in \mathcal{R}_{n+1}^{k+1} }{\sum_{\mathbf{A} \leftrightarrow \mathbf{B}}}} \left| \int_{P_{b(\mathbf{A})}} e^{-2 i \pi \text{Re}\left( \overline{\xi} g_{\mathbf{A} * \mathbf{B}}(x) \right)}   w_{\mathbf{a}_k}(x) d\mu(x) \right|^2 + \kappa^{-2n}. $$

\end{lemma}

\begin{proof}

Notice that $w_{\mathbf{A} * \mathbf{B}}(x)$ and $w_{\mathbf{a}_k}(x)$ are related by
$$ w_{\mathbf{A} * \mathbf{B}}(x) = w_{\mathbf{A} \# \mathbf{B}}(g_{\mathbf{a}_k}(x)) w_{\mathbf{a}_k}(x) .$$
For each admissible word $\mathbf{a}$ of any length, fix once and for all a point $x_\mathbf{a} \in \text{int}_J P_{\mathbf{a}}$.
To get the term $w_{\mathbf{A} \# \mathbf{B}}(g_{\mathbf{a}_k}(x))$ out of the integral, we will compare it to $w_{\mathbf{A} \# \mathbf{B}}(x_{\mathbf{a}_k})$.
Recall that $\varphi$ has exponentially decreasing variations: we can write
$$ \max_{ \textbf{a} \in \mathcal{W}_n } \sup_{x,y \in P_{\textbf{a}}} |\varphi(x) - \varphi(y)| \lesssim \kappa^{-n} .$$
So we can write:
$$  \frac{w_{\mathbf{A} \# \mathbf{B}}(g_{\mathbf{a}_k}(x))}{w_{\mathbf{A} \# \mathbf{B}}(x_{\mathbf{a}_k})} = \exp \left( S_{2nk}\varphi(g_{\mathbf{A} \# \mathbf{B}}(g_{\mathbf{a}_k}(x))) - S_{2kn}\varphi( g_{\mathbf{A} \# \mathbf{B}}(x_{\mathbf{a}_k})) \right) ,$$
with 
$$ \left| S_{2nk}\varphi(g_{\mathbf{A} \# \mathbf{B}}(g_{\mathbf{a}_k}(x))) - S_{2kn}\varphi( g_{\mathbf{A} \# \mathbf{B}}(x_{\mathbf{a}_k}))  \right| \lesssim \sum_{j=0}^{2nk-1} \kappa^{- n(2k+1) + j } \lesssim \kappa^{-n} .$$
Hence, there exists some constant $C>0$ such that
$$ e^{-C \kappa^{-n}} w_{\mathbf{A} \# \mathbf{B}}(x_{\mathbf{a}_k}) \leq w_{\mathbf{A} \# \mathbf{B}}(g_{\mathbf{a}_k}(x)) \leq e^{C \kappa^{-n}} w_{\mathbf{A} \# \mathbf{B}}(x_{\mathbf{a}_k}) ,$$
which gives:
$$ \left| w_{\mathbf{A} \# \mathbf{B}}(g_{\mathbf{a}_k}(x)) -  w_{\mathbf{A} \# \mathbf{B}}(x_{\mathbf{a}_k})  \right| \leq \max\left| e^{\pm C \kappa^{-n}} -1 \right|  w_{\mathbf{A} \# \mathbf{B}}(x_{\mathbf{a}_k}) \lesssim \kappa^{-n} w_{\mathbf{A} \# \mathbf{B}}(x_{\mathbf{a}_k}) .$$
From this, we get that
$$ \Bigg{|} \underset{\mathbf{B} \in \mathcal{R}_{n+1}^k}{\underset{\mathbf{A} \in \mathcal{R}_{n+1}^{k+1} }{\sum_{\mathbf{A} \leftrightarrow \mathbf{B}}}} \int_{P_{b(\mathbf{A})}} e^{-2 i \pi \text{Re}\left( \overline{\xi} g_{\mathbf{A} * \mathbf{B}}(x) \right)} \left(w_{\mathbf{A} * \mathbf{B}}(x) - w_{\mathbf{A} \# \mathbf{B}}(x_{\mathbf{a}_k}) w_{\mathbf{a}_k}(x) \right) d\mu(x) \Bigg{|}   $$
$$ \leq \underset{\mathbf{B} \in \mathcal{R}_{n+1}^k}{\underset{\mathbf{A} \in \mathcal{R}_{n+1}^{k+1} }{\sum_{\mathbf{A} \leftrightarrow \mathbf{B}}}} \int_{P_{b(\mathbf{A})}} \left| \left(w_{\mathbf{A} * \mathbf{B}}(x) - w_{\mathbf{A} \# \mathbf{B}}(x_{\mathbf{a}_k}) w_{\mathbf{a}_k}(x) \right) \right| d\mu(x) $$
$$ \lesssim \kappa^{-n} \underset{\mathbf{B} \in \mathcal{R}_{n+1}^k}{\underset{\mathbf{A} \in \mathcal{R}_{n+1}^{k+1} }{\sum_{\mathbf{A} \leftrightarrow \mathbf{B}}}} \int_{P_{b(\mathbf{A})}}  w_{\mathbf{A} \# \mathbf{B}}(x_{\mathbf{a}_k}) w_{\mathbf{a}_k}(x) d\mu(x) \ \lesssim e^{\varepsilon \alpha n} \kappa^{-n} $$
for some positive constant $\alpha$, by Proposition 3.1 and 3.4.
Moreover, by Cauchy-Schwartz,
$$ \Bigg{|} \underset{\mathbf{B} \in \mathcal{R}_{n+1}^k}{\underset{\mathbf{A} \in \mathcal{R}_{n+1}^{k+1} }{\sum_{\mathbf{A} \leftrightarrow \mathbf{B}}}} \int_{P_{b(\mathbf{A})}} e^{-2 i \pi \text{Re}\left( \overline{\xi} g_{\mathbf{A} * \mathbf{B}}(x) \right)}  w_{\mathbf{A} \# \mathbf{B}}(x_{\mathbf{a}_k}) w_{\mathbf{a}_k}(x) d\mu(x) \Bigg{|}^2   $$
$$ = \Bigg{|} \underset{\mathbf{B} \in \mathcal{R}_{n+1}^k}{\underset{\mathbf{A} \in \mathcal{R}_{n+1}^{k+1} }{\sum_{\mathbf{A} \leftrightarrow \mathbf{B}}}} w_{\mathbf{A} \# \mathbf{B}}(x_{\mathbf{a}_k}) \int_{P_{b(\mathbf{A})}} e^{-2 i \pi \text{Re}\left( \overline{\xi} g_{\mathbf{A} * \mathbf{B}}(x) \right)}   w_{\mathbf{a}_k}(x) d\mu(x) \Bigg{|}^2  $$
$$ \leq  \underset{\mathbf{B} \in \mathcal{R}_{n+1}^k}{\underset{\mathbf{A} \in \mathcal{R}_{n+1}^{k+1} }{\sum_{\mathbf{A} \leftrightarrow \mathbf{B}}}} w_{\mathbf{A} \# \mathbf{B}}(x_{\mathbf{a}_k})^2 \underset{\mathbf{B} \in \mathcal{R}_{n+1}^k}{\underset{\mathbf{A} \in \mathcal{R}_{n+1}^{k+1} }{\sum_{\mathbf{A} \leftrightarrow \mathbf{B}}}} \left| \int_{P_{b(\mathbf{A})}} e^{-2 i \pi \text{Re}\left( \overline{\xi} g_{\mathbf{A} * \mathbf{B}}(x) \right)}   w_{\mathbf{a}_k}(x) d\mu(x) \right|^2 $$
$$  \lesssim e^{\varepsilon \alpha n} e^{-\lambda \delta (2k-1) n} \underset{\mathbf{B} \in \mathcal{R}_{n+1}^k}{\underset{\mathbf{A} \in \mathcal{R}_{n+1}^{k+1} }{\sum_{\mathbf{A} \leftrightarrow \mathbf{B}}}} \left| \int_{P_{b(\mathbf{A})}} e^{-2 i \pi \text{Re}\left( \overline{\xi} g_{\mathbf{A} * \mathbf{B}}(x) \right)}   w_{\mathbf{a}_k}(x) d\mu(x) \right|^2 ,$$
by Proposition 3.1 and 3.4, where one could increase $\alpha$ if necessary.
\end{proof}

\begin{lemma}
Define $$ \zeta_{j,\mathbf{A}}(\mathbf{b}) = e^{2 \lambda n} g_{\mathbf{a}_{j-1}' \mathbf{b}}'(x_{\mathbf{a}_j})$$ and $$\eta(x,y) :=  \overline{\xi} \left( g_{\mathbf{a}_k}(x) - g_{\mathbf{a}_k}(y) \right) e^{-2 k \lambda n} .$$
There exists $\alpha>0$ such that, for $ |\xi| \simeq e^{(2k+1) \lambda n} e^{\varepsilon_0 n} $ and $n$ large enough depending on $\varepsilon$:
$$ e^{-\varepsilon \alpha n} e^{-\lambda \delta (2k-1) n} \underset{\mathbf{B} \in \mathcal{R}_{n+1}^k}{\underset{\mathbf{A} \in \mathcal{R}_{n+1}^{k+1} }{\sum_{\mathbf{A} \leftrightarrow \mathbf{B}}}} \left| \int_{P_{b(\mathbf{A})}} e^{-2 i \pi \text{Re}\left( \overline{\xi} g_{\mathbf{A} * \mathbf{B}}(x) \right)}   w_{\mathbf{a}_k}(x) d\mu(x) \right|^2 $$
$$ \lesssim e^{-\lambda \delta (2k+1) n} \sum_{\mathbf{A} \in \mathcal{R}_{n+1}^{k+1}} \iint_{P_{b(\mathbf{A})}^2 }  \Bigg{|} \underset{\mathbf{A} \leftrightarrow \mathbf{B}}{\sum_{\mathbf{B} \in \mathcal{R}_{n+1}^k}} e^{2 i \pi \text{Re}\left( \eta(x,y) \zeta_{1,\mathbf{A}}(\mathbf{b}_1) \dots \zeta_{k,\mathbf{A}}(\mathbf{b}_k) \right)} \Bigg| d\mu(x) d\mu(y) +  e^{-(\lambda-\varepsilon_0) n } .$$

\end{lemma}

\begin{proof}

We expand the integral term and use Proposition 3.1 to get

$$ e^{-\lambda \delta (2k-1) n} \underset{\mathbf{B} \in \mathcal{R}_{n+1}^k}{\underset{\mathbf{A} \in \mathcal{R}_{n+1}^{k+1} }{\sum_{\mathbf{A} \leftrightarrow \mathbf{B}}}} \iint_{P_{b(\mathbf{A})}^2}  e^{2 i \pi \text{Re}\left( \overline{\xi} \left( g_{\mathbf{A} * \mathbf{B}}(x) -  g_{\mathbf{A} * \mathbf{B}}(y) \right) \right)}   w_{\mathbf{a}_k}(x) w_{\mathbf{a}_k}(y) d\mu(x) d\mu(y)  $$
$$ \leq e^{-\lambda \delta (2k-1) n} \sum_{\mathbf{A} \in \mathcal{R}_{n+1}^{k+1}} \iint_{P_{b(\mathbf{A})}^2} w_{\mathbf{a}_k}(x) w_{\mathbf{a}_k}(y) \Bigg{|} \underset{\mathbf{A} \leftrightarrow \mathbf{B}}{\sum_{\mathbf{B} \in \mathcal{R}_{n+1}^k}} e^{2 i \pi \text{Re}\left( \overline{\xi} \left( g_{\mathbf{A} * \mathbf{B}}(x) -  g_{\mathbf{A} * \mathbf{B}}(y) \right) \right)} \Bigg| d\mu(x) d\mu(y) $$
$$ \lesssim e^{\varepsilon \alpha n}  e^{-\lambda \delta (2k+1) n} \sum_{\mathbf{A} \in \mathcal{R}_{n+1}^{k+1}} \iint_{P_{b(\mathbf{A})}^2}  \Bigg{|} \underset{\mathbf{A} \leftrightarrow \mathbf{B}}{\sum_{\mathbf{B} \in \mathcal{R}_{n+1}^k}} e^{2 i \pi \text{Re}\left( \overline{\xi} \left( g_{\mathbf{A} * \mathbf{B}}(x) -  g_{\mathbf{A} * \mathbf{B}}(y) \right) \right)} \Bigg| d\mu(x) d\mu(y) .$$

The next step is to carefully linearize the phase. Here again, the construction of the $(D_a)_{a \in \mathcal{A}}$ as convex sets is really useful. \\

Fix some $\mathbf{A} \in \mathcal{R}_{n+1}^{k+1}$. For $x,y \in P_{b(\mathbf{A})}$, set $ \widehat{x} := g_{\mathbf{a}_k}(x)$ and $\widehat{y} = g_{\mathbf{a}_k}(y)$.
These are elements of $P_{b(\mathbf{B})}$, and so $[\widehat{x},\widehat{y}] \subset D_{b(\mathbf{B})}$. Hence, the following identity makes sense:

$$ g_{\mathbf{A} * \mathbf{B}}(x) -  g_{\mathbf{A} * \mathbf{B}}(y) = g_{\mathbf{A} \# \mathbf{B}}(\widehat{x}) -  g_{\mathbf{A} \# \mathbf{B}}(\widehat{y}) = \int_{[\widehat{x},\widehat{y}]} g_{\mathbf{A} \# \mathbf{B}}'(z) dz.$$
Therefore, we get
$$ \left| g_{\mathbf{A} * \mathbf{B}}(x) -  g_{\mathbf{A} * \mathbf{B}}(y) -  g_{\mathbf{A} \# \mathbf{B}}'(x_{\mathbf{a}_k})(\widehat{x}-\widehat{y}) \right| = \left| \int_{[\widehat{x},\widehat{y}]} \left( g_{\mathbf{A} \# \mathbf{B}}'(z) - g_{\mathbf{A} \# \mathbf{B}}'(x_{\mathbf{a}_k}) \right) dz \right| .$$
Then, $z \in D_{b(\mathbf{B})}$, and so $[z,x_{\mathbf{a}_k}] \subset D_{b(\mathbf{B})}$, and the following is well defined:

$$ g_{\mathbf{A} \# \mathbf{B}}'(z) - g_{\mathbf{A} \# \mathbf{B}}'(x_{\mathbf{a}_k}) = \int_{[z,x_{\mathbf{a}_k}]} g_{\mathbf{A} \# \mathbf{B}}''(\omega) d\omega .$$
Now, notice that since the maps are holomorphic, and since there exists a $\beta>0$ such that $ P_a + B(0,2 \beta) \subset D_a $ for any $a \in \mathcal{A}$, we can write by Cauchy's formula
$$ \left| g_{\mathbf{A} \# \mathbf{B}}''(\omega) \right| = \left|\frac{1}{2 i \pi} \oint_{\mathcal{C}(\omega,\beta)} \frac{g_{\mathbf{A} \# \mathbf{B}}'(s)}{(s-\omega)^2} ds \right| \leq \frac{\|g_{\mathbf{A} \# \mathbf{B}}'\|_{\infty,\mathcal{C}(\omega,\beta)}}{\beta}  ,$$
where $\mathcal{C}(\omega,\beta)$ is a circle centered at $\omega$ with radius $\beta$.
And so 

$$|g_{\mathbf{A} \# \mathbf{B}}'(z) - g_{\mathbf{A} \# \mathbf{B}}'(x_{\mathbf{a}_k})| \lesssim \|g_{\mathbf{A} \# \mathbf{B} }' \|_{\infty,D_{b(\mathbf{B})}} |z-x_{\mathbf{a}_k}| \lesssim e^{\varepsilon \alpha n} e^{- 2 k \lambda n } e^{- \lambda n} $$

by Proposition 3.1. Hence,
$$ \left| g_{\mathbf{A} * \mathbf{B}}(x) -  g_{\mathbf{A} * \mathbf{B}}(y) -  g_{\mathbf{A} \# \mathbf{B}}'(x_{\mathbf{a}_k})(\widehat{x}-\widehat{y}) \right| = \left| \int_{[\widehat{x},\widehat{y}]} \left( g_{\mathbf{A} \# \mathbf{B}}'(z) - g_{\mathbf{A} \# \mathbf{B}}'(x_{\mathbf{a}_k}) \right) dz \right| \lesssim e^{\varepsilon \alpha n} e^{- (2 k + 2) \lambda n } .$$

Then we relate $g_{\mathbf{A} \# \mathbf{B}}'(x_{\mathbf{a}_k})$ to $g_{\mathbf{a}_0' \mathbf{b}_1}'(x_{\mathbf{a}_1}) \dots g_{\mathbf{a}_{k-1}' \mathbf{b}_k}'(x_{\mathbf{a}_k}) $, using Cauchy's formula again:
$$ \left| g_{\mathbf{A} \# \mathbf{B}}'(x_{\mathbf{a}_k}) - g_{\mathbf{a}_0' \mathbf{b}_1}'(x_{\mathbf{a}_1}) \dots g_{\mathbf{a}_{k-1}' \mathbf{b}_k}'(x_{\mathbf{a}_k}) \right| =  \left| \prod_{j=1}^k g_{\mathbf{a}_{j-1}'\mathbf{b}_j}'(g_{\mathbf{a}_j' \mathbf{b}_{j+1}' \dots \mathbf{a}_{k-1}' \mathbf{b}_k}(x_{\mathbf{a}_{k}})) - \prod_{j=1}^k g_{\mathbf{a}_{j-1}'\mathbf{b}_j}'(x_{\mathbf{a}_{j}}) \right|  $$
$$ \leq \sum_{i=0}^{k-1} \left| \prod_{j=1}^i g_{\mathbf{a}_{j-1}'\mathbf{b}_j}'(x_{\mathbf{a}_{j}})  \prod_{j=i+1}^k g_{\mathbf{a}_{j-1}'\mathbf{b}_j}'(g_{\mathbf{a}_j' \mathbf{b}_{j+1}' \dots \mathbf{b}_k}(x_{\mathbf{a}_{k}})) - \prod_{j=1}^{i+1} g_{\mathbf{a}_{j-1}'\mathbf{b}_j}'(x_{\mathbf{a}_{j}}) \prod_{j=i+2}^k g_{\mathbf{a}_{j-1}'\mathbf{b}_j}'(g_{\mathbf{a}_j' \mathbf{b}_{j+1}' \dots \mathbf{b}_k}(x_{\mathbf{a}_{k}}))   \right| $$
$$ \leq \sum_{i=0}^{k-1} e^{\varepsilon \alpha n} e^{-2 (k-1) \lambda n} | g_{\mathbf{a}_{i}'\mathbf{b}_{i+1}}'(g_{\mathbf{a}_{i+1}' \mathbf{b}_{i+2}' \dots \mathbf{b}_k}(x_{\mathbf{a}_{k}})) - g_{\mathbf{a}_{i}'\mathbf{b}_{i+1}}'(x_{\mathbf{a}_{i+1}}))| $$
$$ \lesssim e^{\varepsilon \alpha n} e^{-2 \lambda  (k-1) n} \|g_{\mathbf{a}_{i}' \mathbf{b}_{i+1}}'\|_\infty \text{diam}(P_{\mathbf{a}_{i+1}}) \lesssim e^{\varepsilon \alpha n} e^{ -(2k+1) \lambda n}  .$$

Hence, since $|\widehat{x} - \widehat{y}| \lesssim e^{\varepsilon \alpha n} e^{- \lambda n}$,
$$ \left| g_{\mathbf{A} * \mathbf{B}}(x) -  g_{\mathbf{A} * \mathbf{B}}(y) - (\widehat{x}-\widehat{y}) \prod_{j=1}^k g_{\mathbf{a}_{j-1}'\mathbf{b}_j}'(x_{\mathbf{a}_{j}})  \right| \lesssim  e^{\varepsilon \alpha n} e^{ -(2k+2) \lambda n} .$$

From this estimate, we can relate our problem to a linearized one, as follows:

 $$ \scalemath{0.9}{ e^{-\lambda \delta (2k+1) n} \sum_{\mathbf{A} \in \mathcal{R}_{n+1}^{k+1}} \iint_{P_{b(\mathbf{A})^2}}  \Bigg{|} \underset{\mathbf{A} \leftrightarrow \mathbf{B}}{\sum_{\mathbf{B} \in \mathcal{R}_{n+1}^k}} e^{2 i \pi \text{Re}\left( \overline{\xi} \left( g_{\mathbf{A} * \mathbf{B}}(x) -  g_{\mathbf{A} * \mathbf{B}}(y) \right) \right)} - e^{2 i \pi \text{Re}( \overline{\xi}   g_{\mathbf{a}_{0}'\mathbf{b}_1}'(x_{\mathbf{a}_{1}}) \dots g_{\mathbf{a}_{k-1}'\mathbf{b}_k}'(x_{\mathbf{a}_{k}}) (\widehat{x}-\widehat{y}) )} \Bigg| d\mu(x) d\mu(y) }$$
$$\scalemath{0.9}{ = e^{-\lambda \delta (2k+1) n} \sum_{\mathbf{A} \in \mathcal{R}_{n+1}^{k+1}} \iint_{P_{b(\mathbf{A})^2}} \underset{\mathbf{A} \leftrightarrow \mathbf{B}}{\sum_{\mathbf{B} \in \mathcal{R}_{n+1}^k}} \Bigg{|}  e^{2 i \pi \text{Re}\left( \overline{\xi} \left( g_{\mathbf{A} * \mathbf{B}}(x) -  g_{\mathbf{A} * \mathbf{B}}(y) - g_{\mathbf{a}_{0}'\mathbf{b}_1}'(x_{\mathbf{a}_{1}}) \dots g_{\mathbf{a}_{k-1}'\mathbf{b}_k}'(x_{\mathbf{a}_{k}})\left( \widehat{x} - \widehat{y} \right) \right) \right)} - 1 \Bigg| d\mu(x) d\mu(y) }$$
$$\scalemath{0.9}{ \lesssim e^{-\lambda \delta (2k+1) n} \sum_{\mathbf{A} \in \mathcal{R}_{n+1}^{k+1}} \iint_{P_{b(\mathbf{A})^2}} \underset{\mathbf{A} \leftrightarrow \mathbf{B}}{\sum_{\mathbf{B} \in \mathcal{R}_{n+1}^k}}  | \xi| \left| g_{\mathbf{A} * \mathbf{B}}(x) -  g_{\mathbf{A} * \mathbf{B}}(y) - g_{\mathbf{a}_{0}'\mathbf{b}_1}'(x_{\mathbf{a}_{1}}) \dots g_{\mathbf{a}_{k-1}'\mathbf{b}_k}'(x_{\mathbf{a}_{k}})( \widehat{x} - \widehat{y} ) \right|   d\mu(x) d\mu(y) }$$

$$ \lesssim e^{ \varepsilon \alpha n} e^{-\lambda \delta (2k+1) n} e^{(k+1) \delta \lambda n} e^{k \delta \lambda n}  | \xi|  e^{- (2 k + 2) \lambda n } \simeq |\xi| e^{\varepsilon \alpha n} e^{-(2k+2)\lambda n}. $$
Now we see why we need to relate $n$ to $\xi$. We follow $\cite{SS20}$ and fix

$$ e^{(2k+1) \lambda (n-1)} e^{\varepsilon_0 (n-1)} \leq |\xi| \leq e^{(2k+1) \lambda n} e^{\varepsilon_0 n} .$$
This choice will ensure that the normalized phase $\eta$ will grow at a slow pace, of order of magnitude $e^{\varepsilon_0 n}$. It will be useful in the last section, where we will finally adjust $\varepsilon_0$.
 This relationship being fixed from now on, we get:
$$ \scalemath{0.9}{ e^{- \lambda \delta (2k+1) n}  \sum_{\mathbf{A} \in \mathcal{R}_{n+1}^{k+1}} \iint_{P_{b(\mathbf{A})^2}}  \Bigg{|} \underset{\mathbf{A} \leftrightarrow \mathbf{B}}{\sum_{\mathbf{B} \in \mathcal{R}_{n+1}^k}} e^{2 i \pi \text{Re}\left( \overline{\xi} \left( g_{\mathbf{A} * \mathbf{B}}(x) -  g_{\mathbf{A} * \mathbf{B}}(y) \right) \right)} - e^{2 i \pi \text{Re}\left( \overline{\xi} g_{\mathbf{a}_{0}'\mathbf{b}_1}'(x_{\mathbf{a}_{1}}) \dots g_{\mathbf{a}_{k-1}'\mathbf{b}_k}'(x_{\mathbf{a}_{k}})(\widehat{x}-\widehat{y}) \right)} \Bigg| d\mu(x) d\mu(y)} $$ $$ \lesssim e^{\varepsilon \alpha n} e^{-(\lambda-\varepsilon_0) n} .$$

So now we can focus on the term 
$$ e^{\varepsilon \alpha n}  e^{-\lambda \delta (2k+1) n} \sum_{\mathbf{A} \in \mathcal{R}_{n+1}^{k+1}} \iint_{P_{b(\mathbf{A})}}  \Bigg{|} \underset{\mathbf{A} \leftrightarrow \mathbf{B}}{\sum_{\mathbf{B} \in \mathcal{R}_{n+1}^k}} e^{2 i \pi \text{Re}\left( \overline{\xi} g_{\mathbf{a}_{0}'\mathbf{b}_1}'(x_{\mathbf{a}_{1}}) \dots g_{\mathbf{a}_{k-1}'\mathbf{b}_k}'(x_{\mathbf{a}_{k}}) \left( \widehat{x} - \widehat{y} \right) \right)} \Bigg| d\mu(x) d\mu(y) .$$

We re-scale the phase by defining, for any $\mathbf{b}$ such that $\mathbf{a}_{j-1} \rightsquigarrow \mathbf{b} \rightsquigarrow \mathbf{a}_j$:
$$ \zeta_{j,\mathbf{A}}(\mathbf{b}) = e^{2 \lambda n} g_{\mathbf{a}_{j-1}' \mathbf{b}}'(x_{\mathbf{a}_j}) $$
So that $\zeta_{j,\mathbf{A}}(\mathbf{b}) \sim 1$. We can then write
$$ \overline{\xi} g_{\mathbf{a}_{0}'\mathbf{b}_1}'(x_{\mathbf{a}_{1}}) \dots g_{\mathbf{a}_{k-1}'\mathbf{b}_k}'(x_{\mathbf{a}_{k}}) (\widehat{x}-\widehat{y}) = \eta(x,y) \zeta_{1,\mathbf{A}}(\mathbf{b}_1) \dots \zeta_{k,\mathbf{A}}(\mathbf{b}_k)$$
where $\eta(x,y) :=  \overline{\xi} \left( \widehat{x} - \widehat{y} \right) e^{-2 k \lambda n}   $.
\end{proof}

\begin{lemma}
Define $$ J_n := \{ e^{\varepsilon_0 n/2} \leq |\eta| \leq e^{2 \varepsilon_0 n}  \} .$$
There exists $\alpha > 0$ such that, for $ |\xi| \simeq e^{(2k+1) \lambda n} e^{\varepsilon_0 n} $ and $n$ large enough depending on $\varepsilon$:

$$e^{-\varepsilon \alpha n} e^{-\lambda \delta (2k+1) n} \sum_{\mathbf{A} \in \mathcal{R}_{n+1}^{k+1}} \iint_{P_{b(\mathbf{A})}^2}  \Bigg{|} \underset{\mathbf{A} \leftrightarrow \mathbf{B}}{\sum_{\mathbf{B} \in \mathcal{R}_{n+1}^k}} e^{2 i \pi \text{Re}\left( \eta(x,y) \zeta_{1,\mathbf{A}}(\mathbf{b}_1) \dots \zeta_{k,\mathbf{A}}(\mathbf{b}_k) \right)} \Bigg| d\mu(x) d\mu(y)  $$
$$ \lesssim  e^{-\lambda \delta (2k+1) n} \sum_{\mathbf{A} \in \mathcal{R}_{n+1}^{k+1}} \sup_{\eta \in J_n} \Bigg{|} \underset{\mathbf{A} \leftrightarrow \mathbf{B}}{\sum_{\mathbf{B} \in \mathcal{R}_{n+1}^k}} e^{2 i \pi \text{Re}\left( \eta \zeta_{1,\mathbf{A}}(\mathbf{b}_1) \dots \zeta_{k,\mathbf{A}}(\mathbf{b}_k) \right)} \Bigg|  + e^{- \delta_{AD} \varepsilon_0 n/2} .$$

\end{lemma}

\begin{proof}

To estimate $\eta(x,y)$, we need to control $|\widehat{x}-\widehat{y}|$. A first inequality is easy to get by convexity of $D_{b(\mathbf{A})}$:
$$ |\widehat{x}-\widehat{y}| = \left| \int_{[x,y]} g_{\mathbf{a}}'(z) dz \right| \leq e^{\varepsilon n} e^{- \lambda n} |x-y|. $$
The other inequality is subtler to get, but we already did the hard work. Recall lemma 2.3: it tells us that $\text{Conv}(P_{\mathbf{a}_k}) \subset D_{\mathbf{a}_k}$, thanks to Koebe's quarter theorem. Consequently,  $[\widehat{x},\widehat{y}] \subset D_{\mathbf{a}_k}$, and so we can write safely that
$$ |x-y| = \left| \int_{[\widehat{x},\widehat{y}]} (f^n)'(z) dz \right| \leq e^{\varepsilon n} e^{\lambda n} |\widehat{x} - \widehat{y}| .$$
Combining this with the fact that $\zeta_{j,\mathbf{A}}(\mathbf{b}) \sim 1$ gives us the following estimate for $\eta$ : 
$$ \forall x,y \in P_{b(\mathbf{A})}, \  e^{- \varepsilon (4k+1) n} e^{\varepsilon_0 n} |x-y| \leq |\eta(x,y)| \leq e^{\varepsilon (4k+1) n} e^{\varepsilon_0 n} |x-y| .$$

Choosing $\varepsilon$ and our Markov partition small enough ensures us that 
$$ \forall x,y \in P_{b(\mathbf{A})}, \  e^{- \varepsilon (4k+1) n} e^{\varepsilon_0 n} |x-y| \leq |\eta(x,y)| \leq  e^{2 \varepsilon_0 n} .$$

Finally, we need to reduce our problem to the case where $\eta$ is not too small, so that the sum of exponential may enjoys some cancellations. For this, we use the upper regularity of $\mu$. We have, if $x \in P_{b(\mathbf{A})}$:
$$ \mu\left( \left\{ y \in J, \ |x-y| \leq  e^{\varepsilon (4k+1) n} e^{- \varepsilon_0 n/2} \right\} \right) \lesssim e^{\varepsilon (4k+1) n \delta_{AD} } e^{- \varepsilon_0 \delta_{AD}n/2} .$$

Integrating in $x$ yields 
$$ \mu \otimes \mu\left( \left\{ x,y \in J, \ |x-y| \leq  e^{\varepsilon (4k+1) n} e^{- \varepsilon_0 n/2} \right\} \right) \lesssim e^{\varepsilon (4k+1) n \delta_{AD} } e^{- \varepsilon_0 \delta_{AD}n/2} .$$
This allows us to reduce our principal integral term to the part where $\eta$ is not small. Indeed, we get:
$$ e^{-\lambda \delta (2k+1) n} \sum_{\mathbf{A} \in \mathcal{R}_{n+1}^{k+1}} \iint_{P_{b(\mathbf{A})}^2}  \Bigg{|} \underset{\mathbf{A} \leftrightarrow \mathbf{B}}{\sum_{\mathbf{B} \in \mathcal{R}_{n+1}^k}} e^{2 i \pi \text{Re}\left( \eta(x,y) \zeta_{1,\mathbf{A}}(\mathbf{b}_1) \dots \zeta_{k,\mathbf{A}}(\mathbf{b}_k) \right)} \Bigg| d\mu(x) d\mu(y)  $$
$$ \lesssim e^{-\lambda \delta (2k+1) n} \sum_{\mathbf{A} \in \mathcal{R}_{n+1}^{k+1}} \iint_{ \{ x,y \in P_{b(\mathbf{A})}, \ |x-y| >  e^{\varepsilon (4k+1) n} e^{- \varepsilon_0 n/2} \} }  \Bigg{|} \underset{\mathbf{A} \leftrightarrow \mathbf{B}}{\sum_{\mathbf{B} \in \mathcal{R}_{n+1}^k}} e^{2 i \pi \text{Re}\left( \eta(x,y) \zeta_{1,\mathbf{A}}(\mathbf{b}_1) \dots \zeta_{k,\mathbf{A}}(\mathbf{b}_k) \right)} \Bigg| d\mu(x) d\mu(y)  $$
$$ + e^{\varepsilon \alpha n} e^{\varepsilon n (4k+1) \delta_{AD} } e^{- \varepsilon_0 \delta_{AD}n/2} , $$
by the cardinality estimate on $\mathcal{R}_{n+1}^{k}$.
In this integral term, we get $$  \eta(x,y) \geq  e^{- \varepsilon (4k+1) n} e^{\varepsilon_0 n} e^{\varepsilon (4k+1) n} e^{- \varepsilon_0 n/2} = e^{\varepsilon_0 n/2}, $$
and so we can bound:
$$  e^{-\lambda \delta (2k+1) n} \sum_{\mathbf{A} \in \mathcal{R}_{n+1}^{k+1}} \iint_{ \{ x,y \in P_{b(\mathbf{A})}, \ |x-y| >  e^{\varepsilon (4k+1) n} e^{- \varepsilon_0 n/2} \} }  \Bigg{|} \underset{\mathbf{A} \leftrightarrow \mathbf{B}}{\sum_{\mathbf{B} \in \mathcal{R}_{n+1}^k}} e^{2 i \pi \text{Re}\left( \eta(x,y) \zeta_{1,\mathbf{A}}(\mathbf{b}_1) \dots \zeta_{k,\mathbf{A}}(\mathbf{b}_k) \right)} \Bigg| d\mu(x) d\mu(y)  $$
$$ \leq e^{-\lambda \delta (2k+1) n} \sum_{\mathbf{A} \in \mathcal{R}_{n+1}^{k+1}} \sup_{\eta \in J_n} \Bigg{|} \underset{\mathbf{A} \leftrightarrow \mathbf{B}}{\sum_{\mathbf{B} \in \mathcal{R}_{n+1}^k}} e^{2 i \pi \text{Re}\left( \eta \zeta_{1,\mathbf{A}}(\mathbf{b}_1) \dots \zeta_{k,\mathbf{A}}(\mathbf{b}_k) \right)} \Bigg|  .$$ \end{proof}

Combining all those lemmas gives us Theorem 4.1.


\section{The sum product phenomenon}

\subsection{A key theorem }

We will use the following theorem of Li, which generalize a previous theorem of Bourgain in the complex case. It can be found as follows in \cite{LNP19}, and a proof can be found in $\cite{Li18}$.

\begin{theorem}

Given $\gamma > 0$, there exist $ \varepsilon_2 \in \ ]0,1[$ and $k \in \mathbb{N}^*$ such that the following holds for
$\eta \in \mathbb{C}$ with $|\eta| > 1$. Let $C_0 > 1$ and let $\lambda_1, \dots , \lambda_k$ be Borel measures supported on the annulus $\{ z \in \mathbb{C} \ , \ C_0^{-1} \leq |z| \leq C_0 \}$ with total mass less than $C_0$. Assume that each $\lambda_j$ satisfies the projective non concentration property, that is,
$$\forall \sigma \in [C_0 |\eta|^{−1}, C_0^{-1} |\eta|^{- \varepsilon_2}], \ \sup_{a,\theta \in \mathbb{R}} \ \lambda_j \{z \in \mathbb{C}, \ | \text{Re} (e^{i\theta} z) − a|  \leq \sigma \} \leq C_0 \sigma^\gamma .$$
Then there exists a constant $C_1$ depending only on $C_0$ and $\gamma$ such that
$$ \left| \int \exp(2 i \pi Re( \eta z_1 \dots z_k))  d\lambda_1(z_1) \dots d\lambda_k(z_k) \right| \leq C_1 |\eta|^{- \varepsilon_2} .$$

\end{theorem}
Unfortunately, in our case the use of large deviations does not allow us to apply it straightforwardly. To highlight the dependence of $C_1$ when $C_0$ is permitted to grow gently, we prove the following theorem.

\begin{theorem}

Fix $0< \gamma < 1$. There exist $ \varepsilon_1 > 0$ and $k \in \mathbb{N}^*$ such that the following holds for
$\eta \in \mathbb{C}$ with $|\eta|$ large enough. Let $1< R < |\eta|^{\varepsilon_1}$ and let $\lambda_1, \dots , \lambda_k$ be Borel measures supported on the annulus $\{ z \in \mathbb{C} \ , \ R^{-1} \leq |z| \leq R \}$ with total mass less than $R$. Assume that each $\lambda_j$ satisfies the following projective non concentration property:
$$\forall \sigma \in [|\eta|^{−2}, |\eta|^{- \varepsilon_1}], \ \sup_{a,\theta \in \mathbb{R}} \ \lambda_j \{z \in \mathbb{C}, \ | \text{Re} (e^{i\theta} z) − a|  \leq \sigma \} \leq \sigma^\gamma .$$
Then there exists a constant $c>0$ depending only on $\gamma$ such that
$$ \left| \int \exp(2 i \pi Re( \eta z_1 \dots z_k))  d\lambda_1(z_1) \dots d\lambda_k(z_k) \right| \leq c |\eta|^{- \varepsilon_1} .$$

\end{theorem}

\begin{proof}

 Fix $0<\gamma<1$, and let $\varepsilon_2$ and $k$ given by the previous theorem. Choose $\varepsilon_1 := \frac{\varepsilon_2}{2(2k+1)}$. Let $1<R<|\eta|^{\varepsilon_1}$, and let $\lambda_1 , \dots, \lambda_k$ be measures that satisfy the hypothesis of Theorem 5.2. We are going to use a dyadic decomposition. \\

Let $m := \lfloor \log_2(R) \rfloor + 1$.
Then $\lambda_j$ is supported in the annulus
$$ \{ z \in \mathbb{C} \ , \ 2^{-m} \leq |z| \leq 2^m \}. $$
Define, for $A$ a borel subset of $\mathbb{C}$ and for $r=-m+1,\dots , m$:  $$ \lambda_{j,r}(A) := R^{-1} \lambda_j\left( 2^{r} \left( A \cap \{ 2^{-1} \leq |z| < 1 \} \right) \right) .$$

Those measures are all supported in $ \{ 1/2 \leq |z| \leq 1 \} $, and have total mass $\lambda_{j,r}(\mathbb{C}) \leq 1$. \\

Moreover, a non concentration property is satisfied by each $\lambda_{j,r}$. If we fix some $r_1,\dots,r_k$ between $-m+1$ and $m$ and define ${\eta}_{r_1\dots r_k} := 2^{r_1+\dots+ r_k} \eta$, then $|\eta_{r_1,\dots, r_k}| \geq (2R)^{-k} |\eta| > 2^{-k} |\eta|^{1-k \varepsilon_1} > 1$ if $\eta$ is large enough. Let $\sigma \in [|\eta_{r_1,\dots, r_k}|^{-1}, |\eta_{r_1,\dots, r_k}|^{-{\varepsilon_2}}]$. Then
$$\lambda_{j,r} \left( \{ | \text{Re} (e^{i\theta} z) − a|  \leq \sigma \} \right) = R^{-1} \lambda_j \left( 2^{r} \left( \{ | \text{Re} (e^{i\theta} z) − a|  \leq \sigma \} \cap \{ 2^{-1} \leq |z| < 1 \} \right)\right) $$
$$ = R^{-1} \lambda_j \left( \{ | \text{Re} (e^{i\theta} z) − 2^{r} a|  \leq 2^{r} \sigma \} \cap \{ 2^{r-1} \leq |z| < 2^{r} \} \right) $$
$$ \leq R^{-1} \lambda_j \left( \{ | \text{Re} (e^{i\theta} z) − 2^{r} a|  \leq 2^{r} \sigma \} \right) .$$

Since $2^{r} \sigma \in \left[ 2^{r} |\eta_{r_1,\dots,r_k}|^{-1} , 2^{r} |\eta_{r_1,\dots,r_k}|^{- \varepsilon_2} \right] \subset \left[ (2R)^{-(k+1)} |\eta|^{-1} , (2R)^{k+1}|\eta|^{- \varepsilon_2}\right] \subset \left[|\eta|^{-2}, |\eta|^{-\varepsilon_1}\right]$ if $|\eta|$ is large enough, we can use the non-concentration hypothesis assumed for each $\lambda_j$ to get:
$$ \lambda_{j,r} \left( \{ | \text{Re} (e^{i\theta} z) − a|  \leq \sigma \right) \leq R^{-1} (2^{r} \sigma)^\gamma \leq 2 \sigma^\gamma .$$

Hence, by the previous theorem, there exists a constant $C_1$ depending only on $\gamma$ such that
$$ \left| \int \exp(2 i \pi Re( \eta_{r_1 \dots r_k} z_1 \dots z_k))  d\lambda_{1,r_1}(z_1) \dots d\lambda_{k,r_k}(z_k) \right| \leq C_1 |\eta_{r_1 \dots r_k}|^{- \varepsilon_2} . $$
Finally, since $$ \lambda_j(A) = R \sum_{r=-m+1}^m \lambda_{j,r}(2^{-r} A), $$
we get that:
$$  \left| \int \exp(2 i \pi Re( \eta z_1 \dots z_k))  d\lambda_1(z_1) \dots d\lambda_k(z_k) \right| $$
$$ \leq \sum_{r_1 , \dots r_k} R^k \left| \int \exp(2 i \pi Re( \eta z_1 \dots z_k))  d\lambda_{1,r_1}(2^{-r_1} z_1) \dots d\lambda_{k,r_k}(2^{-r_k} z_k) \right| $$
$$ =  \sum_{r_1 , \dots r_k} R^k \left| \int \exp(2 i \pi Re( \eta_{r_1 \dots r_k} z_1 \dots z_k))  d\lambda_{1,r_1}(z_1) \dots d\lambda_{k,r_k}(z_k) \right|$$
$$ \leq C_1 (2m)^k R^k |\eta_{r_1 \dots r_k}|^{- \varepsilon_2}  \leq 4^k C_1 m^k R^{2k} |\eta|^{- \varepsilon_2} .$$

Since $m \leq \log_2(R) +1$, and since $k$ depends only on $\gamma$, there exists a constant $c$ that depends only on $\gamma$ such that $ 4^k C_1 m^k R^{2k} \leq c R^{2k+1}$ for any $R> 1$. Finally, $ c R^{2k+1} |\eta|^{- \varepsilon_2} \leq |\eta|^{-\varepsilon_1} $. \end{proof}

\begin{theorem}

Fix $0 < \gamma < 1$. Then there exist $k \in \mathbb{N}^*$ and $\varepsilon_1 > 0$ depending only on $\gamma$ such that the following holds for $\eta \in \mathbb{C}$ with $|\eta|$ large enough. Let $1 < R < |\eta|^{\varepsilon_1}$ , $N > 1$ and $\mathcal{Z}_1, . . . , \mathcal{Z}_k$ be finite sets such that $ \# \mathcal{Z}_j ≤ RN$. Consider some maps $\zeta_j : \mathcal{Z}_j \rightarrow \mathbb{C} $, $j = 1, \dots , k$, such that, for all $j$:
$$ \zeta_j ( \mathcal{Z}_j ) \subset \{ z \in \mathbb{C} \ , \ R^{-1} \leq |z| \leq R \}$$ and 
$$\forall \sigma \in [|\eta|^{−2}, |\eta|^{- \varepsilon_1}], \quad \sup_{a,\theta \in \mathbb{R}} \ \# \{\mathbf{b} \in \mathcal{Z}_j , \ | \text{Re} (e^{i\theta} \zeta_j(\mathbf{b})) − a|  \leq \sigma \} ≤ N \sigma^{\gamma}.$$
Then there exists a constant $c > 0$ depending only on $\gamma$ such that 
$$ \left| N^{−k} \sum_{\mathbf{b}_1 \in \mathcal{Z}_1,\dots,\mathbf{b}_k \in \mathcal{Z}_k} \exp\left(2 i \pi \text{Re}\left(\eta \zeta_1(\mathbf{b}_1) \dots \zeta_k(\mathbf{b}_k) \right) \right) \right| ≤ c |\eta|^{−{\varepsilon_1}}.$$

\end{theorem}

\begin{proof}

Define our measures as sums of dirac mass: $$ \lambda_j := \frac{1}{N} \sum_{\mathbf{b} \in \mathcal{Z}_j} \delta_{\zeta_j(\mathbf{b})} .$$

We see that $\lambda_j$ is supported in the annulus $ \{ z \in \mathbb{C} \ , \ R^{-1} \leq |z| \leq R \} $. The total mass is bounded by
$$ \lambda_j(\mathbb{C}) \leq N^{-1} \# \mathcal{Z}_j \leq R.$$

Then, if $\sigma \in [|\eta|^{−2}, |\eta|^{- \varepsilon_1}]$, we have, for any $a,\theta \in \mathbb{R}$:
$$ \lambda_j \{z \in \mathbb{C}, \ | \text{Re} (e^{i\theta} z) − a|  \leq \sigma \} = \frac{1}{N} \# \left\{\mathbf{b} \in \mathcal{Z}_j  , \ |\text{Re}(e^{i \theta} \zeta_j(\mathbf{b}))-a| \leq \sigma  \right\} \leq  \sigma^\gamma. $$
Hence, the previous theorem applies directly, and gives us the desired result. \end{proof}

\subsection{End of the proof assuming non concentration}

We will use Theorem 5.3 on the maps $\zeta_{j,\mathbf{A}}$.
Let's carefully define the framework. \\
For some fixed $\mathbf{A} \in \mathcal{R}_{n+1}^{k+1}$, define for $j=1, \dots, k$  $$ \mathcal{Z}_j := \{ \mathbf{b} \in \mathcal{R}_{n+1} , \mathbf{a}_{j-1} \rightsquigarrow \mathbf{b} \rightsquigarrow \mathbf{a}_j \  \} .$$

The maps $\zeta_{j,\mathbf{A}}(\mathbf{b}) := e^{2 \lambda n} g_{\mathbf{a}_{j-1}' \mathbf{b}}'(x_{{\mathbf{a}}_j})$ are defined on $\mathcal{Z}_j$. There exists a constant $\alpha>0$ (which will be fixed from now on) such that
$$\# \mathcal{Z}_j \leq e^{\varepsilon \alpha n} e^{ \delta \lambda n} $$
and
$$ \zeta_{j,\mathbf{A}}( \mathcal{Z}_j ) \subset \{ z \in \mathbb{C} \ , \ e^{-\varepsilon \alpha n} \leq |z| \leq e^{\varepsilon \alpha n} \} .$$

Let $\gamma>0$ small enough. Theorem 5.3 then fixes $k$ and some $\varepsilon_1$.
The goal is to apply Theorem 5.3 to the maps $\zeta_{j,\mathbf{A}}$, for $N := e^{\lambda \delta n}$, $R:=e^{\varepsilon \alpha n}$ and $\eta \in J_n$. Notice that choosing $\varepsilon$ small enough ensures that $R<|\eta|^{\varepsilon_1}$, and taking $n$ large enough ensures that $|\eta|$ is large. 
If we are able to prove the non concentration hypothesis in this context, then Theorem 5.3 can be applied and we would be able to conclude the proof of the main Theorem 1.4.
Indeed, we already know that
$$ e^{- \varepsilon  \alpha n} |\widehat{\mu}(\xi)|^2 \lesssim e^{-\lambda \delta (2k+1) n} \sum_{\mathbf{A} \in \mathcal{R}_{n+1}^{k+1}} \sup_{\eta \in J_n} \Bigg{|} \underset{\mathbf{A} \leftrightarrow \mathbf{B}}{\sum_{\mathbf{B} \in \mathcal{R}_{n+1}^k}} e^{2 i \pi \text{Re}\left( \eta \zeta_{1,\mathbf{A}}(\mathbf{b}_1) \dots \zeta_{k,\mathbf{A}}(\mathbf{b}_k) \right)} \Bigg| $$
$$ \quad \quad  \quad \quad  \quad \quad  \quad \quad +  e^{- \varepsilon  \alpha n} \mu(J \setminus R_{n+1}^{2k+1}(\varepsilon) )^2 + \kappa^{-2n} + e^{-(\lambda-\varepsilon_0) n } + e^{-\varepsilon_0 \delta_{AD} n/2} $$

by Proposition 4.1. Since every error term already enjoys exponential decay in $n$, we just have to deal with the sum of exponentials. By Theorem 5.3, we can then write
$$ \sup_{\eta \in J_n} \Bigg{|} \underset{\mathbf{A} \leftrightarrow \mathbf{B}}{\sum_{\mathbf{B} \in \mathcal{R}_{n+1}^k}} e^{2 i \pi \text{Re}\left( \eta \zeta_{1,\mathbf{A}}(\mathbf{b}_1) \dots \zeta_{k,\mathbf{A}}(\mathbf{b}_k) \right)} \Bigg| \leq c e^{\lambda k \delta n} e^{- \varepsilon_0 \varepsilon_1 n/2 } ,$$

and hence we get
$$ e^{-\lambda \delta (2k+1) n} \sum_{\mathbf{A} \in \mathcal{R}_{n+1}^{k+1}} \sup_{\eta \in J_n} \Bigg{|} \underset{\mathbf{A} \leftrightarrow \mathbf{B}}{\sum_{\mathbf{B} \in \mathcal{R}_{n+1}^k}} e^{2 i \pi \text{Re}\left( \eta \zeta_{1,\mathbf{A}}(\mathbf{b}_1) \dots \zeta_{k,\mathbf{A}}(\mathbf{b}_k) \right)} \Bigg| $$ $$ \lesssim e^{\varepsilon \alpha n} e^{-\lambda \delta  (2k+1) n} e^{\lambda \delta (k+1) n} e^{\lambda \delta k n} e^{- \varepsilon_0 \varepsilon_1 n/2} \lesssim e^{\varepsilon \alpha n} e^{-\varepsilon_0 \varepsilon_1 n/2} .$$

Now, we see that we can choose $\varepsilon$ small enough so that all terms enjoy exponential decay in $n$, and since $ |\xi| \simeq e^{(2k+1) \lambda n} e^{\varepsilon_0 n} $, we have proved polynomial decay of $|\widehat{\mu}|^2$. \\

\section{The non-concentration hypothesis}

The last part of this paper is devoted to the proof of the non-concentration hypothesis that we just used.

\subsection{Statement of the non-concentration theorem}

\begin{definition}

For a given $\mathbf{A} \in \mathcal{R}_{n+1}^{k+1}$, define for $j=1, \dots, k$  $$ \mathcal{Z}_j := \{ \mathbf{b} \in \mathcal{R}_{n+1} , \ \mathbf{a}_{j-1} \rightsquigarrow \mathbf{b} \rightsquigarrow \mathbf{a}_j \  \} $$

Then define $$\zeta_{j,\mathbf{A}}(\mathbf{b}) := e^{2 \lambda n} g_{\mathbf{a}_{j-1}' \mathbf{b}}'(x_{{\mathbf{a}}_j})$$ on $\mathcal{Z}_j$. The following is satisfied, for some fixed constant $\alpha>0$:

 $$\# \mathcal{Z}_j \leq e^{\varepsilon \alpha n} e^{ \delta \lambda n} $$
 and
$$ \zeta_{j,\mathbf{A}}( \mathcal{Z}_j ) \subset \{ z \in \mathbb{C} \ , \ e^{-\varepsilon \alpha n} \leq |z| \leq e^{\varepsilon \alpha n} \} .$$

\end{definition}

We are going to prove the following fact, which will allow us to apply Theorem 5.3 for $\eta \in J_n$, $R:=e^{\varepsilon \alpha n}$ and $N:= e^{\lambda \delta n} $.

\begin{theorem}[non concentration]
There exists $\gamma>0$, and we can choose $\varepsilon_0>0$, such that the following holds. \\

Let $\eta \in \{ e^{\varepsilon_0 n/2} \leq |\eta| \leq e^{2 \varepsilon_0 n} \}$. Let $\mathbf{A} \in \mathcal{R}_{n+1}^{k+1}$. Then, if $n$ is large enough,

$$\forall \sigma \in [ |\eta|^{−2}, |\eta|^{- \varepsilon_1}], \quad \sup_{a,\theta \in \mathbb{R}} \ \# \left\{\mathbf{b} \in \mathcal{Z}_j , \ | \emph{\text{Re}} (e^{i\theta} \zeta_{j,\mathbf{A}}(\mathbf{b})) − a|  \leq \sigma \right\} ≤ N \sigma^{\gamma}, $$

where $R:= e^{\varepsilon \alpha n}$, $N:= e^{\lambda \delta n}$ and $\varepsilon_1$ and $k$ are fixed by Theorem 5.3.

\end{theorem}

\subsection{Beginning of the proof}

The proof of Theorem 6.1 is in two parts. First of all, we see that the non-concentration hypothesis formulated above counts how many $\zeta_{j,\mathbf{A}}$ are in a strip. We begin by reducing the non-concentration to a counting problem in small disks.

\begin{lemma}

If $\varepsilon_0$ and $\gamma$ are such that, for $\sigma \in [ e^{- 5 \varepsilon_0 n} ,  e^{ - \varepsilon_1 \varepsilon_0 n/5 }  ]$, $$\sup_{R^{-1} \leq |a| \leq R} \#  \{ \mathbf{b} \in \mathcal{Z}_j,  \ \zeta_{j,\mathbf{A}}(\mathbf{b}) \in B(a,\sigma) \} \leq N \sigma^{1+\gamma} ,$$

then Theorem 6.1 is true.

\end{lemma}

\begin{proof}

Suppose that the result in lemma 6.2 is true.
Then, we know that squares $C_{c,\theta,\sigma} := e^{-i \theta} B_\infty(c,\sigma) =  \{ z \in \mathbb{C}, \ |\text{Re}(e^{i \theta}z-c)| \leq \sigma, |\text{Im}(e^{i \theta}z-c)| \leq \sigma \} $
are included in disks $B(c,\sigma \sqrt{2})$. (We note $B_\infty$ the balls for the $L^\infty$ norm.)
Hence,
$$ \forall \sigma \in [ e^{-  5 \varepsilon_0 n} , e^{ - \varepsilon_1 \varepsilon_0 n/5 }  ], $$
$$ \# \{ \mathbf{b} \in \mathcal{Z}_j , \ \zeta_{j,\mathbf{A}} \in C_{c,\theta,\sigma} \} \leq \# \{ \mathbf{b} \in \mathcal{Z}_j , \ \zeta_{j,\mathbf{A}} \in B(c,\sigma \sqrt{2}) \} \leq N \sqrt{2}^{1+\gamma} \sigma^{1+\gamma} .$$
Our next move is to cover the strip $S_{a,\theta,\sigma} := \left\{z \in \mathbb{C} , \ |  \text{Re} (e^{i\theta} z) − a|  \leq \sigma \right\}$
by squares $C_{c,\theta,\sigma}$. First of all, recall that $\zeta_{j,\mathbf{A}}(\mathcal{Z}_j) \subset B(0,R) $.
Hence, we can write, for a fixed $a$ and $\theta$:

$$ \# \{ \mathbf{b} \in \mathcal{Z}_j , \ \zeta_{j,\mathbf{A}}(\mathbf{b}) \in S_{a,\theta,\sigma}  \} \leq \sum_{c \in K(\sigma,R)} \# \{ \mathbf{b} \in \mathcal{Z}_j , \ \zeta_{j,\mathbf{A}} \in C_{c,\theta,\sigma}  \} $$

where $K(\sigma,n) := \{ e^{- i \theta}( a + i k \sigma ) \ | \ k = - \lfloor R/\sigma \rfloor, \dots , \lfloor R/\sigma \rfloor  \}$ is the set of the centers of the squares, chosen so that it covers our restricted strip. Hence, for $\sigma \in [ e^{-4 \varepsilon_0 n} , e^{ - \varepsilon_1 \varepsilon_0 n/2 }  ]$, 

$$ \# \{ \mathbf{b} \in \mathcal{Z}_j , \ \zeta_{j,\mathbf{A}}(\mathbf{b}) \in S_{a,\theta,\sigma}  \}  \lesssim \frac{R}{\sigma} N \sqrt{2}^{1+\gamma} \sigma^{1+\gamma} . $$

Then, since $\sigma$ goes to zero exponentially fast in $n$, and since $R$ grows slowly since $\varepsilon$ can be chosen as small as we want, we can just take $n$ large enough so that

$$  \# \{ \mathbf{b} \in \mathcal{Z}_j , \ \zeta_{j,\mathbf{A}}(\mathbf{b}) \in S_{a,\theta,\sigma}  \}  \lesssim \frac{R}{\sigma} N \sqrt{2}^{1+\gamma} \sigma^{1+\gamma} \leq N \sigma^{\gamma/2}, $$

and we are done. \end{proof}

\begin{definition}

Since $\exp : \mathbb{C} \rightarrow \mathbb{C}^* $ is a surjective, holomorphic morphism, with kernel $2 i \pi \mathbb{Z}$, it induces a biholomorphism $ \exp : \mathbb{C}/2 i \pi \mathbb{Z} \rightarrow \mathbb{C}^* $. Define by $ \log : \mathbb{C}^* \rightarrow \mathbb{C}/2 i \pi \mathbb{Z} $ its holomorphic inverse.
Note $ \text{mod}_{2 i \pi} : \mathbb{C} \rightarrow \mathbb{C}/2 i \pi \mathbb{Z} $ the projection.
\end{definition}

Now, we reduce the problem to a counting estimate on $\log(\zeta_{j,\mathbf{A}})$.

\begin{lemma}

If $\varepsilon_0$ and $\gamma$ are such that, for $\sigma \in [ e^{-6 \varepsilon_0 n} , e^{ - \varepsilon_1 \varepsilon_0 n/6 }  ]$, $$\sup_{a \in \mathbb{C}} \#  \left\{ \mathbf{b} \in \mathcal{Z}_j,  \ \log g_{\mathbf{a}_{j-1}' \mathbf{b}}' (x_{\mathbf{a}_j}) \in \emph{\text{mod}}_{2 i \pi} \left(B_\infty(a,\sigma) \right) \right\} \leq N \sigma^{1+\gamma} ,$$

then Theorem 6.1 is true.

\end{lemma}

\begin{proof}

Suppose that the estimate is true.
Let $ \sigma \in [ e^{-5 \varepsilon_0 n} , e^{ - \varepsilon_1 \varepsilon_0 n/5 }  ] $. \\

Fix an euclidean ball $B(a,\sigma)$, where $a=r_0e^{i \theta_0}$ satisfies $r_0 \in [R^{-1},R]$ and $\theta_0 \in ]-\pi,\pi]$. Elementary trigonometry allows us to see that 
$$ B(a,\sigma) \subset \left\{ re^{i \theta} \in \mathbb{C} | \ r \in [r_0 - \sigma, r_0 + \sigma] , \ \theta \in [\theta_0 - \arctan(\sigma/r_0), \theta_0 + \arctan(\sigma/r_0) )] \ \right\} .$$

Then, since for $n$ large enough $$\ln(r_0+\sigma)-\ln(r_0-\sigma) = \ln(1+ \sigma r_0^{-1}) - \ln(1 - \sigma r_0^{-1}) \leq 4 \sigma r_0^{-1} \leq 4 \sigma R $$ and $$ 2\arctan(\sigma/r_0) \leq 4 \sigma r_0^{-1} \leq 4 \sigma R ,$$

we find that
$$ B(a,\sigma) \subset \exp B_\infty( (\ln(r_0),\theta_0), 4 \sigma R ) . $$

Hence:
$$  \#  \{ \mathbf{b} \in \mathcal{Z}_j,  \ \zeta_{j,\mathbf{A}}(\mathbf{b}) \in B(a,\sigma) \} \leq  \#  \{ \mathbf{b} \in \mathcal{Z}_j,  \ \zeta_{j,\mathbf{A}}(\mathbf{b}) \in \exp B_\infty( (\ln(r_0),\theta_0) , 4 R \sigma)  \} $$
$$ = \#  \{ \mathbf{b} \in \mathcal{Z}_j,  \ \log \zeta_{j,\mathbf{A}}(\mathbf{b}) \in \text{mod}_{2 i \pi} \left(  B_\infty( (\ln(r_0),\theta_0) , 4 R \sigma) \right) \}$$
$$ = \#  \{ \mathbf{b} \in \mathcal{Z}_j,  \ \log g_{\mathbf{a}_{j-1}' \mathbf{b}}'(x_{\mathbf{a}_j}) \in \text{mod}_{2 i \pi} \left(  B_\infty( (\ln(r_0) - 2 n \lambda,\theta_0) , 4 R \sigma) \right) \} $$
$$ \leq N (4 R \sigma)^{1+\gamma} \leq N \sigma^{1+\gamma/2} $$
provided $n$ is large enough. So the inequality of lemma 6.2 is satisfied, and so Theorem 6.1 is true. \end{proof}

\subsection{End of the proof}

We are going to prove that the estimate in lemma 6.3 is satisfied for all $C^1(V,\mathbb{R})$ potentials $\varphi$. (The dependence in $\varphi$ is hidden in the definition of the $\varphi$-regular words.) For this, we will need a generalization of a theorem, borrowed from $\cite{OW17}$.

\begin{theorem}

We work on $\mathfrak{U} := \bigsqcup_{a \in \mathcal{A}} U_a $, the formal disjoint union of the $U_a$. Define $$C^1_b(\mathfrak{U},\mathbb{C}) := \left\{ \mathfrak{h} = (h_a)_{a \in \mathcal{A}} \ | \ h_a \in C^1(U_a,\mathbb{C}) , \ \| (h_a)_{a} \|_{C^1_b(\mathfrak{U},\mathbb{C}) }< \infty \right\} , $$ where $\| \cdot \|_{C^1_b(\mathfrak{U},\mathbb{C})}$ is the usual $C^1$ norm
$$ \| \mathfrak{h} \|_{C^1_b(\mathfrak{U},\mathbb{C})} = \sum_{a \in \mathcal{A}} \left( \|h_a\|_{\infty,U_a} + \| \nabla h_a \|_{\infty,U_a} \right). $$

On this Banach space, for $\varphi$ a normalized potential, $s \in \mathbb{C}$ and $l \in \mathbb{Z}$, we define a twisted transfer operator $\mathcal{L}_{\varphi,s,l} : C^1_b(\mathfrak{U},\mathbb{C}) \rightarrow C^1_b(\mathfrak{U},\mathbb{C}) $ as follows:

$$ \forall x \in U_b, \ \mathcal{L}_{\varphi,s,l}\mathfrak{h}(x) := \sum_{M_{ab}=1} e^{\varphi(g_{ab}(x))} |g_{ab}'(x)|^s \left(\frac{g_{ab}'(x)}{|g_{ab}'(x)|} \right)^{-l} \mathfrak{h}( g_{ab}(x) ) ,$$

where $g_{ab} : U_b \rightarrow U_a$. Iterating this transfer operator yields:

$$ \forall x \in U_b, \ \mathcal{L}_{\varphi,s,l}^n \mathfrak{h} (x) = \underset{\mathbf{a} \rightsquigarrow b}{\sum_{\mathbf{a} \in \mathcal{W}_{n+1}}} w_{\mathbf{a}}(x) |g_{\mathbf{a}}'(x)|^s \left(\frac{g_{\mathbf{a}}'(x)}{|g_{\mathbf{a}}'(x)|} \right)^{-l} \mathfrak{h}(g_{\mathbf{a}}(x)) .$$

Since $J$ is supposed to be not included in a circle, we have the following result. There exists $C>0$ and $\rho<1$ such that, for any $s \in \mathbb{C}$ such that $\text{Re}(s)=0$ and $|\text{Im}(s)|+|l| >1$,

$$ \| \mathcal{L}_{\varphi,s,l}^n \|_{C^1_b(\mathfrak{U},\mathbb{C})} \leq C_0 (|\text{Im}(s)|+|l|)^2 \rho^n .$$

\end{theorem}

It means that this twisted transfer operator is eventually uniformly contracting for large $l$ and $\text{Im}(s)$. This theorem will play another key role in this paper. 

\begin{remark}

In \cite{OW17}, the theorem was proved for the conformal measure. In $\cite{ShSt20}$, section 3.3, Sharp and Stylianou explain how we can generalize the theorem for a more general family of potentials, which covers the case of the measure of maximal entropy. The fully general theorem can be proved with some very minor modifications from the proof developed in $\cite{OW17}$: it will be explained in appendix B.
\end{remark}

\begin{proposition}

Define $\varepsilon_0 := \min( - \ln(\rho)/30,  \lambda/2  ) $. There exists $\gamma>0$ such that, \newline for $\sigma \in [e^{-6 \varepsilon_0 n} , e^{ - \varepsilon_1 \varepsilon_0 n/6 }  ]$ and if $n$ is large enough, $$\sup_{a} \#  \left\{ \mathbf{b} \in \mathcal{Z}_j,  \ -\log g_{\mathbf{a}_{j-1}' \mathbf{b}}' (x_{\mathbf{a}_j}) \in \emph{\text{mod}}_{2 i \pi} \left(B_\infty(a,\sigma) \right) \right\} \leq N \sigma^{1+\gamma} .$$

\end{proposition}

\begin{proof}
In the proof to come, all the $\simeq$ or $\lesssim$ will be uniform in $a$: the only relevant information here will be $\sigma$.
So fix $\sigma \in [ e^{- 6 \varepsilon_0 n} , e^{ - \varepsilon_1 \varepsilon_0 n/6 }  ]$, and fix a small square  $\text{mod}_{2 i \pi} \left(B_\infty(a,\sigma)\right) \subset \mathbb{C}/2 i \pi \mathbb{Z}$. The area of this square is $\sigma^2$.
Lift this square somewhere in $\mathbb{C}$, for example as $B_\infty(a,\sigma)$, and then define a bump function $\chi$ such that $\chi = 1$ on $B_\infty(a,\sigma)$, $\text{supp}(\chi) \subset B_\infty(a,2 \sigma)$ and such that $ \| \chi \|_{L^1(\mathbb{C})} \simeq \sigma^2 $. We can suppose that $ \| \partial_x^{k_1} \partial_y^{k_2} \chi \|_{L^1(\mathbb{C})} \simeq \sigma^{2-k_1-k_2} $. (For example, take $\chi(x) := \chi_0((x-a)\sigma^{-1})$ for $\chi_0$ a bump function around 0.) \\

Then, we can consider $h$, the $2 \pi \mathbb{Z}[i] := 2\pi(\mathbb{Z} + i \mathbb{Z})$ periodic map obtained by periodizing $\chi$. We can see it either as a smooth $2 \pi \mathbb{Z}[i]$ periodic map on $\mathbb{C}$, or as a smooth $2 \pi \mathbb{Z}$-periodic map on $\mathbb{C}/2 i \pi \mathbb{Z}$, or just as a smooth map on $\mathbb{C}/2  \pi \mathbb{Z}[i]$. \\

Then by construction, the periodicity of $h$ allows us to see that
$$ \mathbb{1}_{ \text{mod}_{2i \pi} \left( B_\infty(a,\sigma) \right) } \leq h .$$
Moreover, $h\left(-\log g_{\mathbf{a}_{j-1}' \mathbf{b}}' \right)(x_{\mathbf{a}_j}) $ is well defined, and so we can bound the desired cardinality with it. We have the following \say{convex combination} bound:

$$ \#  \left\{ \mathbf{b} \in \mathcal{Z}_j,  \ -\log g_{\mathbf{a}_{j-1}' \mathbf{b}}' (x_{\mathbf{a}_j}) \in \text{mod}_{2 i \pi} \left(B_\infty(a,\sigma) \right) \right\} \leq \sum_{\mathbf{b} \in \mathcal{Z}_j} h(-\log g_{\mathbf{a}_{j-1}' \mathbf{b}}' (x_{\mathbf{a}_j}) ) $$
$$ =  \sum_{\mathbf{b} \in \mathcal{Z}_j} \frac{w_{ \mathbf{b}}(x_{\mathbf{a}_j})}{w_{ \mathbf{b}}(x_{\mathbf{a}_j})}  h(-\log g_{\mathbf{a}_{j-1}' \mathbf{b}}' (x_{\mathbf{a}_j}) ) $$
$$ \leq R N{\sum_{\mathbf{b} \in \mathcal{Z}_{j}}} w_{ \mathbf{b}}(x_{\mathbf{a}_j}) h(-\log g_{\mathbf{a}_{j-1}' \mathbf{b}}' (x_{\mathbf{a}_j}) ) $$
$$ \leq R N \underset{\mathbf{a}_{j-1} \rightsquigarrow \mathbf{b} \rightsquigarrow \mathbf{a}_j}{\sum_{\mathbf{b} \in \mathcal{W}_{n+1}}} w_{ \mathbf{b}}(x_{\mathbf{a}_j}) h(-\log g_{\mathbf{a}_{j-1}' \mathbf{b}}' (x_{\mathbf{a}_j}) ) .$$
Then, since our map $h$ is $2 \pi \mathbb{Z}[i]$-periodic and smooth, we can develop it using Fourier series. We can write:
$$ \forall z=x+iy \in \mathbb{C}/2 i \pi \mathbb{Z}, \ h(z) = \sum_{(\mu,\nu) \in \mathbb{Z}^2 } c_{\mu \nu}(h) e^{i \left( \mu x + \nu y \right)} ,$$

where $$c_{\mu \nu}(h) = (4 \pi^2)^{-1} \int_{B_\infty(a,\pi)} h(x + iy) e^{-i \left( \mu x + \nu y \right)} dx dy .$$ 
Notice that $$ \mu^{k_1} \nu^{k_2} |c_{\mu \nu}(h)| \simeq |c_{\mu \nu}( \partial_x^{k_1} \partial_y^{k_2} h)| $$ $$\leq (4 \pi^2)^{-1} \|\partial_x^{k_1} \partial_y^{k_2} h\|_{L^1(B_\infty(a,\pi))} = (4 \pi^2)^{-1} \|\partial_x^{k_1} \partial_y^{k_2} \chi\|_{L^1(\mathbb{C})} \simeq \sigma^{2-k_1 - k_2} .$$

Plugging $ -\log g_{\mathbf{a}_{j-1}' \mathbf{b}}' (x_{\mathbf{a}_j}) $ in this expression yields

$$ h\left( -\log g_{\mathbf{a}_{j-1}' \mathbf{b}}' (x_{\mathbf{a}_j}) \right) = \sum_{(\mu,\nu) \in \mathbb{Z}^2 } c_{\mu \nu}(h) \exp\left( - i \mu \ln( | g_{\mathbf{a}_{j-1}' \mathbf{b}}' (x_{\mathbf{a}_j}) |  ) - i \nu \arg   g_{\mathbf{a}_{j-1}' \mathbf{b}}' (x_{\mathbf{a}_j})  \right) $$
$$ =  \sum_{(\mu,\nu) \in \mathbb{Z}^2 } c_{\mu \nu}(h) | g_{\mathbf{a}_{j-1}' \mathbf{b}}' (x_{\mathbf{a}_j})|^{- i \mu} \left( \frac{ g_{\mathbf{a}_{j-1}' \mathbf{b}}' (x_{\mathbf{a}_j})}{| g_{\mathbf{a}_{j-1}' \mathbf{b}}' (x_{\mathbf{a}_j})|} \right)^{- \nu} , $$

and so
$$  \#  \left\{ \mathbf{b} \in \mathcal{Z}_j,  \ -\log g_{\mathbf{a}_{j-1}' \mathbf{b}}' (x_{\mathbf{a}_j}) \in \text{mod}_{2i \pi}\left(B_\infty(a,\sigma) \right) \right\} $$
$$ \leq R N  \sum_{\mu \nu} c_{\mu \nu}(h) \underset{\mathbf{a}_{j-1} \rightsquigarrow \mathbf{b} \rightsquigarrow \mathbf{a}_j}{\sum_{\mathbf{b} \in \mathcal{W}_{n+1}}} w_{ \mathbf{b}}(x_{\mathbf{a}_j}) | g_{\mathbf{a}_{j-1}' \mathbf{b}}' (x_{\mathbf{a}_j})|^{- i \mu} \left( \frac{ g_{\mathbf{a}_{j-1}' \mathbf{b}}' (x_{\mathbf{a}_j})}{| g_{\mathbf{a}_{j-1}' \mathbf{b}}' (x_{\mathbf{a}_j})|} \right)^{- \nu}  .$$

For any word $\mathbf{a}$, define $\mathfrak g_{\mathbf{a}}'$ on $C_b^1(\mathfrak{U},\mathbb{C})$ by
$$ \forall x \in U_{b(\mathbf{a})}, \ \mathfrak g_{\mathbf{a}}'(x) := g_{\mathbf{a}}'(x) \quad , \quad \forall x \in U_b , \ b \neq b(\mathbf{a}), \ \mathfrak g_{\mathbf{a}}'(x) := 0 .$$
With this notation, we may rewrite the sum on $\mathbf{b}$ as follows:
$$  \underset{\mathbf{a}_{j-1} \rightsquigarrow \mathbf{b} \rightsquigarrow \mathbf{a}_j}{\sum_{\mathbf{b} \in \mathcal{W}_{n+1}}} w_{ \mathbf{b}}(x_{\mathbf{a}_j}) | g_{\mathbf{a}_{j-1}' \mathbf{b}}' (x_{\mathbf{a}_j})|^{ - i \mu} \left( \frac{ g_{\mathbf{a}_{j-1}' \mathbf{b}}' (x_{\mathbf{a}_j})}{| g_{\mathbf{a}_{j-1}' \mathbf{b}}' (x_{\mathbf{a}_j})|} \right)^{- \nu} $$
$$ = \underset{\mathbf{a}_{j-1} \rightsquigarrow \mathbf{b} \rightsquigarrow \mathbf{a}_j}{\sum_{\mathbf{b} \in \mathcal{W}_{n+1}}}  | g_{\mathbf{a}_{j-1}}' (g_{ \mathbf{b}}(x_{\mathbf{a}_{j}}))|^{- i \mu} \left( \frac{ g_{\mathbf{a}_{j-1} }' (g_{ \mathbf{b}}(x_{\mathbf{a}_{j}} ))} {| g_{\mathbf{a}_{j-1}}' (g_{ \mathbf{b}}(x_{\mathbf{a}_{j}}))|} \right)^{- \nu}  w_{\mathbf{b}}(x_{\mathbf{a}_j}) |g_{\mathbf{b}}' (x_{\mathbf{a}_j})|^{- i \mu} \left( \frac{ g_{ \mathbf{b}}' (x_{\mathbf{a}_j})}{| g_{ \mathbf{b}}' (x_{\mathbf{a}_j})|} \right)^{- \nu} $$
$$ = \underset{\mathbf{b} \rightsquigarrow \mathbf{a}_{j-1} }{\sum_{\mathbf{b} \in \mathcal{W}_{n+1}}} w_{\mathbf{b}}(x_{\mathbf{a}_j}) | g_{\mathbf{b}}' (x_{\mathbf{a}_j})|^{- i \mu} \left( \frac{ g_{ \mathbf{b}}' (x_{\mathbf{a}_j})}{| g_{ \mathbf{b}}' (x_{\mathbf{a}_j})|} \right)^{- \nu}  \left( |\mathfrak g_{\mathbf{a}_{j-1}}'|^{-i \mu}  \left(\frac{\mathfrak g_{\mathbf{a}_{j-1}}'}{|\mathfrak g_{\mathbf{a}_{j-1}}'|} \right)^{- \nu} \right) \left( g_{\mathbf{b}}(x_{\mathbf{a}_j}) \right) $$
$$ = \mathcal{L}_{\varphi,-i \mu,\nu}^n \left(  |\mathfrak g_{\mathbf{a}_{j-1}}'|^{-i \mu}  \left(\frac{\mathfrak g_{\mathbf{a}_{j-1}}'}{|\mathfrak g_{\mathbf{a}_{j-1}}'|} \right)^{- \nu} \right)(x_{\mathbf{a}_j}). $$

For clarity, set $\mathfrak{h}_{\mathbf{A},j} :=  |\mathfrak g_{\mathbf{a}_{j-1}}'|^{-i \mu}  \left(\frac{\mathfrak g_{\mathbf{a}_{j-1}}'}{|\mathfrak g_{\mathbf{a}_{j-1}}'|} \right)^{- \nu}  $.
A direct computation, and the holomorphicity of the $(g_{\mathbf{a}})_\mathbf{a}$ allows us to see that
$$ \| \mathfrak{h}_{\mathbf{A},j} \|_{C_b^1(\mathfrak{U},\mathbb{C})} \lesssim (1+|\mu|+|\nu|). $$

We can now break the estimate into two pieces: high frequencies are controlled by the contraction property of this transfer operator, and the low frequencies are controlled by the Gibbs property of $\mu$. We also use the estimates on the Fourier coefficients on $h$. 

$$ \#  \left\{ \mathbf{b} \in \mathcal{Z}_j,  \ -\log g_{\mathbf{a}_{j-1}' \mathbf{b}}' (x_{\mathbf{a}_j}) \in \text{mod}_{2 i \pi} \left(B_\infty(a,\sigma) \right) \right\}  $$
$$ \leq R N \sum_{\mu \nu} c_{\mu \nu}(h)\underset{\mathbf{a}_{j-1} \rightsquigarrow \mathbf{b} \rightsquigarrow \mathbf{a}_j}{\sum_{\mathbf{b} \in \mathcal{W}_{n+1}}}  w_{\mathbf{b}}(x_{\mathbf{a}_j}) | g_{\mathbf{a}_{j-1}' \mathbf{b}}' (x_{\mathbf{a}_j})|^{ - i \mu} \left( \frac{ g_{\mathbf{a}_{j-1}' \mathbf{b}}' (x_{\mathbf{a}_j})}{| g_{\mathbf{a}_{j-1}' \mathbf{b}}' (x_{\mathbf{a}_j})|} \right)^{- \nu} $$
$$ \leq R N \left( \sum_{|\mu| + |\nu| \leq 1} |c_{\mu \nu}(h)|  \sum_{\mathbf{b} \in \mathcal{W}_{n+1}}  w_{\mathbf{b}}(x_{\mathbf{a}_j})  +  \sum_{|\mu| + |\nu| > 1} |c_{\mu \nu}(h)| |\mathcal{L}_{\varphi,-i \mu,\nu}^n (\mathfrak{h}_{\mathbf{A},j})(x_{\mathbf{a}_j})| \right) $$
$$ \lesssim R N \left( 5 \sigma^2 \sum_{\mathbf{b} \in \mathcal{W}_{n+1}} \mu(P_{\mathbf{b}})  + \sum_{|\mu| + |\nu| > 1} |c_{\mu \nu}(h)| \|\mathcal{L}_{\varphi,-i \mu,\nu}^n (\mathfrak{h}_{\mathbf{A},j})\|_{C_b^1(\mathfrak{U},\mathbb{C})} \right) $$
$$ \lesssim  R N \sigma^2 +R N  \sum_{|\mu| + |\nu| > 1} |c_{\mu \nu}(h)| (|\mu|+|\nu|)^2 \rho^n \| \mathfrak{h}_{\mathbf{A},j} \|_{C^1_b(\mathfrak{U},\mathbb{C})} $$

$$ \lesssim  R N \sigma^2 + R N \rho^n \sum_{|\mu| + |\nu| > 1} |c_{\mu \nu}(h)| (|\mu|+|\nu|)^5 (|\mu|+|\nu|)^{-2}    $$
$$ \leq C R N (\sigma^2 + \rho^n \sigma^{-3}) , $$
for some constant $C>0$.
We are nearly done. Since $\sigma \in  [ e^{-6 \varepsilon_0 n} , e^{ - \varepsilon_1 \varepsilon_0 n/6 }  ]$, we know that $ \sigma^{-3} \leq e^{18 \varepsilon_0 n } $.
Now is the time where we fix $\varepsilon_0$: choose $$ \varepsilon_0 := \min( - \ln(\rho)/30,  \lambda/2  ). $$

Then $ \rho^n \sigma^{-3} \leq  e^{ n(\ln(\rho) + 18 \varepsilon_0 )} \leq  e^{ - 12 \varepsilon_0 n} \leq \sigma^2 $ for $n$ large enough. Hence, we get

$$  \#  \left\{ \mathbf{b} \in \mathcal{Z}_j,  \ -\log g_{\mathbf{a}_{j-1}' \mathbf{b}}' (x_{\mathbf{a}_j}) \in \text{mod}_{2 i \pi} \left(B_\infty(a,\sigma) \right) \right\} \leq 2 C R N \sigma^{2} . $$

Finally, since $\sigma^{1/2}$ is quickly decaying compared to $R$, we have

$$ 2 C R N \sigma^{2} \leq N \sigma^{3/2} $$

provided $n$ is large enough. The proof is done.
\end{proof}

\begin{appendices}

\section{Large deviations.}

The goal of this section is to prove the large deviation Theorem 2.8, by using properties of the pressure. \\

The link between the spectral radius of $\mathcal{L}_\varphi$ and the pressure given by the Perron-Frobenius-Ruelle theorem allows us to get the following useful formula. We extract the first one from $\cite{Ru78}$, Theorem 7.20 and remark 7.28, and the second one from $\cite{Ru89}$, lemma 4.5.

\begin{proposition}

 $$ P(\varphi) = \lim_{n \rightarrow \infty} \frac{1}{n} \log \sum_{f^n(x)=x} e^{S_n \varphi(x)}$$ $$\quad P(\varphi) = \lim_{n \rightarrow \infty} \frac{1}{n} \max_{b \in \mathcal{A}} \sup_{x \in P_b} \log \underset{\mathbf{a} \rightsquigarrow b}{\sum_{\textbf{a} \in \mathcal{W}_{n+1}} }  e^{S_n \varphi( g_{\textbf{a}}(x) )}  $$

\end{proposition}

We begin by proving another avatar of those spectral radius formulas (which is nothing new).

\begin{lemma}

Choose any $x_{\textbf{a}}$ in each of the $P_{\textbf{a}}$, $\textbf{a} \in \mathcal{W}_{n}$, $\forall n$.
Then
$$ P(\varphi) = \lim_n \frac{1}{n} \log \sum_{\textbf{a} \in \mathcal{W}_{n+1}} e^{S_n \varphi(x_\textbf{a})} .$$

\end{lemma}
\begin{proof}

Since $P_{\textbf{a}}$ is compact, and by continuity, for every $n$ there exists $b^{(n)} \in \mathcal{A}$ and $y_{b^{(n)}}^{(n)} \in P_{b^{(n)}}$ such that $$ \max_{b \in \mathcal{A}} \sup_{x \in P_b} \log \underset{\mathbf{a} \rightsquigarrow b}{\sum_{\textbf{a} \in \mathcal{W}_{n+1}} }  e^{S_n \varphi( g_{\textbf{a}}(x) )} = \log \underset{\mathbf{a} \rightsquigarrow b^{(n)}}{\sum_{\textbf{a} \in \mathcal{W}_{n+1}} } e^{S_n \varphi( g_{\textbf{a}}(y_{b^{(n)}}^{(n)}) )}. $$  
Define $y_{\textbf{a}} := g_{\textbf{a}}(y_{b^{(n)}}^{(n)}) \in P_{\textbf{a}}$ for clarity. The dependence on $n$ is not lost since it is contained in the length of the word. First of all, since $\varphi$ has exponentially vanishing variations, there exists a constant $C_1>0$ such that
$$ \forall x,y \in P_{\textbf{a}}, \ | S_n \varphi(x) - S_n \varphi(y) |\leq C_1 .$$

Now we want to relate the sums with the $x_\mathbf{a}$'s and the $y_{\mathbf{a}}$'s, but the indices are different. To do it properly, we are going to use the fact that $f$ is topologically mixing: there exists some $N \in \mathbb{N}$ such that the matrix $M^N$ has all its entries positive. In particular, it means that
$$ \forall b \in \mathcal{A}, \ \forall \textbf{a} \in \mathcal{W}_{n+1}, \ \exists \textbf{c} \in \mathcal{W}_{N}, \ \textbf{ac}b \in \mathcal{W}_{n+N+1} .$$
The point is that we are sure that the word is admissible. \\
For a given $\textbf{a} \in \mathcal{W}_{n+1}$, there exists a $\textbf{c} \in \mathcal{W}_{N}$ such that $ \textbf{ac}b^{(n+N+1)} \in \mathcal{W}_{n+N+1}$, and so, using the fact that $e^{S_n \varphi} \geq 0$, we get:

$$ e^{S_n \varphi(x_{\textbf{a}})} \leq  \underset{ \textbf{ac}b^{(n+N+1)} \in \mathcal{W}_{n+N+1}}{\sum_{\textbf{c} \in \mathcal{W}_{N} }} e^{C_1} e^{S_n \varphi(y_{\textbf{ac}b^{(n+N+1)}})} .$$
Then, since $S_{n}(\varphi) \leq S_{n+N}(\varphi) +N \|\varphi\|_{\infty,J}$, we have:

$$ e^{S_n \varphi(x_{\textbf{a}})} \leq  \underset{ \textbf{ac}b^{(n+N+1)} \in \mathcal{W}_{n+N+1}}{\sum_{\textbf{c} \in \mathcal{W}_{N} }} e^{C_2} e^{S_{n+N}(\varphi)(y_{\textbf{ac}b^{(n+N+1)}})} .$$
Hence
$$ \log \left( {\sum_{\textbf{a} \in \mathcal{W}_{n+1}} }  e^{S_n \varphi( x_{\textbf{a}} )} \right) \leq \log \left( \sum_{\textbf{a} \in \mathcal{W}_{n+1}} \underset{ \textbf{ac}b^{(n+N+1)} \in \mathcal{W}_{n+N+1}}{\sum_{\textbf{c} \in \mathcal{W}_{N} }} e^{C_2} e^{S_{n+N} \varphi(y_{\textbf{ac}b^{(n+N+1)}})}  \right) $$
$$ =C_2 + \log \left( \underset{\mathbf{d} \rightsquigarrow b^{(n+N+1)}}{\sum_{\textbf{d} \in \mathcal{W}_{n+N+1} }} e^{S_n \varphi(y_{\textbf{d}})} \right) ,$$
and so $$ \limsup_{n \rightarrow \infty} \frac{1}{n} \log \left( {\sum_{\textbf{a} \in \mathcal{W}_{n+1}} }  e^{S_n \varphi( x_{\textbf{a}} )}  \right) \leq P(\varphi) .$$

The other inequality is easier, we have

$$ \log \left( \underset{ \mathbf{a} \rightsquigarrow b^{(n)}}{\sum_{\textbf{a} \in \mathcal{W}_{n+1}} } e^{S_n \varphi( y_{\textbf{a}}  )} \right) \leq C_1 + \log \left( \sum_{\textbf{a} \in \mathcal{W}_{n+1}} e^{S_n \varphi(x_{\textbf{a}})} \right), $$
which gives us $$ P(\varphi) \leq \liminf \frac{1}{n} \log \left( \sum_{\textbf{a} \in \mathcal{W}_{n+1}} e^{S_n \varphi(x_{\textbf{a}})} \right) . $$ \end{proof}

Another useful formula is the computation of the differential of the pressure.

\begin{theorem}

The map $P : C^1(U,\mathbb{R}) \rightarrow \mathbb{R} $ is differentiable. 
If $\varphi \in C^1(U,\mathbb{R})$ is a normalized potential, then we have:

$$ \forall \psi \in C^1(U,\mathbb{R}), \  (dP)_{\varphi}(\psi) = \int_J \psi d\mu_{\varphi} .$$

\end{theorem}

\begin{proof}
This is the corollary 5.2 in $\cite{Ru89}$. Loosely, the argument goes as follows. \\

The differentiability is essentially a consequence of the fact that $e^{P(\psi)}$ is an isolated eigenvalue of $\mathcal{L}_\psi$. To compute the differential, consider $v_t \in C^1$ the normalized eigenfunction for $\mathcal{L}_{\varphi + t \psi}$ such that $v_0=1$. We have, for small $t$:

$$ \mathcal{L}_{\varphi + t \psi} v_t = e^{P(\varphi + t \psi)} v_t $$
Hence, $$ \mathcal{L}_{\varphi + t \psi} (\psi v_t ) + \mathcal{L}_{\varphi + t \psi} (\partial_t v_t) = v_t e^{P(\varphi + t \psi)} \frac{d}{dt} P(\varphi + t \psi) + e^{P(\varphi + t \psi)} \partial_t v_t $$
Taking $t=0$ and integrating against $\mu_\varphi$ gives
$$ (d P)_{\varphi}(\psi) = \int_J \mathcal{L}_{\varphi}(\psi) d\mu_\varphi = \int_J \psi d\mu_\varphi.$$ \end{proof}

Now, we are ready to prove \textbf{Theorem 2.8}. The proof is adapted from \cite{JS16}, subsection 4.

\begin{proof}

Let $\varphi$ be a normalized potential, and let $\psi$ be another $C^1$ potential.
Let $\varepsilon > 0$. Let $j(t) := P\left( (\psi - \int \psi d\mu_\varphi - \varepsilon)t + \varphi \right)$. We know by Theorem A.3 that $j'(0) = -\varepsilon < 0$. Hence, there exists $t_0> 0$ such that $P( (\psi - \int \psi d\mu_\varphi - \varepsilon)t_0 + \varphi ) < 0$. \\

Define $2 \delta_0 := -P\left( (\psi - \int \psi d\mu_\varphi - \varepsilon)t_0 + \varphi \right)   $. We then have
$$\mu_\varphi\left( \left\{ x \in J \ , \ \frac{1}{n} S_n \psi(x) - \int_J \psi d\mu_{\varphi}  \geq \varepsilon \right\} \right) \leq \sum_{\textbf{a} \in C_{n+1}} \mu_\varphi (P_\mathbf{a}) ,$$

where $C_{n+1} := \{ \mathbf{a} \in \mathcal{W}_{n+1} \ | \ \exists x \in P_{\mathbf{a}} , \ S_n \psi(x)/n - \int \psi d\mu_{\varphi}  \geq \varepsilon \} $.
For each $\mathbf{a}$ in some $C_{n+1}$, choose $x_{\mathbf{a} } \in P_{\mathbf{a} }$ such that $S_n \psi(x_{\mathbf{a} })/n - \int \psi d\mu_{\varphi}  \geq \varepsilon$. For the other $\mathbf{a}$, choose $x_{\mathbf{a}} \in P_{\textbf{a}}$ randomly. \\
Now, since $\mu_\varphi$ is a Gibbs measure, there exists $C_0>0$ such that:

$$ \sum_{\textbf{a} \in C_{n+1}} \mu_\varphi (P_\mathbf{a})  \leq C_0 \sum_{\textbf{a} \in C_{n+1}} \exp(S_n \varphi(x_\textbf{a})) $$
$$ \leq C_0 \sum_{\textbf{a} \in C_{n+1}} \exp\left({S_n\left( \left(\psi-\int \psi d\mu_\varphi - \varepsilon\right)t_0 + \varphi \right)(x_\textbf{a} )} \right) $$
$$ \leq C_0 \sum_{\textbf{a} \in \mathcal{W}_{n+1}} \exp\left({S_n\left( \left(\psi-\int \psi d\mu_\varphi - \varepsilon\right)t_0 + \varphi \right)(x_\textbf{a} )} \right) .$$

Then, by the lemma A.2, we can write for $n \geq n_0$ large enough:
$$ C_0 \sum_{\textbf{a} \in \mathcal{W}_{n+1}} \exp\left({S_n\left( \left(\psi-\int \psi d\mu_\varphi - \varepsilon\right)t_0 + \varphi \right)(x_\textbf{a} )} \right) \leq C e^{n \delta_0} e^{n P( (\psi - \int \psi d\mu_\varphi - \varepsilon)t_0 + \varphi )} \leq C e^{-  \delta_0 n} ,$$

and so $$ \mu_\varphi\left( \left\{ x \in J \ , \ \frac{1}{n} S_n \psi(x) - \int_X \psi d\mu_{\varphi}  \geq \varepsilon \right\} \right) \leq Ce^{- n \delta_0} .$$

The symmetric case is done by replacing $\psi$ by $-\psi$, and combining the two gives us the desired bound. \end{proof}

\section{Uniform spectral estimate for a family of twisted transfer operator}

Here, we will show how to prove \textbf{Theorem 6.4}. It is a generalization of Theorem 2.5 in $\cite{OW17}$: we will explain what we need to change in the original paper for the theorem to hold more generally. \\

Proving that a complex transfer operator is eventually contracting is linked to analytic extensions results for dynamical zeta functions, and is often referred to as a spectral gap. Such results are of great interest to study, for example, periodic orbit distribution in hyperbolic dynamical systems (see for example the chapter 5 and 6 in \cite{PP90}), or asymptotics for dynamically defined quantities (as in \cite{OW17} or \cite{PU17}). One of the first result of this kind can be found in a work of Dolgopyat \cite{Do98}, in which he used a method that has been broadly extended since. We can find various versions of Dolgopyat's method in papers of Naud $\cite{Na05}$, Stoyanov \cite{St11}, Petkov \cite{PS16}, Oh-Winter \cite{OW17}, Li \cite{Li18b}, and Sharp-Stylianou \cite{ShSt20}, to only name a few. \\

In this annex, we will outline the argument of Dolgopyat's method as explained in $\cite{OW17}$ adapted to our general setting. We need three ingredients to make the method work: the NLI (non local integrability), the NCP (another non concentration property), and a doubling property.

\begin{definition}

Define $\tau(x) := \log |f'(x)| \in \mathbb{R}$ and $\theta(x) := \arg f'(x) \in \mathbb{R}/2 \pi \mathbb{Z}.$ The transfer operator in Theorem $6.4$ acts on $C^1_b(\mathfrak{U},\mathbb{C})$ and may be rewritten in the form $$ \mathcal{L}_{\varphi,it,l} = \mathcal{L}_{\varphi - i t \mathcal{\tau} - i l \mathcal{\theta}} .$$

For some normalized $\varphi \in C^1(U,\mathbb{R})$, $l \in \mathbb{Z}$ and $t \in \mathbb{R}$.
\end{definition}

With those notations, Theorem 6.4 can be rewritten as follows.

\begin{theorem}
Suppose that $J$ is not included in a circle.
For any $\varepsilon>0$, there exists $C>0$, $\rho<1$ such that for any $n \geq 1$ and any $t\in \mathbb{R}$, $l \in \mathbb{Z}$ such that $|t|+|l| > 1$,

$$ \| \mathcal{L}_{\varphi - i(t \tau + l \theta)}^n \|_{C^1_b(\mathfrak{U},\mathbb{C})} \leq C (|l|+|t|)^{1+\varepsilon} \rho^n $$

\end{theorem}

Now we may recall the three main technical ingredients.

\begin{theorem}[NLI, \cite{OW17} section 3]

The function $\tau$ satisfies the NLI property if there exists $a_0 \in \mathcal{A}$,  $x_1 \in P_{a_0}$ , $N \in \mathbb{N}$, admissible words $\mathbf{a}, \mathbf{b} \in \mathcal{W}_{N+1}$ with $a_0\rightsquigarrow \mathbf{a},\mathbf{b}$, and an open neighborhood $U_0$ of $x_1$ such that for any $n \geq N$, the map $$ (\tilde{\tau}, \tilde{\theta}) := (S_n \tau \circ g_\mathbf{a} − S_n \tau \circ g_\mathbf{b}, S_n\theta \circ g_\mathbf{a} − S_n\theta \circ g_\mathbf{b}) : U_0 \rightarrow \mathbb{R} \times \mathbb{R}/2 \pi \mathbb{Z}$$
is a local diffeomorphism.
    
\end{theorem}

\begin{remark}
Remark 4.7 in \cite{SS20} and Proposition 3.8 in \cite{OW17} points out the fact that the NLI is a consequence of our non-linear setting, which itself comes from the fact that we supposed that our Julia set is different from a circle. 
\end{remark}

\begin{theorem}[NCP, \cite{OW17} section 4]

For each $n \in \mathbb{N}$, for any $\mathbf{a} \in \mathcal{W}_{n+1}$, there
exists $0 < \delta < 1$ such that, for all $x \in P_\mathbf{a}$, all $w \in \mathbb{C}$ of unit length, and all $\varepsilon \in (0, 1) $,
$$B(x,\varepsilon) \cap \{ y \in P_\mathbf{a}, \ |\langle y − x, w \rangle| > \delta \varepsilon \} \neq \emptyset $$
where $\langle a + bi, c + di \rangle = ac + bd$ for $a, b, c, d \in \mathbb{R}$.
\end{theorem}

\begin{remark}
The NCP is a consequence of the fractal behavior of our Julia set. This time, if $J$ is included in any smooth set, the NCP fails. But in our case, this is equivalent to being included in a circle, see \cite{ES11}. Notice that this non concentration property has nothing to do with our previous non concentration hypothesis.
\end{remark}

\begin{theorem}[Doubling]

Let, for $a \in \mathcal{A}$, $\mu_a$ be the equilibrium measure $\mu_\varphi$ restricted to $P_a$. Then each $\mu_a$ is doubling, that is:

$$ \exists C>0, \ \forall x \in P_a, \ \forall r<1, \  \mu_a(B(x,2r)) \leq C \mu_a(B(x,r)). $$

\end{theorem}

\begin{proof}

It follows from Theorem A.2 in \cite{PW97} that $\mu_\varphi$ is doubling: the proof uses the conformality of the dynamics. To prove that $\mu_a := \mu_{| P_a}$ is still doubling, which is not clear a priori, we follow Proposition 4.5 in \cite{OW17} and prove that there exists $c>0$ such that for any $a \in \mathcal{A}$, for any $x \in P_a$, and for any $r>0$ small enough, $$ \frac{\mu_\varphi(B(x,r) \cap P_a)}{\mu_\varphi(B(x,r))} > c .$$

For this we use a Moran cover $\mathcal{P}_r$ associated to our Markov partition, see the proof of Proposition 2.7 for a definition. Recall that any element $P \in \mathcal{P}_r$ have diameter strictly less than $r$, and recall that there exists a constant $M>0$ independent of $x$ and $r$ such that we can cover the ball $B(x,r)$ with $M$ elements of $\mathcal{P}_r$. Moreover, lemma 2.2 in \cite{WW17} allows us to do so using elements $P\in \mathcal{P}_r$ of the form $P_\mathbf{a}$ for $\mathbf{a}$ in some $\mathcal{W}_n$, $N_0 \leq n \leq N_0 + L$ for some $N_0(x,r)$ and some constant $L$ (independent of $x$ and $r$). We can then conclude as follows. Let $P^{(1)} , \dots P^{(M)} \in \mathcal{P}_r$ that covers $B(x,r)$. There exists $i$ such that $P^{(i)} \subset P_a$. Hence, by the Gibbs property of $\mu_\varphi$:

$$ \frac{\mu_\varphi(B(x,r) \cap P_a)}{\mu_\varphi(B(x,r))} \geq  \frac{\mu_\varphi(P^{(i)})}{ {\sum_{j=1}^M} \mu_\varphi(P^{(j)})} \geq M^{-1} C_0^{-2} e^{-(L+2)\| \varphi \|_\infty}.$$
 \end{proof}

\begin{remark}

The doubling property (or Federer property) is a regularity assumption made on the measure that is central for the execution of this version of Dolgopyat's method. It allows us to control integrals over $J$ by integrals over smaller pieces of $J$, provided some regularity assumption on the integrand. 

\end{remark}

Now we will outline the argument of Dolgopyat's method as used in $\cite{OW17}$. It can be decomposed into four main steps. \\

\underline{\textbf{Step 1:}} We reduce Theorem B.1 to a $L^2(\mu)$ estimate. \\
\begin{quote}
We need to define a modified $C^1$ norm. Denote by $\| \cdot \|_r$ a new norm, defined by
$$ \|h\|_r := \left \{
\begin{array}{rcl}
\|h\|_{\infty,\mathfrak{U}} + \frac{\|\nabla h\|_{\infty,\mathfrak{U}}}{r}  \ \text{if} \ r \geq 1  \\
\|h\|_{\infty,\mathfrak{U}} + \| \nabla h \|_{\infty,\mathfrak{U}} \ \text{if} \ r < 1 
\end{array}
\right.$$
Moreover, we do a slight abuse of notation and write $\mu$ for $\sum_{a \in \mathcal{A}} \mu_a$, seen as a measure on $\mathfrak{U}$. This measure is supported on $J$, seen as the set $\bigsqcup_a P_a \subset \bigsqcup_a U_a = \mathfrak{U}$.
The first step is to show that Theorem B.1 reduces to the following claim.

\begin{theorem}

Suppose that the Julia set of $f$ is not contained in a circle. Then there exists $C>0$ and $ \rho \in (0,1)$ such that for any $h \in C^1_b(\mathfrak{U},\mathbb{C})$
and any $n \in \mathbb{N}$,
$$ ||\mathcal{L}_{\varphi-i(t \tau + l \theta)}^n h||_{L^2(\mu)} \leq C \rho^n \|h \|_{|t|+|l|} $$
for all $t \in \mathbb{R}$ and $ℓ \in \mathbb{Z}$ with $|t| + |l| \geq 1$.
\end{theorem}

A clear account for this reduction may be found in \cite{Na05}, section 5. This step holds in great generality without any major difficulty. Intuitively, Theorem B.1 follows from Theorem B.5 by the Lasota-Yorke inequality, by the quasicompactness of $\mathcal{L}_\varphi$, and by the Perron-Frobenius-Ruelle theorem, which implies that $\mathcal{L}^N_\varphi h$ is comparable to $\int h d\mu$ for $N$ large. The difference between the two can be controlled using $C^1$ bounds. \\
\end{quote}

\underline{\textbf{Step 2:}} We show that the oscillations in the sum induce enough cancellations.  \\
\begin{quote}

Loosely, the argument goes as follows. 
We write, for a well chosen and large $N$:

$$ \forall x \in U_b, \ \mathcal{L}^N_{\varphi - i (t \tau + l \theta)}h(x) = \underset{\mathbf{a}  \rightsquigarrow b} {\sum_{\mathbf{a} \in \mathcal{W}_{N+1}} } e^{ i (t S_N \tau + l S_N \theta )(g_\mathbf{a}(x))} h(g_\mathbf{a}(x))e^{S_N \varphi (g_\mathbf{a}(x) ) } .$$

If we choose $x$ in a suitable open set $\widehat{S} \subset U_0$, the NLI and the NCP tell us that we can extract words $\mathbf{a}$ and $\mathbf{b}$ from this sum such that some cancellations happen. Indeed, if we isolate the term given by the words from the NLI,
$$  e^{ i (t S_N \tau + l S_N \theta )(g_\mathbf{a}(x))} h(g_\mathbf{a}(x))e^{S_N \varphi (g_\mathbf{a}(x) ) } +  e^{ i (t S_N \tau + l S_N \theta )(g_\mathbf{b}(x))} h(g_\mathbf{b}(x))e^{S_N \varphi (g_\mathbf{b}(x) ) } ,$$
we see that a difference in argument might give us some cancellations. The effect of $h$ in the difference of argument can be carefully controlled by the $C^1$ norm of $h$. The interesting part comes from the complex exponential. The difference of arguments of this part is
$$ t (S_N \tau \circ g_\mathbf{a} -  S_N \tau \circ g_\mathbf{b}) + l (S_N \theta \circ g_{\mathbf{a}} - S_N \theta \circ g_{\mathbf{b}}) ,$$
which might be rewriten in the form
$$ \langle (t,l) , (\tilde{\tau},\tilde{\theta}) \rangle .$$
Then we proceed as follows. Choose a large number of points $(x_k)$ in $U_0$. If, for a given $x_k$, the difference of argument $\langle (t,l),(\tilde{\tau},\tilde{\theta}) \rangle (x_k)$ is not large enough, we might use the NCP to construct another point $y_k$ next to $x_k$ such that $ \langle (t,l),(\tilde{\tau},\tilde{\theta}) \rangle (y_k) $ become larger. The construction goes as follows: the NLI ensures that $\nabla \langle (t,l) , (\tilde{\tau},\tilde{\theta}) \rangle (x_k) =: w_k \neq 0 $. Hence, the direction $\widehat{w}_k := \frac{w_k}{|w_k|}$ is well defined. The NCP then ensures us the existence of some $y_k \in J$ which are very close to $x_k$ and such that $x_k - y_k$ is a vector pointing in a direction comparable to $\widehat{w}_k$. As we are following the gradient of $\langle (t,l) , (\tilde{\tau},\tilde{\theta}) \rangle$, we are sure that $\langle (t,l) , (\tilde{\tau},\tilde{\theta}) \rangle(y_k)$ will be larger than before. \\

We then let $S$ be the set containing the points where the difference in argument is large enough, so it contains some $x_k$ and some $y_k$. This large enough difference in argument that is true in $S$ is also true in a small open neighborhood $\widehat{S}$ of $S$. \\

We can then write, for $x \in \widehat{S}$, an inequality of the form: 
$$ \left|  e^{ i (t S_N \tau + l S_N \theta )(g_\mathbf{a}(x))} h(g_\mathbf{a}(x))e^{S_N \varphi (g_\mathbf{a}(x) ) } +  e^{ i (t S_N \tau + l S_N \theta )(g_\mathbf{b}(x))} h(g_\mathbf{b}(x))e^{S_N \varphi (g_\mathbf{b}(x) ) } \right| $$
$$ \leq (1-\eta) |h(g_\mathbf{a}(x))| e^{S_N \varphi(g_\mathbf{a}(x))} + |h(g_\mathbf{b}(x))| e^{S_N \varphi(g_\mathbf{b}(x))}  ,$$

where the $(1-\eta)$ comes in front of the part with the smaller modulus. This is, in spirit, lemma 5.2 of \cite{OW17}.
We can then summarize the information in the form of a function $\beta$ that is $1$ most of the time, but that is less than $(1-\eta)^{1/2}$ on $\widehat{S}$.
This allows us to write the following bound:
$$ \mathcal{L}_{\varphi-i(t \tau + l \theta)}^N h \leq  \mathcal{L}_{\varphi}^N \left( |h| \beta \right).$$

One of the main difficulty of this part is to make sure that $\widehat{S}$ is a set of large enough measure, while still managing not to make the $C^1$-norm of $\beta$ explode. All the hidden technicalities in this part forces us to only get this bound for \emph{a well chosen $N$}. 
\end{quote}

\underline{\textbf{Step 3:}} These cancellations allow us to compare $\mathcal{L}$ to an operator that is contracting on a cone. 

\begin{quote}
We define the following cone, on which the soon-to-be-defined Dolgopyat operator will be well behaved. Define
$$ K_R(\mathfrak{U}) := \{ H \in C^1_b(\mathfrak{U}) \ | \ H \text{ is positive, and} \ |\nabla H| \leq R H  \} ,$$

and then define the Dolgopyat operator by $\mathcal{M} H := \mathcal{L}_\varphi^N (H \beta)$. We then show that, if $H \in K_R(\mathfrak{U})$ (for a well chosen $R$):
\begin{enumerate}
    \item $\mathcal{M}(H) \in K_R(\mathfrak{U}) $
    \item $ \| \mathcal{M}(H) \|_{L^2(\mu)}^2 \leq (1-\varepsilon) \| H \|_{L^2(\mu)}^2$.
\end{enumerate}

The first point is done using the Lasota-Yorke inequalities, see lemma 5.1 in \cite{OW17}. The second point goes, loosely, as follows.

We write, using Cauchy-Schwartz:
$$ (\mathcal{M}H)^2 = \mathcal{L}_\varphi^N(H \beta)^2 \leq \mathcal{L}^N_\varphi(H^2)\mathcal{L}^N_\varphi(\beta^2)  . $$

On $\widehat{S}$, the cancellations represented in the function $\beta$ spread, thanks to the fact that $\varphi$ is normalized, as follows:

$$ \mathcal{L}_\varphi^N (\beta^2) = \sum_{\mathbf{a}} e^{ S_N \varphi \circ g_{\mathbf{a}} }   \beta^2 \circ g_\mathbf{a}  $$
$$ = \sum_{\mathbf{a} \text{ where } \beta=1 } e^{S_N \varphi \circ g_\mathbf{a} } \beta^2 \circ g_\mathbf{a} + \sum_{\mathbf{a} \text{ where } \beta \text{ is smaller} } e^{S_N \varphi \circ g_\mathbf{a} } \beta^2 \circ g_\mathbf{a} $$
$$ \leq \sum_{\mathbf{a} \text{ where } \beta=1 } e^{S_N \varphi \circ g_\mathbf{a} } + \sum_{\mathbf{a} \text{ where } \beta \text{ is smaller} } e^{S_N \varphi \circ g_\mathbf{a} } (1-\eta)  $$
$$ = 1 - \eta e^{-N \|\varphi\|_\infty}.$$

Then, we use the doubling property of $\mu=\sum_a \mu_a$ and the control given by the fact that $H \in K_R(\mathfrak{U})$ to bound the integral on all $J$ by the integral on $\widehat{S}$:

$$ \int_J  \mathcal{L}^N_\varphi(H^2) d\mu \leq C_0 \int_{\widehat{S}}  \mathcal{L}^N_\varphi(H^2) d\mu  .$$

Hence, we can write, using the fact that $\mathcal{L}_\varphi$ preserves $\mu$ and the previously mentioned Cauchy-Schwartz inequality: 
$$ \|H\|_{L^2(\mu)}^2 - \|\mathcal{M}H\|_{L^2(\mu)}^2 \geq \int_J \left(\mathcal{L}_\varphi^N(H^2) - \mathcal{L}_\varphi^N(H^2) \mathcal{L}_\varphi^N( \beta^2)  \right)d\mu $$
$$ \geq  \int_{\widehat{S}} \left(\mathcal{L}_\varphi^N(H^2) - \mathcal{L}_\varphi^N(H^2) \mathcal{L}_\varphi^N( \beta^2)  \right)d\mu $$
$$ \geq \eta e^{-N \|\varphi\|_\infty} \int_{\widehat{S}}\mathcal{L}_\varphi^N(H^2) d\mu  $$
$$ \geq \eta C_0^{-1} e^{-N \|\varphi\|_\infty} \| H \|_{L^2(\mu)}^2 := \varepsilon \| H \|_{L^2(\mu)}^2 .$$

Hence $$  \|\mathcal{M}H\|_{L^2(\mu)}^2 \leq (1-\varepsilon)  \|H\|_{L^2(\mu)}^2 . $$

\end{quote}
\underline{\textbf{Step 4:}} We conclude by an iterative argument. \\
\begin{quote}

To conclude, we need to see that we may bound $h$ by some $H \in K_R(\mathfrak{U})$, and also that the contraction property is true for all $n$, not just $N$.

For any $n=kN$, we can inductively prove our bound. If $k=1$, we can choose $H_0:= \|h\|_{|t|+|l|}$. Then, $|h| \leq H_0$ and so
$$ \| \mathcal{L}_{\varphi - i(t \tau + l \theta) }^N h \|_{L^2(\mu)} \leq \| \mathcal{M} H_0 \|_{L^2(\mu)} \leq (1-\varepsilon)^{1/2} \|h\|_{(|t|+|l|)} .$$

Then, choosing $H_{k+1} := \mathcal{M} H_k \in K_R(\mathfrak{U})$, we can proceed to the next step of the induction and get
$$  \| \mathcal{L}_{\varphi - i(t \tau + l \theta) }^{kN} h \|_{L^2(\mu)} \leq \| \mathcal{M} H_{k-1} \|_{L^2(\mu)} \leq (1-\varepsilon)^{k/2} \|h\|_{(|t|+|l|)} .$$

Finally, if $n=kN+r$, with $0 \leq j \leq N-1$, we write $$ \| \mathcal{L}_{\varphi - i(t \tau + l \theta) }^{kN+r} h \|_{L^2(\mu)} \leq (1-\varepsilon)^k \|\mathcal{L}_\varphi^r h\|_{(|t|+|l|)}  \lesssim (1-\varepsilon)^{n/(2N)} \|h\|_{(|t|+|l|)}  ,$$

and the proof is done.

\end{quote}

\end{appendices}


\newpage



\begin{thebibliography}{9}

\bibitem[ARW20]{ARW20}
    {A. Algom, F. Rodriguez, Z. Wang},
    \emph{Pointwise normality and Fourier decay for self-conformal measures},
    Volume 393, 24 December 2021, 108096
    \href{https://doi.org/10.1016/j.aim.2021.108096}{doi:10.1016/j.aim.2021.108096}, \href{https://arxiv.org/abs/2012.06529v3}{	arXiv:2012.06529}
    

\bibitem[BD17]{BD17}
     {J. Bourgain, S. Dyatlov}, 
     \emph{Fourier dimension and spectral gaps for hyperbolic surfaces},
     Geom. Funct. Anal. 27, 744–771 (2017),  \href{https://arxiv.org/abs/1704.02909}{arXiv:1704.02909} 

\bibitem[Be91]{Be91}
    {A.F. Beardon},
    \emph{Iteration of Rational Functions: Complex Analytic Dynamical Systems},
    Graduate Texts in Mathematics, 1991, Springer New York
    
  
\bibitem[Bl96]{Bl96}
    {C. Bluhm},
    \emph{Random recursive construction of Salem sets}
    Ark. Mat. 34 (1996), 51-63.

\bibitem[Bo10]{Bo10}
    {J. Bourgain},
    \emph{The discretized sum-product and projection theorems} 
    JAMA 112, 193-236 (2010)  \href{https://doi.org/10.1007/s11854-010-0028-x}{doi:10.1007/s11854-010-0028-x}

\bibitem[Br19]{Br19}
    {J. Brémont},
    \emph{Self-similar measures and the Rajchman property}
    preprint, \href{https://arxiv.org/abs/1910.03463}{arXiv:1910.03463 }

\bibitem[Do98]{Do98}
    {D. Dolgopyat},
    \emph{On decay of correlations in Anosov flows}, \\
    Ann. of Math. (2) 147 (2) (1998) 357–390.


\bibitem [ES11]{ES11}
    {A. Eremenko and S. Van Strien.}
    \emph{Rational maps with real multipliers},
    Trans. AMS, 363, (2011) p. 6453-6463. \href{https://doi.org/10.1090/S0002-9947-2011-05308-0}{doi:10.1090/S0002-9947-2011-05308-0}

\bibitem[Fr35]{Fr35}
    {O. Frostman},
    \emph{Potentiel d’équilibre et capacité des ensembles avec quelques applications à la
    théorie des fonctions.} 
    Meddel. Lunds Univ. Math. Sem. 3 (1935), 1-118

\bibitem[Gr09]{Gr09}
    {B. Green}, 
    \emph{Sum-product phenomena in $\mathbb{F}_p$: a brief introduction} \\
    Notes written from a Cambridge course on Additive Combinatorics, 2009.  \href{https://arxiv.org/pdf/0904.2075.pdf}{arXiv:0904.2075}


\bibitem[Ha17]{Ha17}
    {K. Hambrook},
    \emph{Explicit Salem sets in $\mathbb{ℝ}^2$.}
    Adv. Math. 311 (2017), 634–648.
 

\bibitem[JS16]{JS16}
    {T. Jordan, T. Sahlsten}
    \emph{ Fourier transforms of Gibbs measures for the Gauss map } \\
    Math. Ann., 364(3-4), 983-1023, 2016. 
    \href{ https://arxiv.org/abs/1312.3619 }{arXiv:1312.3619}



\bibitem[Ka66]{Ka66}
    {J. P. Kahane},
    \emph{Images d’ensembles parfaits par des séries de Fourier Gaussiennes}
    Acad. Sci. Paris 263 (1966), 678-681.


\bibitem[Kau81]{Kau81}
    {R. Kaufman},
    \emph{On the theorem of Jarnik and Besicovitch},
    Acta Arith. 39 (1981), 265-267.

\bibitem[Li17]{Li17}
      {J. Li}, 
     \emph{Decrease of Fourier coefficients of stationary measures},
     Math. Ann. 372, 1189–1238 (2018).  \href{https://doi.org/10.1007/s00208-018-1743-3}{doi:10.1007/s00208-018-1743-3}, \href{https://arxiv.org/abs/1706.07184v3}{	arXiv:1706.07184 }.

\bibitem[Li18]{Li18}
     {J. Li}, 
     \emph{Discretized Sum-product and Fourier decay in }$\mathbb{R}^n$,
     JAMA 143, 763–800 (2021). \newline \href{https://doi.org/10.1007/s11854-021-0169-0}{doi:10.1007/s11854-021-0169-0},  \href{https://arxiv.org/abs/1811.06852v2}{arXiv:1811.06852} 
     

\bibitem[Li18b]{Li18b}
    {J. Li},
    \emph{Fourier decay, Renewal theorem and Spectral gaps for random walks on split semisimple Lie groups}
    preprint 2018, \href{https://arxiv.org/abs/1811.06852v2}{arXiv:1811.06852}


\bibitem[LNP19]{LNP19}
    {J. Li, F. Naud, W. Pan},
    \emph{Kleinian Schottky groups, Patterson-Sullivan measures and Fourier decay} 
    Duke Math. J. 170 (4) 775 - 825, 15 March 2021. \\ \href{https://doi.org/10.1215/00127094-2020-0058}{doi:10.1215/00127094-2020-0058}, \href{https://arxiv.org/abs/1902.01103}{	arXiv:1902.01103}.  

\bibitem[LS19]{LS19}
    {J. Li, T. Sahlsten}
    \emph{Fourier transform of self-affine measures}
    Advances in Mathematics, Volume 374, 18 November 2020, 107349 \href{
https://doi.org/10.1016/j.aim.2020.107349
}{doi: 10.1016/j.aim.2020.107349}, \href{https://arxiv.org/abs/1903.09601v2}{	arXiv:1903.09601}

\bibitem[Ma15]{Ma15}
     {P. Mattila},
     \emph{Fourier analysis and Hausdorff dimension},
     Cambridge University Press, 2015.

\bibitem[Mi90]{Mi90}
    {J. Milnor},
    \emph{Dynamics in one complex variable} \\
    Stony Brook IMS Preprint 1990/5, \href{https://arxiv.org/abs/math/9201272}{	arXiv:math/9201272}
    
\bibitem[MR92]{MR92}
    {R. Mane, L.F. Da Rocha},
    \emph{Julia sets are uniformly perfect} \\
    Proc. Amer. Math. Soc. 116 (1992), no. 1, 251–257.

\bibitem[Na05]{Na05}
    {F. Naud},
    \emph{Expanding maps on Cantor sets and analytic continuation of zeta functions}
    Volume 38, Issue 1, 2005, Pages 116-153, ISSN 0012-9593, \href{https://doi.org/10.1016/j.ansens.2004.11.002}{doi:10.1016/j.ansens.2004.11.002}.    

\bibitem[OW17]{OW17}
    {H. Oh, D. Winter},
    \emph{Prime number theorem and holonomies for hyperbolic rational maps} \\
    Inventiones mathematicae, 2017.  208. \href{https://doi.org/10.1007/s00222-016-0693-1}{doi:10.1007/s00222-016-0693-1.} 
    \href{ https://arxiv.org/abs/1603.00107}{arXiv:1603.00107}
    
\bibitem[Pe98]{Pe98}
    {Y. Pesin}
    \emph{Dimension theory in dynamical systems: contemporary views and applications}
    Chicago Lectures in Mathematics, University of Chicago Press, ISBN 0 226 66222 5.     
    
\bibitem[PP90]{PP90}
    {W. Parry, M. Pollicott},
    \emph{Zeta Functions and the Periodic Orbit Structure of Hyperbolic Dynamics}
    Asterisque; 187-188, société mathématique de France, 1990.

\bibitem[PS16]{PS16}
    {V. Petkov, L. Stoyanov},
    \emph{Ruelle transfer operators with two complex parameters and applications}
    Discrete and Continuous Dynamical Systems, 2016, 36 (11) : 6413-6451. \\ \href{https://www.aimsciences.org/article/doi/10.3934/dcds.2016077}{doi: 10.3934/dcds.2016077}

\bibitem[PU09]{PU09}
    {F. Przytycki, M. Urbański},
    \emph{Conformal fractals: ergodic theory methods} \\
    Cambridge university, 2009.

\bibitem[PU17]{PU17}
    {M. Pollicott and M. Urbański},
    \emph{Asymptotic Counting in Conformal Dynamical Systems.}
    Memoirs of the American Mathematical Society. 271. \href{https://www.ams.org/books/memo/1327/}{doi:10.1090/memo/1327}. (2017).

\bibitem[PW97]{PW97}
    {Y. Pesin and H. Weiss},
    \emph{A multifractal analysis of equilibrium measures for conformal expanding maps and Moranlike geometric constructions}, 
    Journal of Statistical Physics 86, 233—275, (1997)

\bibitem[QR03]{QR03}
    {M. Queffélec and O. Ramaré},
    \emph{Analyse de Fourier des fractions continues à quotients restreints},     Enseign. Math. (2), 49(3-4):335–356, 2003

\bibitem[Ra21]{Ra21}
    {A. Rapaport}
    \emph{On the Rajchman property for self-similar measures on $\mathbb{R}^d$}
    Advances in Mathematics Volume 403, 2022, 108375 
    \href{https://doi.org/10.1016/j.aim.2022.108375}{doi:10.1016/j.aim.2022.108375}, \href{https://arxiv.org/abs/2104.03955}{	arXiv:2104.03955}

\bibitem[Ru78]{Ru78}
    {D. Ruelle}, 
    \emph{Thermodynamic Formalism}, 
    Second edition, Cambridge University Press 2004 

\bibitem[Ru89]{Ru89}
    {D. Ruelle},
    \emph{The thermodynamic formalism for expanding maps}, \\
    Commun. Math. Phys. 125, 239-262 (1989) 
  

\bibitem[Sa51]{Sa51}
    {R. Salem}, \emph{On singular monotonic functions whose spectrum has a given Hausdorff dimension.} 
Ark. Mat., 1:353–365, 1951.


\bibitem[ShSt20]{ShSt20}
    {R. Sharp, A. Stylianou},
    \emph{Statistics of multipliers for hyperbolic rational maps}, \\
    Preprint, 2020. \href{https://arxiv.org/abs/2010.15646}{arXiv:2010.15646}
    
\bibitem[So19]{So19}
    {B. Solomyak},
    \emph{Fourier decay for self-similar measures},
    Proc. Amer. Math. Soc. 149 (2021), 3277-3291
    \href{https://doi.org/10.1090/proc/15515}{doi:10.1090/proc/15515}, \href{https://arxiv.org/abs/1906.12164v2}{	arXiv:1906.12164}
    
 
\bibitem[SS20]{SS20}
    {T. Sahlsten, C. Stevens},
    \emph{Fourier transform and expanding maps on Cantor sets} \\
    preprint, 2020.
    \href{https://arxiv.org/abs/2009.01703v4}{	arXiv:2009.01703 }
    
\bibitem[St11]{St11}
    {L. Stoyanov},
    \emph{Spectra of Ruelle transfer operators for Axiom A flows.}
    Nonlinearity, 24(4):1089, 2011. \href{https://iopscience.iop.org/article/10.1088/0951-7715/24/4/005}{doi:10.1088/0951-7715/24/4/005}
 
 \bibitem[VY20]{VY20}
    {P. P. Varjú, H. Yu}
    \emph{Fourier decay of self-similar measures and self-similar sets of uniqueness} Anal. PDE 15 (3) 843 - 858, 2022. \href{ https://doi.org/10.2140/apde.2022.15.843}{doi:10.2140/apde.2022.15.843}, \href{https://arxiv.org/abs/2004.09358}{arXiv:2004.09358}
 
\bibitem[WW17]{WW17}  
    {C.P. Walkden and T. Withers}
    \emph{The stability index for dynamically defined Weierstrass functions}
    preprint, \href{https://arxiv.org/abs/1709.02451}{arXiv:1709.02451}
    

    

\end{thebibliography}
\end{document}